\documentclass[reqno]{amsart}

\usepackage[margin= 1.0 in]{geometry}
\usepackage{amsfonts}
\usepackage{amsthm}
\usepackage{amsmath}
\usepackage{mathrsfs}
\usepackage{graphicx}
\usepackage{amssymb}
\usepackage{xtab}
\usepackage{tikz-cd} 
\usepackage{xfrac}
\usepackage{stackrel}
\usepackage{color}
\tikzcdset{scale cd/.style={every label/.append style={scale=#1},
    cells={nodes={scale=#1}}}}
\DeclareMathOperator{\coker}{coker}

\newtheorem{theorem}{Theorem}[section]
\newtheorem{corollary}[theorem]{Corollary}
\newtheorem{proposition}[theorem]{Proposition}
\newtheorem{lemma}[theorem]{Lemma}

\theoremstyle{definition}
\newtheorem{remark}[theorem]{Remark}

\newtheorem{definition}[theorem]{Definition}

\newcommand{\Z}{\mathbb{Z}}
\newcommand{\Tor}{\operatorname{Tor}}
\newcommand{\Spin}{\operatorname{Spin}}
\DeclareMathOperator{\Ima}{Im}
\DeclareMathOperator{\Cl}{Cl}

\usepackage{hyperref}
\hypersetup{
    colorlinks=true,
    citecolor=blue,
    linkcolor=violet
    }

\begin{document}

\title{On the Homological Stability of Orthogonal and Spin Groups}

\begin{abstract}
We improve homological stability ranges for the orthogonal group, special orthogonal group, elementary orthogonal group and the spin group over a commutative local ring $R$ with infinite residue field such that $2 \in R^{*}$.
\end{abstract}

\author{Marco Schlichting and Sunny Sood}

\address{Mathematics Institute \\
University of Warwick \\
CV4 7AL \\
United Kingdom}

\email{M.Schlichting@warwick.ac.uk}

\address{Mathematics Institute \\
University of Warwick \\
CV4 7AL \\
United Kingdom}

\email{Sunny.Sood.1@warwick.ac.uk}


\maketitle

\section{Introduction}
In this paper, motivated by Hermitian K-Theory, we improve homological stability results for the orthogonal group, special orthogonal group, elementary orthogonal group and the spin group with respect to the hyperbolic form. 
In the orthogonal case, this improves the range for homological stability given by Mirzaii \cite{mirzaii2004homology} by 1 and generalises the result obtained by Sprehn and Wahl \cite{sprehn2020homological} to the case of local rings. In the special orthogonal case, this generalises the result obtained by Essert \cite{essert2013homological} for infinite fields to the case of local rings, and is the first homological stability result for the special orthogonal group over a local ring. For the elementary orthogonal group, this improves the range for homological stability given by Randal-Williams and Wahl \cite{randal2017homological} by a factor of 3. For the spin group, this coincides with $H_{1}$-stability and $H_{2}$-stability results stated in Hahn-O'Meara \cite{hahn1989classical}, and is the first homological stability result that accounts for all homology groups. 
\vspace{1ex} 

Recall, for a ring $R$, the \textit{(split) orthogonal group} $O_{n,n}(R) \subseteq GL_{2n}(R)$, is the subgroup 
\[
O_{n,n}(R) := \{ A \in GL_{2n}(R) |  \,\, ^{t}A\psi_{2n}A = \psi_{2n} \}
\]
of $R$-linear automorphisms preserving the form 
\[
\psi_{2n} = 
\begin{pmatrix}
\psi_{2} \\
 &\psi_{2} \\
 && \ddots \\
 &&&\psi_{2}
\end{pmatrix} = \bigoplus_{1}^{n}\psi_{2}, \quad \psi_{2} = 
\begin{pmatrix}
0 & 1 \\
1 & 0
\end{pmatrix},
\]
where $^{t}A$ denotes the transpose matrix of $A$. Define $SO_{n,n}(R)$ to be the subgroup of $O_{n,n}(R)$ consisting of all matrices with determinant 1. We will always consider $O_{n,n}(R)$ as a subgroup of $O_{n+1, n+1}(R)$ via the embedding 
\[
O_{n,n}(R) \subseteq O_{n+1, n+1}(R) : A \mapsto 
\begin{pmatrix}
1 & 0 & 0 \\
0 & 1 & 0 \\
0 & 0 & A
\end{pmatrix}.
\]

Consider $ O_{\infty,\infty}(R) := \varinjlim O_{n,n}(R)$. For a commutative ring $R$ with $2 \in R^{*}$, the higher Hermitian K-Theory groups $GW_{i}(R)$  may be modeled as the homotopy groups of the plus construction applied to the classifying space $BO_{\infty,\infty }(R)$:
\[
GW_{i}(R) \cong \pi_{i}(BO_{\infty,\infty}(R)^{+}) \quad \text{for all $i >0$}.
\] 
Here, the plus construction is taken with respect to the maximal perfect subgroup of $O_{\infty,\infty}(R)$, which in this case is equal to the commutator subgroup $[O_{\infty,\infty}(R), O_{\infty,\infty}(R)]$. See for example \cite{schlichting2017hermitian}. Therefore, we have a Hurewicz map from $GW_{i}(R)$ into the homology group $H_{i}( O_{\infty,\infty}(R),\mathbb{Z})$, which is why we are interested in studying this homological stability problem.  

The homology of the orthogonal group $O_{n,n}$ has long been known to stabilise, in quite large generality; see, e.g., \cite{vogtmann1981spherical}, \cite{betley1990homological}, \cite{charney1987generalization}. 
Recently, Sprehn and Wahl in \cite{sprehn2020homological} have shown that for every field $\mathbb{F}$ other than the field $\mathbb{F}_{2}$, $H_{k}(O_{n,n}(\mathbb{F}) , \mathbb{Z}) \rightarrow H_{k}(O_{n+1,n+1}(\mathbb{F}), \mathbb{Z})$ is an isomorphism for $k \leq n-1$ and surjective for $k \leq n$. In the context of fields, this is currently the best known range of stability. However, they were unable to extend their results to local rings, essentially because the framework that they use is only applicable to \textit{vector spaces}, rather than modules over local rings. 
In the context of local rings, the first precise range of stability was given by Mirzaii in \cite{mirzaii2004homology}. Specfically, he proved that for $R$ commutative local ring with infinite residue field, $H_{k}(O_{n,n}(R) , \mathbb{Z}) \rightarrow H_{k}(O_{n+1,n+1}(R), \mathbb{Z})$ is an isomorphism for $k \leq n-2$ and surjective for $k \leq n-1$. 

Our first main result is an improvement on the known stability range for $O_{n,n}$ over local rings with infinite residue field, with the additional assumption that we require 2 to be invertible. Specifically, we prove that:
\begin{theorem}
Let $R$ be a commutative local ring with infinite residue field such that $2 \in R^{*}$. Then, the natural homomorphism 
\[
H_{k}(O_{n,n}(R),\mathbb{Z}) \longrightarrow H_{k}(O_{n+1,n+1}(R), \mathbb{Z})
\]
is an isomorphism for $k \leq n-1$ and surjective for $k \leq n$. 
\end{theorem}
The proof is modelled on the homological stability proofs given in \cite{nesterenko1990homology} and \cite{schlichting2017euler}. Specifically, we consider a highly acyclic chain complex on which $O_{n,n}$ acts, and analyse the resulting hyperhomology spectral sequences. This is a standard method of proving such results, but the main innovation that gives us the improvement in stability is the use of the technique of \textit{localising homology groups}. 
This technique was first introduced in \cite{schlichting2017euler}. It is this technique that makes it possible to analyse the hyperhomology spectral sequences. 

In addition, the methods we use to prove homological stability for $O_{n,n}(R)$ may be used to prove homological stability for $SO_{n,n}(R)$, which gives our second main result:

\begin{theorem}
Let $R$ be a commutative local ring with infinite residue field such that $2 \in R^{*}$. Then, the natural homomorphism 
\[
H_{k}(SO_{n,n}(R),\mathbb{Z}) \longrightarrow H_{k}(SO_{n+1,n+1}(R),\mathbb{Z})
\]
is an isomorphism for $k \leq n-1$ and surjective for $k \leq n$. 
\end{theorem}
This is the first homological stability result for the special orthogonal group over a local ring, and generalises the analogous result for infinite fields obtained by Essert \cite{essert2013homological}.

Next, the elementary orthogonal group $EO_{n,n}(R)$ may be defined in terms of generators and should be viewed as the orthogonal analogue of the elementary linear group $E_{n}(R)$. 

For $r \in R$ and $1 \leq k \neq l \leq n $, define $\gamma_{kl}(r)$ to be the $n \times n$ matrix with $r$ in the $(k,l)$ position, $-r$ in the $(l,k)$ position, and 0 elsewhere. Define $\gamma_{kk}(r)$ to be the zero matrix. In addition, for $1 \leq i \neq j \leq n$, define $e_{ij}(r)$ to be the $n \times n$ elementary linear matrix with $1$ along the diagonal and $r$ in the $(i,j)$ position.  
We then define the family of \textit{elementary orthogonal matrices} as
\begin{align}
E_{2k, 2l}(r) &:= 
\begin{pmatrix}
I_{n} & \\
\gamma_{kl}(r) & I_{n}
\end{pmatrix},    \label{matrix:$E_{2k,2l}$} \\               
E_{2k-1, 2l-1}(r) &:= 
\begin{pmatrix}
I_{n} & \gamma_{kl}(r)   \label{matrix:$E_{2k-1,2l-1}$}\\
& I_{n}
\end{pmatrix},
\end{align}
and for $k \neq l$, 
\begin{align}
E_{2k-1, 2l}(r) &:= 
\begin{pmatrix}
e_{kl}(r) & \\
 & e_{lk}(-r)
\end{pmatrix},   \label{matrix:$E_{2k-1,2l}$} \\
E_{2k, 2l-1}(r) &:= 
\begin{pmatrix}
e_{lk}(-r) &   \\
& e_{kl}(r)
\end{pmatrix}.   \label{matrix:$E_{2k,2l-1}$}
\end{align}
We define the \textit{elementary orthogonal group} $EO_{n,n}(R)$ as the subgroup of $O_{n,n}(R)$ generated by the elementary orthogonal matrices. We refer the reader to \cite[Sections 5.3A and 5.3B]{hahn1989classical} for more information about $EO_{n,n}(R)$, including a list of relations amongst these generators. 
\begin{remark}
For the sake of notation, we have in the above definitions used the convention that the hyperbolic form on $R^{2n}$ is taken with respect to matrix$
\begin{pmatrix}
0 & I_{n} \\
I_{n} & 0
\end{pmatrix}$. 
This convention therefore differs from the standard convention used in this paper up to conjugation by a suitable permutation matrix, and we will always tacitly assume this whenever working with $EO_{n,n}(R)$.
\end{remark}
Our third main result:
\begin{theorem}
Let $R$ be a commutative local ring with infinite residue field such that $2 \in R^{*}$. Then, the natural homomorphism
\[
H_{k}(EO_{n,n}(R), \mathbb{Z}) \longrightarrow H_{k}(EO_{n+1,n+1}(R), \mathbb{Z})
\]
is an isomorphism for $k \leq n-1$ and surjective for $k \leq n$. 
\end{theorem}
This improves the range for homological stability given by Randal-Williams and Wahl in \cite{randal2017homological} by a factor of 3. 

Finally, let $R$ be a commutative ring, which for the purposes of this article is such that $2 \in R^{*}$. We define $\Spin_{n,n}(R)$ to be the Spin group of the quadratic module $(R^{2n}, \langle \cdot , \cdot \rangle)$, where $\langle \cdot, \cdot \rangle$ is the symmetric bilinear form associated to the matrix $\psi_{2n}$ as above. We refer the reader to the appendix for more information about Spin groups. The reader may also want to look at \cite{hahn1989classical}, \cite{scharlau2012quadratic} and \cite{lawson2016spin} as alternative references.

In the case $R$ is a commutative local ring with infinite residue field such that $2 \in R^{*}$, homological stability for $\Spin_{n,n}(R)$ will follow immediately from homological stability of $EO_{n,n}(R)$ via the relative Hochschild-Serre Spectral Sequence applied to short exact sequence 
\[
1 \longrightarrow \mathbb{Z}_{2} \longrightarrow \Spin_{n,n}(R) \longrightarrow EO_{n,n}(R) \longrightarrow 1,
\]
see Theorem \ref{theorem:shortexactsequences} in the appendix. Indeed, for the purposes of this paper, it is perhaps best to think of $EO_{n,n}(R)$ as being \textit{defined} in terms of this short exact sequence. This is the perspective that we will adopt.  

This gives us our fourth main result:
\begin{theorem}
Let $R$ be a commutative local ring with infinite residue field such that $2 \in R^{*}$. Then, the natural homomorphism
\[
H_{k}(\Spin_{n,n}(R),\mathbb{Z}) \longrightarrow H_{k}(\Spin_{n+1,n+1}(R),\mathbb{Z})
\]
is an isomorphism for $k \leq n-1$ and surjective for $k \leq n$. 
\end{theorem}
This coincides with known $H_{1}$ and $H_{2}$-stability results for $\Spin_{n,n}$ given in \cite{hahn1989classical}, and is the first such homological stability result that accounts for all homology groups.

\subsection{Acknowledgements}
The authors thank the anonymous referee for providing very detailed feedback on this paper. The second author also gratefully acknowledges funding from the University of Warwick and the UK Engineering and Physical Sciences Research Council (Grant number: EP/V520226/1) during his time as a PhD student, when this paper was first written. 

\section{The complex of totally isotropic unimodular sequences} \label{section:complex}
In this section, $R$ will be a commutative local ring with infinite residue field.
\subsection{The chain complex}
To define the chain complex we want to consider, we need to make some preliminary definitions. 

\begin{definition}
A {\em space} over a ring $R$ is a projective $R$-module of finite rank.  A submodule $M \subset V$ of a space $V$ is called a subspace if it is a direct factor. 
\end{definition}

\begin{definition}
Let $q\geq 0$ be an integer, and $W$ a free $R$-module of rank $n$.
A sequence of $q$ vectors $(v_{1}, \dots, v_{q})$ in $W$ will be called {\em unimodular} if every subsequence of length $r \leq \min\{n,q\}$ generates a subspace of rank $r$.
We denote by $\mathcal{U}_q(W)$ the set of unimodular sequences of length $q$ in $W$.
\end{definition}

\begin{remark}
For $R$ a local ring, the sequence of vectors $(v_{1}, \dots, v_{q}) $ in $R^{2n}$ is unimodular if and only if $(\bar{v}_{1}, \dots, \bar{v}_{q})$ in $k^{2n}$ is unimodular, where $k$ denotes the residue field of $R$ and $\bar{v}_{i}$ the class of $v_{i}$ in $k^{2n}.$ 
\end{remark}

\begin{definition}
A sequence of vectors $(v_{1}, \dots, v_{q})$ in $R^{2n}$ will be called {\em totally isotropic} if for every $i,j=1,\dots,q$ we have $\langle v_i,v_j\rangle = 0$.
\end{definition}

We now introduce the chain complex that we want to consider.
Specifically, consider chain complex 
\begin{align}
C_{*}(n) := (C_{*}(R^{2n}), d) = \cdots \rightarrow C_{2}(R^{2n}) \rightarrow C_{1}(R^{2n}) \xrightarrow{\varepsilon} \mathbb{Z} \rightarrow 0  \label{expression:chaincomplex}
\end{align}
where for $k \geq 1$, $C_{k}(R^{2n})$ is defined as the free abelian group $C_{k}(R^{2n}) := \mathbb{Z}[\mathcal{IU}_{k}(R^{2n})]$ generated by the set of unimodular totally isotropic sequences of length $k$ in $R^{2n}$:
\[
\mathcal{IU}_{k}(R^{2n}) := \{ (v_{1}, \dots, v_{k}) : v_{i} \in R^{2n}, (v_{1}, \dots, v_{k}) \,\, \text{totally isotropic and unimodular}\}.
\]
We set $C_{0}(n) := \mathbb{Z}$.

The differential $d$ is defined on basis elements by 
\[
d(v_{1}, \dots, v_{k}) := \sum_{i=1}^{k}(-1)^{i+1}d_{i}(v_{1},\dots, v_{k} ),
\]
\[
d_{i}(v_{1},\dots, v_{k} ) := (v_{1}, \dots , \hat{v}_{i}, \dots, v_{k}),
\]
and the map $\varepsilon: C_{1}(R^{2n}) \rightarrow \mathbb{Z}$ is the augmentation map sending a generator $(v)$ to $1$.
\begin{remark}
The simplicial set that gives rise to chain complex (\ref{expression:chaincomplex}) has already been studied in \cite{panin1990homological} and \cite{mirzaii2004homology}. 
As we do not need to consider simplicial sets in this article, we stick to chain complex notation. 
\end{remark}

Note that for $A \in O_{n,n}(R)$, $A$ acts from the left on the chain complex $(C_{*}(R^{2n}), d)$ by acting on basis elements: 
\[
A \cdot(v_{1}, \dots, v_{k}) := (Av_{1}, \dots, Av_{k}).
\] 
For a resolution $P_*$ of the trivial $O_{n,n}(R)$-module $\mathbb{Z}$ by projective right $O_{n,n}(R)$-modules, the bicomplex $P_*\otimes_{O_{n,n}}C_*(n)$ gives rise to two \textit{hyperhomology spectral sequences} 
\begin{align}
E^{2}_{p,q}(n) &= H_{p}(O_{n,n}, H_{q}(C_{*}(n))) \Rightarrow H_{p+q}(O_{n,n}, C_{*}(n)) \label{ss:firsthyperhomology} \\
E^{1}_{p,q}(n) &= H_{q}(O_{n,n}, C_{p}(n)) \Rightarrow H_{p+q}(O_{n,n}, C_{*}(n)).  
\label{ss:secondhyperhomology} 
\end{align}
For a reference on hyperhomology spectral sequences, we refer the reader to \cite[Chapter VII, Section 5]{brown1982cohomology}.
Replacing $O_{n,n}$ with $SO_{n,n}$ and $EO_{n,n}$, we similarly obtain hyperhomology spectral sequences 
\begin{align*}
E^{2}_{p,q}(n) &= H_{p}(SO_{n,n}, H_{q}(C_{*}(n))) \Rightarrow H_{p+q}(SO_{n,n}, C_{*}(n))  \\
E^{1}_{p,q}(n) &= H_{q}(SO_{n,n}, C_{p}(n)) \Rightarrow H_{p+q}(SO_{n,n}, C_{*}(n))
\end{align*}
and 
\begin{align*}
E^{2}_{p,q}(n) &= H_{p}(EO_{n,n}, H_{q}(C_{*}(n))) \Rightarrow H_{p+q}(EO_{n,n}, C_{*}(n))  \\
E^{1}_{p,q}(n) &= H_{q}(EO_{n,n}, C_{p}(n)) \Rightarrow H_{p+q}(EO_{n,n}, C_{*}(n)).\label{ss:elementarysecondhyperhomology}
\end{align*}
These spectral sequences will eventually give us our desired homological stability results. 

\subsection{Proving acyclicity}

We would like to prove that the complex $(C_{*}(n), d)$ is $(n-1)$-acyclic. This has already been proven by Mirzaii \cite{mirzaii2004homology}, but our proof has the advantage that it does not refer to the simplicial techniques used in \cite{kallen1980homology}.  
However, we do make use of a concept \textit{general position}. This was first defined in \cite{panin1990homological}, and used in both \cite{panin1990homological} and \cite{mirzaii2004homology} to prove their respective acyclicity  results. 
We give the definition as stated in \cite{mirzaii2004homology}.

\begin{definition}
\label{dfn:GenPosition}
Let $S = \{v_{1}, \dots, v_{k}\}$ and $T = \{w_{1}, \dots, w_{k'}\}$ be basis of two totally isotropic free summands of $R^{2n}$. We say that $T$ is in general position with $S$, if $k \leq k'$ and the $k' \times k$- matrix $(\langle w_{i}, v_{j}\rangle)$ has a left inverse. 
\end{definition}

We may also say that a totally isotropic subspace $W$ is in general position with respect to a totally isotropic subspace $V$ if there is a basis $T$ of $W$ which is in general position with respect to a basis $S$ of $V$ as in Definition \ref{dfn:GenPosition}. 
The following result, whose proof we refer to \cite[Chapter 2, Proposition 4.2]{mirzaii2004homology}, will be used to prove acyclicity. 

\begin{proposition} \label{prop:gp}
Let $n \geq 2$ be an integer and assume $T_{i}$, $i=1,\dots,\ell$ are finitely many subsets of $R^{2n}$ such that each $T_{i}$ is a basis of a free totally isotropic summand of $R^{2n}$ with $k$ elements, where $k \leq n-1$. 
Then, there is a basis, $T = \{w_{1}, \dots, w_{n}\}$, of a free totally isotropic summand of $R^{2n}$ such that $T$ is in general position with all $T_{i}$, $i=1,\dots,\ell$. 
Moreover, dim$(W \cap V_{i}^{\perp}) = n-k$, where $W =$Span$(T)$ and $V_{i} =$Span$(T_{i})$, $i=1,\dots,\ell$.
\end{proposition} 
\qed

In addition, the following lemma will be both useful and reassuring. 

\begin{lemma} \label{lemma:gpzerointersection}
Let $W$ and $V$ be totally isotropic subspaces of $R^{2n}$.
Assume that $W$ is in general position with respect to $V$. Then $W \cap V = \{0\}$.
\end{lemma}

\begin{remark}
Lemma \ref{lemma:gpzerointersection} implies that if $W$ is in general position with respect to a unimodular sequence $(u_{1}, \dots, u_{k})$ for $k < n$, then $(u_{1}, \dots, u_{k}, w)$ is unimodular for every unimodular vector $w \in W$. 
\end{remark}

\begin{proof}[Proof of Lemma \ref{lemma:gpzerointersection}]
As $W$ is in general position with respect to $V$, the map 
\begin{align*}
\pi: & W \twoheadrightarrow R^{k} \\
& w \mapsto \left( \langle w, v_{1} \rangle, \dots, \langle w, v_{k} \rangle \right)
\end{align*}
is surjective. 
Therefore, for every $1\leq i \leq k$, there exists a $v_{i}^{\#} \in W$ such that $\pi(v_{i}^{\#}) = (0,\dots,0,1,0,\dots,0)$, the 1 being in the $i$th position.
Now, let $y \in W\cap V$. 
Note, as $y \in V$ and $V$ is free with basis $v_{1}, \dots, v_{k}$, we may write $y$ uniquely as $y = \sum_{i} a_{i}v_{i}$ for some $a_{i} \in R$. 
Evaluating $\langle v_{i}^{\#}, \cdot \rangle$ on $y$ and noting that $\langle v_{i}^{\#}, v_{j}\rangle = \delta_{ij}$, we deduce that $a_{i} = \langle v_{i}^{\#}, y \rangle$ for every $i=1,\dots,k$. 
But $v_{i}^{\#}, y \in W$ and $W$ is totally isotropic, so $\langle v_{i}^{\#}, y \rangle = 0$ for $i=1,\dots,k$. Therefore, $y = 0$.  
\end{proof}

For $u=\sum_im_iu_i \in \Z[\mathcal{IU}_{p}(R^{2n})]$ and $v=\sum_jn_jv_j \in \Z[\mathcal{IU}_{q}(R^{2n})]$ such that $(u_i,v_j) \in \mathcal{IU}_{p+q}(R^{2n})$ for all $i,j$, we will write $(u,v)$ for the element
$$(u,v) = \sum_{i,j}m_in_j(u_i,v_j) \in \mathbb{Z}[\mathcal{IU}_{p+q}(R^{2n})].$$
Using Proposition \ref{prop:gp} and Lemma \ref{lemma:gpzerointersection}, we prove the following.

\begin{lemma} \label{lemma:acyclic}
Let $p, q \geq 0$ and $p+q < n$. Let $(u,f) \in \mathbb{Z}[\mathcal{IU}_{p+q}(R^{2n})]$ such that $u \in \mathcal{IU}_{p}$ and $f \in \mathbb{Z}[\mathcal{U}_{q}(W)]$, where $W$ is a free totally isotropic summand of $R^{2n}$ of dimension $n$ in general position with respect to $U= $Span$(u)$.
If $df = 0 \in \mathbb{Z}[\mathcal{U}_{q-1}(W)]$, then there exists an element $g \in \mathbb{Z}[\mathcal{U}_{q+1}(W)]$ such that $dg = f$ and $(u,g) \in \mathbb{Z}[\mathcal{IU}_{p+q+1}(R^{2n})]$.
\end{lemma}

\begin{proof}
Since $f \in \mathbb{Z}[\mathcal{U}_{q}(W)]$ and $(u,f) \in \mathbb{Z}[\mathcal{IU}_{p+q}(R^{2n})]$, we have $f \in \mathbb{Z}[\mathcal{U}_{q}(L)]$, where $L = W \cap U^{\perp}$. 
As $W$ is in general position with respect to $U$, $L$ is a finitely generated free $R$-module of rank $n-p$. 
Write $f = \sum_{i}n_{i}(v_{1}^{i}, \dots, v_{q}^{i})$. 
Then, as $R$ is a local ring with infinite residue field and $L$ is a finitely generated free $R$-module of rank $n-p > q$, we deduce that there exists a $v \in L$, such that $(v, v_{1}^{i}, \dots, v_{q}^{i}) \in \mathcal{U}_{q}(L)$ for every $i$. This is standard; see for instance \cite[Lemmas 5.5 and 5.6]{schlichting2017euler}.
Let $g := \sum_{i}n_{i}(v, v_{1}^{i}, \dots, v_{q}^{i})$. Then $dg = f$ by construction. Moreover as $g \in \mathbb{Z}[\mathcal{U}_{q+1}(L)]$, $(u,g)$ defines a totally isotropic sequence of vectors and as $g \in \mathbb{Z}[\mathcal{U}_{q+1}(W)]$, by Lemma \ref{lemma:gpzerointersection}, $(u,g)$ is a unimodular sequence, so that $(u,g) \in \mathbb{Z}[\mathcal{IU}_{p+q+1}(R^{2n})]$.
\end{proof}

\begin{corollary} \label{corollary:acyclic}
Let $k \leq n-1$, and let $z \in C_{k}(n) = \mathbb{Z}[\mathcal{IU}_{k}(R^{2n})]$ be a cycle. 
Then, $z$ is homologous to a cycle $z'\in \mathbb{Z}[\mathcal{U}_{k}(W)]\subset C_{k}(n)$ contained within a free totally isotropic summand $W$ of $R^{2n}$ of dimension $n$. 
\end{corollary}

\begin{proof}
Suppose $z = \sum_{i}n_{i}u_{i}$ where $n_{i} \in \mathbb{Z}$ and $u_{i} \in \mathcal{IU}_{k}(R^{2n})$. By Proposition \ref{prop:gp}, there exists a free totally isotropic subspace $W$ of rank $n$ in general position with respect to all $U_{i} =$Span$(u_{i})$. 
Choose unimodular vectors $f_{i} \in W \cap U_{i}^{\perp}$ which is possible since $\dim W \cap U_{i}^{\perp} \geq 1$, by Proposition \ref{prop:gp}. 
Note, by Lemma \ref{lemma:gpzerointersection}, $(u_{i}, f_{i}) \in \mathcal{IU}_{k+1}(R^{2n})$ for every $i$. 
Consider the chain $\xi := \sum_{i}n_{i}(u_{i}, f_{i}) \in C_{k+1}(n)$. Note that 
$$d\xi = \sum_{i}n_{i}(du_{i}, f_{i}) + (-1)^{k+1}\sum_{i}n_{i}u_{i} = \sum_{i}n_{i}(du_{i}, f_{i}) + (-1)^{k+1}z, $$ 
so that $z_{1} := (-1)^{k+1}\sum_{i}n_{i}(du_{i}, f_{i})$ is homologous to $z$, which we write as $z_{1} \sim z$.
Now, recursively assume that $z_{q} \in C_{k}(n)$ is cycle such that $z_{q} \sim \sum_{i}(u_{i}, f_{i})$, where $u_{i} \in \mathcal{IU}_{p}(R^{2n})$;  $f_{i} \in \mathbb{Z}[\mathcal{U}_{q}(W)]$,  $W$ is a free totally isotropic summand of $R^{2n}$ of dimension $n$ in general position with respect to all $u_{i}$, $p,q \geq 0$ such that $p+q = k < n$.
Then we collect terms so that $u_{i} \neq u_{j}$ for every $i \neq j$. 
By assumption, we have 
\begin{align*}
0 = dz_{q} &= \sum_{i}d(u_{i}, f_{i}) 
= \sum_{i}\left[ (du_{i}, f_{i}) + (-1)^{p+1}(u_{i}, df_{i}) \right].
\end{align*}
As $u_{i} \neq u_{j}$ and $W$ is in general position with every $u_i$, hence, no column vector of $u_i$ is in $W$, we deduce $df_{i} = 0$ for every $i$.
Therefore, by Lemma \ref{lemma:acyclic}, for every $i$, there exists $g_{i} \in \mathbb{Z}[\mathcal{U}_{q+1}(W)]$ such that $dg_{i} = f_{i}$ and $(u_{i}, g_{i}) \in \mathbb{Z}[\mathcal{IU}_{k+1}(R^{2n})]$.
Note that 
\begin{align*}
d(u_{i}, g_{i}) &= (du_{i}, g_{i}) + (-1)^{p+1}(u_{i}, dg_{i}) 
= (du_{i}, g_{i}) + (-1)^{p+1}(u_{i}, f_{i}). 
\end{align*}
We deduce $z_{q} \sim \sum_{i}(u_{i}, f_{i}) \sim (-1)^p \sum_{i}(du_{i}, g_{i}) =z_{q+1}$.
The corollary is the case $q=k$, $p=0$ setting $z'=z_k$.
\end{proof}

\begin{theorem}
\label{thm:C*acyclic}
The complex $(C_{*}(n), d)$ is $(n-1)-acyclic$, that is, 
$$H_i(C_{*}(n), d)=0\hspace{3ex}\text{ for } i\leq n-1.$$
\end{theorem}

\begin{proof}
Let $k \leq n-1$ and let $z \in C_{k}(n)$ a cycle. 
By Corollary \ref{corollary:acyclic}, $z$ is homologous to a cycle $z'\in \mathbb{Z}[\mathcal{U}_{k}(W)]$
contained within a free totally isotropic summand $W$ of $R^{2n}$ of dimension $n$. 
As $R$ is a local ring with infinite residue field, we deduce that there exists a $\tau \in \mathbb{Z}[\mathcal{U}_{k+1}(W)]\subset C_{k+1}(n)$ such that $d\tau = z'$, by the standard argument recalled in the proof of Lemma \ref{lemma:acyclic}.
In paricular, $z$ is a boundary. 
\end{proof}

\section{Homological stability for $O_{n,n}$} \label{sec:homstability}
From now on, unless stated otherwise, $R$ will be a commutative local ring with infinite residue field and $2 \in R^{*}$. We will also abbreviate the integral homology of a group $G$ as $H_{k}(G)$ whenever it is convenient for us to do so.
\subsection{Transitivity of the group action}
We need to prove that the action of $O_{n,n}$ on $\mathcal{IU}_{p}(R^{2n})$ is \textit{transitive} for all $p \leq n$. It suffices to prove the following lemma.

\begin{lemma} \label{lem:hyperbolicbasis}
Let $p \leq n$ and let $(u_{1}, \dots, u_{p}) \in \mathcal{IU}_{p}(R^{2n})$. Then, $(u_{1},\dots, u_{p})$ may be extended to a hyperbolic basis of $R^{2n}$. 
\end{lemma} 
\begin{proof}
By Witt's Cancellation Theorem, which holds when $R$ is a local ring with 2 invertible (cf. \cite[Chapter I, Theorem 4.4]{milnor1973symmetric}), it will be sufficient to find $u_{1}^{\#}, \dots, u_{p}^{\#}$ such that $(u_{1}, u_{1}^{\#}, \dots, u_{p}, u_{p}^{\#})$ has Gram matrix $\psi_{2p}$. (Note that Span$\{u_{1}, u_{1}^{\#}, \dots, u_{p}, u_{p}^{\#}\}$ is a non-degenerate subspace).

We have that $(u_{1}, \dots, u_{p}) \in \mathcal{IU}_{p}(R^{2n})$, so this sequence is in particular a unimodular sequence of vectors in $R^{2n}$. Thus, the matrix $u=(u_{1}, \dots, u_{p})$ is left invertible. Therefore,  the matrix ${^tu} \psi_{2n}$ is right invertible. This is equivalent to saying that the map
\begin{align*}
T : &R^{2n} \rightarrow R^{p} \\
&x \mapsto \left( \langle u_{1}, x \rangle, \dots, \langle u_{p}, x \rangle  \right)
\end{align*}
is surjective. 
Thus, for $i=1,...,p$, there exists ${u}^{\#}_{i}$ such that $T({u}^{\#}_{i})$ is the $i$-th standard basis vector of $R^p$.
Replacing $u_{i}^{\#}$ with $u_{i}^{\#} - \frac{\langle u_{i}^{\#}, u_{i}^{\#} \rangle}{2}u_{i}$, we conclude the Gram matrix of $(u_{1}, u_{1}^{\#}, \dots, u_{p}, u_{p}^{\#})$ is $\psi_{2p}$.
\end{proof}

\subsection{Analysis of stabilisers} \label{subsection:stabilisers}
\subsubsection{Computation of stabilisers}

Let $G$ be a group acting on a set $S$ from the left. Shapiro's Lemma gives an isomorphism 
\[
\bigoplus_{[x] \in S/G} (i_{x}, x)_{*} : \bigoplus_{[x] \in S/G} H_{*}(G_{x}, \mathbb{Z}) \xrightarrow{\cong} H_{*}(G, \mathbb{Z}[S])
\]
of homology groups, where the direct sum is over a set of representatives $x \in S$ of equivalence classes $[x] \in S/G$; the group $G_{x}$ is the \textit{stabiliser} of $G$ at $x \in S$; the homomorphism $i_{x} : G_{x} \subseteq G$ is the inclusion; and $x$ also denotes the homomorphism of abelian groups $\mathbb{Z} \rightarrow \mathbb{Z}[S]: 1 \mapsto x$. For example, see \cite[Chapter III, Corollary 5.4 and Proposition 6.2]{brown1982cohomology}.

We apply Shapiro's Lemma in the case $G = O_{n,n}(R)$ and $S = \mathcal{IU}_{p}(R^{2n})$. In particular, as the action of $O_{n,n}(R)$ on $\mathcal{IU}_{p}(R^{2n})$ is \textit{transitive} for all $p \leq n$, Shapiro's Lemma gives isomorphisms
\begin{equation}
\label{eqn:ShapiroTkOn}
H_{*}(St(e_{1}, \dots, e_{p})) \xrightarrow{\cong} H_{*}(O_{n,n}, C_{p}(n)), 
\end{equation}
for all $p \leq n$, where $St(e_{1}, \dots, e_{p})$ denotes the stabiliser of $(e_{1}, \dots, e_{p}) \in \mathcal{IU}_{p}(R^{2n})$.
We compute these stabilisers:

\begin{proposition} \label{prop:stabilisers}
Let $1\leq k \leq n$. Then, in the above notation, the stabilisers $A \in St(e_{1}, \dots , e_{k})$ are of the form
\[
A =  
\begin{pmatrix}
1 & c^{1}_{1} & 0 & c^{1}_{2} & \cdots & 0 & c^{1}_{k} & ^{t}u_{1} \\
0 & 1 & 0 & 0 & \cdots & 0 & 0 & 0 \\
0 & c^{2}_{1} & 1 & c^{2}_{2} & \cdots & 0 & c^{2}_{k} & ^{t}u_{2} \\
0 & 0 & 0 & 1 & \cdots & 0 & 0 & 0 \\
\vdots & \vdots & \vdots & \vdots & \ddots & \vdots & \vdots & \vdots \\
0 & c^{k}_{1} & 0 & c^{k}_{2} & \cdots & 1 & c^{k}_{k} & ^{t}u_{k} \\
0 & 0 & 0 & 0 & \cdots & 0 & 1 & 0 \\
0 & x_{1} & 0 & x_{2} & \cdots & 0 & x_{k} & B
\end{pmatrix}
\]
where  $c^{i}_{j} \in R; u_{i}, x_{i} \in R^{2(n-k)}$ and $B \in M_{2(n-k)}(R)$, subject to the conditions
\begin{align}
&u_{i} + \, ^{t}B\psi_{2(n-k)} x_{i} = 0, \label{eqn:stab1} \\                
&c^{i}_{j} + c^{j}_{i} + \langle x_{i}, x_{j} \rangle = 0, \label{eqn:stab2} \\  
&B \in O_{n-k,n-k}. \label{eqn:stab3}           
\end{align}
\end{proposition}

For example, for $k = 1$, we have 
\[
St(e_{1}) = 
\left\{
\begin{pmatrix}
1 & c & ^{t}u \\
0 & 1 & 0 \\
0 & x & B
\end{pmatrix}
\Bigg| \,\, 
u + \, ^{t}B\psi_{2(n-1)} x = 0 ;  \,\,
2c + \langle x, x \rangle = 0 ; \,\,
B \in O_{n-1, n-1}.
\right\}.
\]

\begin{proof}
Let $A \in St (e_{1}, \dots , e_{k})$. Then, $Ae_{i} = e_{i}$ for all $1\leq i \leq k$ by definition, which gives the 1st, 3rd, $\dots$, $(2k-1)$st columns of $A$. 
Moreover, for a fixed $1 \leq i \leq k$ and any $1 \leq j \leq n$, we have 
\begin{align*}
\langle e_{i}, Ae_{j} \rangle = \langle Ae_{i}, Ae_{j} \rangle = \langle e_{i}, e_{j} \rangle = 0
\end{align*}
and 
\begin{align*}
\langle e_{i}, Af_{j} \rangle = \langle Ae_{i}, Af_{j} \rangle = \langle e_{i}, f_{j} \rangle = \delta_{ij}.
\end{align*}
Therefore, as $\langle e_{k}, e_{l} \rangle = 0$ and $\langle e_{k}, f_{l} \rangle = \delta_{kl}$ for all $1\leq k, l \leq n$, we deduce that the coefficient of $f_{i}$ in the expression for $Ae_{j}$ and $Af_{j}$ is $0$ for all $j \neq i$ and the coefficient of $f_{i}$ in the expression for $Af_{i}$ is $1$. This gives the 2nd, 4th, $\dots$, $2k$th rows of $A$.
The remaining coefficients give the $c^{i}_{j} \in R; u_{i}, x_{i} \in R^{2(n-k)}$ and $B \in M_{2(n-k)}(R)$. We use the equation $^{t}A \psi_{2n} A = \psi_{2n}$ to determine the conditions on these variables. 
Specifically, one has that for any $A \in St(e_{1}, \dots, e_{k})$, 
{\tiny
$$\begin{array}{l}
^{t}A \psi A 
 = \begin{pmatrix}
1 & 0 & 0 & 0 & \cdots & 0 & 0 & 0 \\
c^{1}_{1} & 1 & c^{2}_{1} & 0 & \cdots & c^{k}_{1} & 0 & ^{t}x_{1} \\
0 & 0 & 1 & 0 & \cdots & 0 & 0 & 0 \\
c^{1}_{2} & 0 & c^{2}_{2} & 1 & \cdots & c^{k}_{2} & 0 & ^{t}x_{2}  \\
\vdots & \vdots & \vdots & \vdots & \ddots & \vdots & \vdots & \vdots \\
0 & 0 & 0 & 0 & \cdots & 1 & 0 & 0 \\
c^{1}_{k} & 0 & c^{2}_{k} & 0 & \cdots & c^{k}_{k} & 1 & ^{t}x_{k} \\
u_{1} & 0 & u_{2} &0 & \cdots & u_{k} & 0 & ^{t}B
\end{pmatrix}
\psi
\begin{pmatrix}
1 & c^{1}_{1} & 0 & c^{1}_{2} & \cdots & 0 & c^{1}_{k} & ^{t}u_{1} \\
0 & 1 & 0 & 0 & \cdots & 0 & 0 & 0 \\
0 & c^{2}_{1} & 1 & c^{2}_{2} & \cdots & 0 & c^{2}_{k} & ^{t}u_{2} \\
0 & 0 & 0 & 1 & \cdots & 0 & 0 & 0 \\
\vdots & \vdots & \vdots & \vdots & \ddots & \vdots & \vdots & \vdots \\
0 & c^{k}_{1} & 0 & c^{k}_{2} & \cdots & 1 & c^{k}_{k} & ^{t}u_{k} \\
0 & 0 & 0 & 0 & \cdots & 0 & 1 & 0 \\
0 & x_{1} & 0 & x_{2} & \cdots & 0 & x_{k} & B
\end{pmatrix} \\
\\
= \begin{pmatrix}
0 & 1 & 0 & 0 & \cdots & 0 & 0 & 0 \\
1 & c^{1}_{1} + c^{1}_{1} + \langle x_{1}, x_{1} \rangle & 0 & c^{1}_{2} + c^{2}_{1} + \langle x_{1}, x_{2}\rangle & \cdots & 0 & c^{1}_{k} + c^{k}_{1} + \langle x_{1}, x_{k}\rangle & ^{t}u_{1} + \, ^{t}x_{1}\psi B \\
0 & 0 & 0 & 1 & \cdots & 0 & 0 & 0 \\
0 & c^{1}_{2} + c^{2}_{1} + \langle x_{2}, x_{1} \rangle & 1 & c^{2}_{2} + c^{2}_{2} + \langle x_{2}, x_{2} \rangle & \cdots & 0 & c^{2}_{k} + c^{k}_{2} + \langle x_{2}, x_{k} \rangle & ^{t}u_{2} + \, ^{t}x_{2}\psi B   \\
\vdots & \vdots & \vdots & \vdots & \ddots & \vdots & \vdots & \vdots \\
0 & 0 & 0 & 0 & \cdots & 0 & 1 & 0 \\
0 & c^{1}_{k} + c^{k}_{1} + \langle x_{k}, x_{1} \rangle & 0 & c^{2}_{k} + c^{k}_{2} + \langle x_{k}, x_{2} \rangle & \cdots & 1 & c^{k}_{k} + c^{k}_{k} + \langle x_{k}, x_{k} \rangle & ^{t}u_{k} + \, ^{t}x_{k}\psi B  \\
0 & u_{1} + \, ^{t}B\psi x_{1} & 0 & u_{2} + \, ^{t}B\psi x_{2} & \cdots & 0 & u_{k} + \, ^{t}B\psi x_{k} & ^{t}B\psi B
\end{pmatrix} \\
\\
= \psi.
\end{array}$$
}
Whence the equations. 
\end{proof}

To ease notation, we will from now on denote $T_{k} := St(e_{1}, \dots, e_{k})$. We will use the convention that $T_{0} = O_{n,n}$. 
Note, we may see from the structure of the matrices in $T_{k}$ that the projection map $\rho:T_{k} \twoheadrightarrow O_{n-k, n-k}$, sending the matrix $A$ in Proposition \ref{prop:stabilisers} to $\rho(A)=B$, defines a \textit{group homomorphism}. 
We denote its kernel by $L_{k}$, so that we have a short exact sequence of groups 
\begin{align}
1 \rightarrow L_{k} \rightarrow T_{k} \xrightarrow{\rho} O_{n-k, n-k} \rightarrow 1.    \label{s.e.s:localising}
\end{align} 
The associated Hochschild-Serre Spectral Sequence is 
\begin{align}
E^{2}_{p,q} = H_{p}(O_{n-k, n-k}; H_{q}(L_{k})) \Rightarrow H_{p+q}(T_{k}).        \label{ss:localising}
\end{align}

\subsubsection{The local $R^*$-action} \label{subsubsec:local $R^*$-action}

In this section, we will define an $R^{*}$-action on short exact sequence (\ref{s.e.s:localising}) which we call `local action'.
Using Spectral Sequence (\ref{ss:localising}), we will show that, \textit{after localisation}, the homology of $T_{k}$ and $O_{n-k, n-k}$ coincide.
In the next section, we will see that the local actions are induced by a `global' $R^{*}$-action on the Spectral Sequence (\ref{ss:secondhyperhomology}). 

\vspace{1ex}

\begin{definition}[Local action]
Let $a \in R^{*}$. For $0 \leq k \leq n$, define a $2n \times 2n$ matrix $D_{a,k}$ by
\[
D_{a,k} := 
\begin{pmatrix}
D_{a} \\
& \ddots \\
&&D_{a} \\
&&&1_{2n-2k}
\end{pmatrix}
= \left(\bigoplus_{1}^{k}D_{a}\right) \bigoplus 1_{2(n-k)}, \quad D_{a} :=
\begin{pmatrix}
a & 0 \\
0 & a^{-1}
\end{pmatrix}.
\]
Note that $D_{a,k} \in O_{n,n}(R)$. 
The {\em local action} of $R^*$ on $T_k$ is the conjugation action of $D_{a,k}$ on $T_{k}$. 
\end{definition}

The local action preserves $T_k$ because
\begin{align*}
D_{a,k} 
&\begin{pmatrix}
1 & c^{1}_{1} & 0 & c^{1}_{2} & \cdots & 0 & c^{1}_{k} & ^{t}u_{1} \\
0 & 1 & 0 & 0 & \cdots & 0 & 0 & 0 \\
0 & c^{2}_{1} & 1 & c^{2}_{2} & \cdots & 0 & c^{2}_{k} & ^{t}u_{2} \\
0 & 0 & 0 & 1 & \cdots & 0 & 0 & 0 \\
\vdots & \vdots & \vdots & \vdots & \ddots & \vdots & \vdots & \vdots \\
0 & c^{k}_{1} & 0 & c^{k}_{2} & \cdots & 1 & c^{k}_{k} & ^{t}u_{k} \\
0 & 0 & 0 & 0 & \cdots & 0 & 1 & 0 \\
0 & x_{1} & 0 & x_{2} & \cdots & 0 & x_{k} & B
\end{pmatrix}
D_{a,k}^{-1} \\
&\\
= &\begin{pmatrix}
1 & a^{2}c^{1}_{1} & 0 & a^{2}c^{1}_{2} & \cdots & 0 & a^{2}c^{1}_{k} & {}^{t}u_{1}a \\
0 & 1 & 0 & 0 & \cdots & 0 & 0 & 0 \\
0 & a^{2}c^{2}_{1} & 1 & a^{2}c^{2}_{2} & \cdots & 0 & a^{2}c^{2}_{k} & {}^{t}u_{2}a \\
0 & 0 & 0 & 1 & \cdots & 0 & 0 & 0 \\
\vdots & \vdots & \vdots & \vdots & \ddots & \vdots & \vdots & \vdots \\
0 & a^{2}c^{k}_{1} & 0 & a^{2}c^{k}_{2} & \cdots & 1 & a^{2}c^{k}_{k} & {}^{t}u_{k}a \\
0 & 0 & 0 & 0 & \cdots & 0 & 1 & 0 \\
0 & ax_{1} & 0 & ax_{2} & \cdots & 0 & ax_{k} & B
\end{pmatrix} \in T_{k}.
\end{align*}
Also notice that this conjugation action restricted to $O_{n-k, n-k} \subset T_{k}$ is \textit{trivial}. Thus, we have an $R^{*}$-action on short exact sequence (\ref{s.e.s:localising}), which will induce an $R^{*}$-action on Spectral Sequence (\ref{ss:localising}).

We now introduce the idea of \textit{localising homology groups}.
Let $m \geq 1$ be an integer. 
Choose units $a_{1}, \dots, a_{m} \in R^{*}$ such that for every non-empty subset $I \subset \{1,\dots,m\}$, the partial sum $a_{I} := \sum_{i\in I} a_{i}$ is a unit in $R$. 
Call such a sequence $(a_{1}, \dots, a_{m})$ an $S(m)$-sequence. 
Choosing an $S(m)$-sequence is possible for every $m >0$  because $R$ has infinite residue field. 

Let $s_{m} \in \mathbb{Z}[R^{*}]$ be the element 
\[
s_{m} = -\sum_{\emptyset \neq I \subset \{1,\dots,m\}}(-1)^{|I|}\langle a_{I} \rangle \in \mathbb{Z}[R^{*}],
\]
first considered in \cite{schlichting2017euler}, where $\langle u \rangle \in \mathbb{Z}[R^{*}]$ denotes the element of the group ring corresponding to $u \in R^{*}$. 
Note that 
\[
1 = -\sum_{\emptyset \neq I \subset \{1,\dots,m\}}(-1)^{|I|},
\]
so that a trivial $R^{*}$-action induces a trivial action by the elements $s_{m}$. 
If $R^*$ acts on a group $G$ through group homomorphisms, then the homology groups $H_n(G)$ aquire an $R^*$-action by functoriality of group homology. 
This makes the groups $H_n(G)$ into a left module over the commutative ring $\Z[R^*]$, and we can localize them with respect to the element $s_{m} \in \Z[R^*]$ to obtain the abelian groups $s_{m}^{-1}H_n(G)$.
The magic of the elements $s_{m}$ lie in the following proposition, for the proof of which we refer the reader to \cite[Proposition D.4.]{schlichting2019higher}.
\vspace{1ex}

\begin{proposition} 
\label{prop:magicallykill}
Let $R$ be a commutative ring and  $u = (u_{1}, \dots, u_{m})$ an $S(m)$-sequence in $R$. Let 
\[
1\rightarrow N \rightarrow G \rightarrow A \rightarrow 1
\]
be a central extension of groups. Assume that the group of units $R^{*}$ acts on the exact sequence. Assume furthermore that the groups $A$ and $N$ are the underlying abelian groups $(A, +, 0)$ and $(N,+,0)$ of $R$-modules $(A, +, 0,\cdot)$ and $(N,+,0,\cdot)$, and the $R^{*}$-actions on $A$ and $N$ in the exact sequence are  given by 
\[
R^{*} \times A \rightarrow A : (t,a) \mapsto t \cdot a
\]
and 
\[
R^{*} \times N \rightarrow N : (t,y) \mapsto t^{2} \cdot y
\]
respectively. 
Then, for every $1 \leq n <  m/2 $, 
\[
s_{m}^{-1}H_{n}(G) = 0.
\]
\end{proposition}
\qed

We may now \textit{localise} the Spectral Sequence  (\ref{ss:localising}) with respect to the elements $s_{m}$ to obtain for all $m \geq 1$ the localised spectral sequences
\begin{equation}
\label{ss:localised}
\begin{array}{ll}
s_{m}^{-1}E^{2}_{pq} &= s_{m}^{-1}H_{p}(O_{n-k,n-k}; H_{q}(L_{k}) )\\
&\\
& \cong H_{p}(O_{n-k,n-k}; s_{m}^{-1}H_{q}(L_{k})) 
\end{array}
\Rightarrow s_{m}^{-1}H_{p+q}(T_{k}),   
\end{equation}
the isomorphism coming from the fact that $R^{*}$ acts trivially on $O_{n-k,n-k}$. 

We show that localising with respect to the elements $s_{m}$ \textit{kills} the non-zero homology groups of $L_{k}$ when $m$ is taken to infinity. 

\begin{lemma} \label{lemma:kill}
We have $s_{m}^{-1}H_{0}(L_{k}) = \mathbb{Z} $ and for all $1\leq 2q < m$, $s_{m}^{-1}H_{q}(L_{k}) = 0$.
\end{lemma}
\begin{proof}
We claim there is a short exact sequence of groups 
\begin{align}
1 \rightarrow (R^{k \choose 2},+) \rightarrow L_{k} \rightarrow ((R^{2(n-k)})^{k},+) \rightarrow 1. \label{ses:junk}
\end{align}
The first arrow maps 
\[
(c_{1}, \dots ) \mapsto A_{(c_{1}, \dots )}
\]
where $A_{(c_{1}, \dots )} \in L_{k}$ is defined by the conditions (\ref{eqn:stab1}), (\ref{eqn:stab2}) and (\ref{eqn:stab3}) subject to $B = 1$, $x_{i} = 0$ and using equation (\ref{eqn:stab2}) to determine the remaining constants (with some ordering specified beforehand). Note that we have used here that 2 is invertible, as (\ref{eqn:stab2}) implies $2c^{i}_{i} = 0$ for all $i$. 
The second arrow maps 
\[
\begin{pmatrix}
1 & c^{1}_{1} & 0 & c^{1}_{2} & \cdots & 0 & c^{1}_{k} & ^{t}u_{1} \\
0 & 1 & 0 & 0 & \cdots & 0 & 0 & 0 \\
0 & c^{2}_{1} & 1 & c^{2}_{2} & \cdots & 0 & c^{2}_{k} & ^{t}u_{2} \\
0 & 0 & 0 & 1 & \cdots & 0 & 0 & 0 \\
\vdots & \vdots & \vdots & \vdots & \ddots & \vdots & \vdots & \vdots \\
0 & c^{k}_{1} & 0 & c^{k}_{2} & \cdots & 1 & c^{k}_{k} & ^{t}u_{k} \\
0 & 0 & 0 & 0 & \cdots & 0 & 1 & 0 \\
0 & x_{1} & 0 & x_{2} & \cdots & 0 & x_{k} & 1
\end{pmatrix}
\mapsto (x_{1}, \dots, x_{k}).
\]
One may check that these arrows define group homomorphisms, fitting into the short exact sequence  (\ref{ses:junk}), and  (\ref{ses:junk}) is actually a central extension. 

Furthermore, this central extension is $R^{*}$-equivariant where $b \in R^{*}$ acts on $(R^{k \choose 2},+)$ via pointwise multiplication by $b^{2}$, the element $b\in R^*$ acts on $L_{k}$ via conjugation by $D_{b,k}$, and it acts on $((R^{2(n-k)})^{k},+)$ via pointwise multiplication by $b$. 
By Proposition \ref{prop:magicallykill}, $s_{m}^{-1}H_{q}(L_{k}) = 0$ for all $1\leq 2q < m$. The equality $s_{m}^{-1}H_{0}(L_{k}) = \mathbb{Z}$ follows from fact that $R^{*}$ acts trivially on $H_{0}(L_{k})$.
\end{proof}
\begin{corollary} \label{corollary:stabiso}
The inclusion $O_{n-k,n-k} \hookrightarrow T_{k}$ induces isomorphism 
\[
H_{t}(O_{n-k, n-k}) \xrightarrow{\cong} s_{m}^{-1}H_{t}(T_{k})
\]
for all $t < m/2$.
\end{corollary}
\begin{proof}
By Lemma \ref{lemma:kill}, the localised Hochschild-Serre Spectral Sequence  \label{ss:localised} degenerates at $E^{2}$ for $1 \leq 2t < m$ to yield isomorphism 
\[
\rho: s_{m}^{-1}H_{t}(T_{k}) \xrightarrow{\cong} H_{t}(O_{n-k, n-k})
\]
for all $t < m/2$. Since $\rho$ is a retract of the inclusion, we are done. 
\end{proof}

\subsubsection{A global action on the spectral sequence} \label{subsubsection:globalaction}

Next, we want to realise these `local actions' as a `global action' on the spectral sequence 
\begin{align}
E^{1}_{p,q}(n) &= H_{q}(O_{n,n}, C_{p}(n)) \Rightarrow H_{p+q}(O_{n,n}, C_{*}(n))  \label{ss:hyperhomologyII}.
\end{align}
We do this by defining an action on the associated exact couple with abutment. 

Recall that for a group $G$ and a chain complex of $G$-modules $C_{*}$, the spectral sequence 
\[
E^{1}_{p,q} = H_{q}(G, C_{p}) \Rightarrow H_{p+q}(G, C_{*}) 
\]
may be obtained from the exact couple with abutment 
\begin{equation}
\begin{tikzcd}
\bigoplus_{p,q} E^{1}_{p,q} \arrow{r}{k} 
  & \bigoplus_{p,q} D^{1}_{p,q} \arrow{d}{i} \arrow{r}{\sigma} 
    & \bigoplus_{p+q} A_{p+q}  \\
    & \bigoplus_{p,q} D^{1}_{p,q} \arrow{lu}{j} \arrow{ru}{\sigma}
\end{tikzcd} \label{exactcouple:hyperhomologyII}
\end{equation}
with $E^{1}_{p,q} = H_{p+q}(G, C_{\leq p}/C_{\leq p-1})$; $D^{1}_{p,q} = H_{p+q}(G, C_{\leq p})$; $A_{p+q} = H_{p+q}(G,C_{*})$; the maps $i, j ,k$ being the maps of the long exact sequence of homology groups associated to the short exact sequence of complexes 
\[
0 \rightarrow C_{\leq p-1} \rightarrow C_{\leq p} \rightarrow C_{\leq p}/C_{\leq p-1} \rightarrow 0,
\]
and $\sigma$ is induced by the inclusion.

To define the global action, it will be convenient to introduce the \textit{general} split orthogonal group, which is defined as follows.
\vspace{1ex}

\begin{definition}
For a ring $R$, define $GO_{n,n}(R) \subset GL_{2n}(R)$ as the subgroup 
\[
GO_{n,n}(R):= \{ A \in GL_{2n}(R) |  \,\, ^{t}A\psi_{2n}A = a\psi_{2n}, \,\,\text{for some} \,\, a \in R^{*} \}.
\] 
In the above notation, we will call $a \in R^{*}$ the \textit{associated unit} of $A$.
\end{definition}

For $n\geq 1$ we have short exact sequence of groups
\begin{align}
1 \rightarrow O_{n,n} \rightarrow GO_{n,n} \rightarrow R^{*} \rightarrow 1    \label{ses:GOnn}
\end{align}
where the first arrow is given by the inclusion and the second arrow maps $A\in GO_{n,n}$ to its associated unit. 
For instance, for $a \in R^{*}$, the matrix
\[
B_{a} := 
\begin{pmatrix}
1 \\
 &a \\
 && \ddots \\
 &&&1 \\
 &&&&a
\end{pmatrix}
\]
is in $GO_{n,n}$ and has associated unit $a$ which proves exactness at the right.

\begin{definition}[Global action]
The group homomorphism $GO_{n,n}(R) \to R^*$ makes $\Z[R^*]$ into a right $GO_{n,n}(R)$-module and left $R^*$-module, and both actions commute.
In particular, for any bounded below complex of $GO_{n,n}$-modules $M_{*}$, 
the groups 
$$\Tor^{GO_{n,n}}_{i}(\Z[R^*],M_{*}) = H_i(\mathbb{Z}[R^{*}]\otimes_{GO_{n,n}}^{\mathbb{L}}M_{*})$$
are left $\Z[R^*]$-modules functorial in $M_{*}$, and the spectral sequence
\begin{equation}
\label{eqn:RstarSpSeq}
E^1_{p,q}=\Tor^{GO_{n,n}}_{q}(\Z[R^*],M_p) \Rightarrow \Tor^{GO_{n,n}}_{p+q}(\Z[R^*],M_{*})
\end{equation}
is a spectral sequence of left $R^*$-modules. (For more information about the derived tensor product, see \cite{weibel1994introduction}). This spectral sequence is the spectral sequence of the exact couple (\ref{exactcouple:hyperhomologyII}) with 

$E^{1}_{p,q} = \Tor^{GO_{n,n}}_{p+q}(\Z[R^*], M_{\leq p}/M_{\leq p-1})$; $D^{1}_{p,q} = \Tor^{GO_{n,n}}_{p+q}(\Z[R^*], M_{\leq p})$; $A_{p+q} = \Tor^{GO_{n,n}}_{p+q}(\Z[R^*],M_{*})$.
For $n\geq 1$, the inclusions $\Z \subset \Z[R^*]:1 \mapsto 1$ and $O_{n,n} \subset GO_{n,n}$ yield isomorphisms 
$$\mathbb{Z}\otimes_{O_{n,n}}^{\mathbb{L}}M_{*} \stackrel{\sim}{\longrightarrow}
\mathbb{Z}[R^{*}]\otimes_{GO_{n,n}}^{\mathbb{L}}M_{*}$$
by Shapiro's Lemma.
For $M_{*}=C_{*}(n)$, this identifies the Spectral Sequence (\ref{eqn:RstarSpSeq}) with (\ref{ss:hyperhomologyII})  and makes the latter  into a spectral sequence of $R^*$-modules.
We use this structure to define the {\em global action} of $R^*$ on (\ref{ss:hyperhomologyII}).

\end{definition}

Specifically, we now show that, under the isomorphism (\ref{eqn:ShapiroTkOn}), the local actions corresponding to conjugation with $D_{a,k}$ are induced by the global action corresponding to multiplication with $a^{-2} \in R^{*}$. 
For this end, we will need to prove that the appropriate diagrams commute. 
We will use the following two lemmas.

\begin{lemma} \label{lemma:homologycommute}
Let $G$ and $K$ be groups.
Let $M$ and $N$ be a $G$-module and a $K$-module, respectively.
Consider the diagram of morphisms
\[
(G,M) \stackrel[(f_{2}, \varphi_{2})]{(f_{1}, \varphi_{1})}\rightrightarrows(K,N)
\]
where $f_{1}, f_{2}$ are group homomorphisms and $\varphi_{1}, \varphi_{2}$ $G$-module homomorphisms, $N$ is considered a $G$-module via $f_{1}$ and $f_{2}$ respectively. 
Suppose there exists a $\kappa \in K$ such that for all $g \in G$ and for all $m \in M$, 
$$f_{2}(g) = \kappa f_{1}(g)\kappa^{-1} \hspace{4ex}\text{and}\hspace{4ex}
\varphi_{2}(m) = \kappa \varphi_{1}(m).$$
Then
\[
(f_{1}, \varphi_{1})_{*} = (f_{2}, \varphi_{2})_{*} : H_{*}(G,M) \rightarrow H_{*}(K,N).
\]
\end{lemma}

\begin{proof}
By assumption, we have the following commutative diagram:
\[
\begin{tikzcd}
H_{*}(G,M)\arrow[swap]{d}{(f_{1}, \varphi_{1})_{*}} \arrow{dr}{(f_{2}, \varphi_{2})_{*}} \\
H_{*}(K,N) \arrow[swap]{r}{(c_{\kappa}, \mu_{\kappa})_{*}} & H_{*}(K,N),
\end{tikzcd}
\]
where $(c_{\kappa}, \mu_{\kappa}) : (K,N) \rightarrow (K,N)$ is the map $(k,n) \mapsto (\kappa k \kappa^{-1}, \kappa n)$.
By \cite[Chapter III.8]{brown1982cohomology}, the bottom horizontal map equals the identity. 
\end{proof}

We will also need to recall functoriality of Tor. This is given by the following lemma.

\begin{lemma} \label{lemma:torcommute}
Let $G$ and $K$ be groups; let $M$ and $P$ be a right $G$-module and right $K$-module respectively; and let $N$ and $Q$ be a left $G$-module and left $K$-module respectively. 

Consider the diagram of morphisms
\[
(M,G,N) \stackrel[(f_{2}, \varphi_{2}, g_{2})]{(f_{1}, \varphi_{1},g_{1})}\rightrightarrows(P,K,Q)
\]
where $\varphi_{1}, \varphi_{2}$ are group homomorphisms; $f_{1}, f_{2}$ right $G$-module homomorphisms where $P$ is considered a right $G$-module via $\varphi_{1}$ and $\varphi_{2}$ respectively and  $g_{1}, g_{2}$ left $G$-module homomorphisms where $Q$ is considered a left $G$-module via $\varphi_{1}$ and $\varphi_{2}$ respectively.

Suppose there exists a $\kappa \in K$ such that for all $g \in G$; for all $m \in M$ and for all $n \in N$,
$$f_{2}(m) = f_{1}(m)\kappa^{-1} \hspace{4ex},\hspace{4ex}
\varphi_{2}(g) = \kappa \varphi_{1}(g) \kappa^{-1} \hspace{4ex}\text{and}\hspace{4ex} 
g_{2}(m) = \kappa g_{1}(n).$$
Then
\[
(f_{1}, \varphi_{1}, g_{1})_{*} = (f_{2}, \varphi_{2},g_{2})_{*} : \Tor_{*}^{G}(M,N) \rightarrow \Tor_{*}^{K}(P,Q).
\]
\end{lemma}

\begin{proof}
By assumption, the following diagram commutes:
\[
\begin{tikzcd}
H_{*}(M \otimes^{\mathbb{L}}_{G}N) \arrow[swap]{d}{(f_{1}, \varphi_{1}, g_{1})_{*}} \arrow{drr}{(f_{2}, \varphi_{2}, g_{2})_{*}} \\
H_{*}(P \otimes^{\mathbb{L}}_{K}Q) \arrow[swap]{rr}{(\lambda_{\kappa^{-1}},c_{\kappa}, \mu_{\kappa})_{*}} & &H_{*}(P \otimes^{\mathbb{L}}_{K}Q),
\end{tikzcd}
\]
where 
$$\lambda_{\kappa^{-1}} : P \rightarrow P, p \mapsto p\kappa^{-1}, \quad
c_{\kappa} :K \rightarrow K, k \mapsto \kappa k \kappa^{-1},\hspace{4ex} 
\mu_{\kappa} : Q \rightarrow Q, q \mapsto \kappa q.$$
Therefore, the bottom horizontal map is induced by the map 
\begin{align*}
&P\otimes_{G}Q \rightarrow P\otimes_{G}Q \\
&p\otimes q \mapsto p\kappa^{-1}\otimes \kappa q = p\otimes q
\end{align*}
which is the identity. 
\end{proof}
By definition the $R^{*}$-action on $\Tor_{q}^{GO_{n,n}}(\mathbb{Z}[R^{*}], C_{k}(n))$ corresponding to left multiplication with $a \in R^{*}$ is induced by the map 
\[
(\mathbb{Z}[R^{*}], GO_{n,n}, C_{k}(n)) \xrightarrow{(a,id,id)} (\mathbb{Z}[R^{*}], GO_{n,n}, C_{k}(n)) 
\]
where $a : \mathbb{Z}[R^{*}] \rightarrow \mathbb{Z}[R^{*}] $ is the map corresponding to left multiplication with $a$. 
With this, we have the following proposition, which gives us a model of this action in terms of the groups $\Tor_{q}^{O_{n,n}}(\mathbb{Z}, C_{k}(n)) \cong H_{q}(O_{n,n}, C_{k}(n))$.

\begin{proposition}
\label{prop:ActionaBa}
Let $k,q \geq 0$ and $n\geq 1$. Then, for all $a \in R^{*}$, the following diagram commutes:
\[
\begin{tikzcd}
\Tor_{q}^{GO_{n,n}}(\mathbb{Z}[R^{*}], C_{k}(n)) \arrow{r}{(a,id,id)_{*}}
  & \Tor_{q}^{GO_{n,n}}(\mathbb{Z}[R^{*}], C_{k}(n))\\
\Tor_{q}^{O_{n,n}}(\mathbb{Z}, C_{k}(n)) \arrow{u}{(i,i,id)_{*}}[swap]{\cong} \arrow[swap]{r}{({id,C_{B_{a}}, B_{a})_{*}}}
  & \Tor_{q}^{O_{n,n}}(\mathbb{Z}, C_{k}(n)) \arrow[swap]{u}{(i,i,id)_{*}}[swap]{\cong},
\end{tikzcd}
\]
where the vertical maps are the isomorphisms given by Shapiro's Lemma; $B_{a}$ denotes left multiplication by $B_{a} \in GO_{n,n}$ and $C_{B_{a}}$ denotes the map induced by conjugation with the element $B_{a}$ on $O_{n,n}$. 
\end{proposition}

\begin{proof}
We use Lemma \ref{lemma:torcommute}. Specifically, consider the diagram 
\[
(\mathbb{Z},O_{n,n},C_{k}(n)) \stackrel[(f_{2}, \varphi_{2}, g_{2})]{(f_{1}, \varphi_{1},g_{1})}\rightrightarrows(\mathbb{Z}[R^{*}],GO_{n,n},C_{k}(n))
\]
where $(f_{1}, \varphi_{1},g_{1}) := (\mu_{a}, i, 1)$ and $(f_{2}, \varphi_{2},g_{2}) := (i, C_{B_{a}}, \mu_{B_{a}})$. Here, $\mu_{a}$ is defined via $\mu_{a}(1) := a$ and $\mu_{B_{a}}$ is defined via left multiplication on basis elements by $B_{a}$.
Let $\kappa:= B_{a} \in GO_{n,n}$. Note that from short exact sequence  (\ref{ses:GOnn}), we deduce $B_{a}$ acts on $R^{*}$ by multiplication with $a$. Thus, $i(1) = 1 = \mu_{a}(1)\kappa^{-1}$.
Furthermore, $C_{B_{a}} = \kappa i \kappa^{-1}$ and for every $(v_{1}, \dots, v_{k}) \in \mathcal{IU}_{k}(R^{2n})$, $\mu_{B_{a}}(v_{1}, \dots, v_{k}) = \kappa (v_{1}, \dots, v_{k})$ (the case $k=0$ being trivial). Thus, by Lemma \ref{lemma:torcommute}, the diagram commutes. 
\end{proof}

Next, note that the action on $\Tor_{q}^{O_{n,n}}(\mathbb{Z}, C_{k}(n))$ induced by 
\[
(\mathbb{Z}, O_{n,n}, C_{k}(n)) \xrightarrow{(id,C_{B_{a}},B_{a})} (\mathbb{Z}, O_{n,n}, C_{k}(n)) 
\]
is equivalent to the action on $H_{q}(O_{n,n}, C_{k}(n))$ induced by 
\[
(O_{n,n}, C_{k}(n)) \xrightarrow{(C_{B_{a}},B_{a})} (O_{n,n}, C_{k}(n)).
\]
To make the connection with the action induced by conjugation with $D_{a,k}$, we prove the following intermediate proposition. 

\begin{proposition}
\label{prop:ActionBa2Phia}
Let $k,q \geq 0$ and $n\geq 1$. Then, for all $a \in R^{*}$, the following diagram commutes:
\[
\begin{tikzcd}
H_{q}(O_{n,n}, C_{k}(n)) \arrow{rr}{(C_{B_{a^{-2}}},B_{a^{-2}})_{*}}
  && H_{q}(O_{n,n}, C_{k}(n)) \\
H_{q}(O_{n,n},C_{k}(n)) \arrow{u}{id} \arrow[swap]{rr}{(id, \phi_{a})_{*}}
  && H_{q}(O_{n,n},C_{k}(n)) \arrow[swap]{u}{id},
\end{tikzcd}
\]
where for $a \in R^{*}$, the map 
\[
(id, \phi_{a}): (O_{n,n}, C_{k}(n)) \rightarrow (O_{n,n}, C_{k}(n))
\]
is defined to be the identity on $O_{n,n}$ and on basis elements of $C_{k}(n)$ as 
\[
\phi_{a} : (v_{1}, \dots , v_{k}) \mapsto (a^{-1}v_{1}, \dots, a^{-1}v_{k}).
\]
\end{proposition}

\begin{proof}
We use Lemma \ref{lemma:homologycommute}. Specifically, consider the diagram 
\[
(O_{n,n},C_{k}(n)) \stackrel[(f_{2}, \varphi_{2})]{(f_{1}, \varphi_{1})}\rightrightarrows(O_{n,n},C_{k}(n))
\]
where $(f_{1}, \varphi_{1}) := (id, \phi_{a})$ and $(f_{2}, \varphi_{2}) := (C_{B_{a^{-2}}}, B_{a^{-2}})$. 
Define
\[
\kappa := D_{a,n} =
 \begin{pmatrix}
a \\
& a^{-1}\\
&&\ddots \\
&&& a \\
&&&& a^{-1}
\end{pmatrix}.
\]
Denoting for $a \in R^{*}$, 
\[
\underline{a} :=  \begin{pmatrix}
a \\
& a\\
&&\ddots \\
&&& a \\
&&&& a
\end{pmatrix},
\]
note that $B_{a^{-2}} = \kappa \underline{a}^{-1}$, so that $C_{B_{a^{-2}}} = C_{\kappa}C_{\underline{a}^{-1}}$. But, $C_{\underline{a}^{-1}} = id$, so that $C_{B_{a^{-2}}} = C_{\kappa}$. 
Furthermore, note that for every $(v_{1}, \dots, v_{k}) \in \mathcal{IU}_{k}(R^{2n})$, $B_{a^{-2}}(v_{1}, \dots, v_{k}) = \kappa \phi_{a} (v_{1}, \dots, v_{k})$, since $B_{a^{-2}} = \kappa \underline{a}^{-1}$. Thus, by Lemma \ref{lemma:homologycommute}, the diagram commutes. 
\end{proof}

Finally, we show that $(id, \phi_{a})$ induces the desired local actions.
\begin{proposition}
\label{prop:ActionDaPhia}
Let $k,q \geq 0$. Then, for all $a \in R^{*}$, the following diagram commutes:
\[
\begin{tikzcd}
H_{q}(O_{n,n}, C_{k}(n)) \arrow{r}{(id,\phi_{a})_{*}}
  & H_{q}(O_{n,n}, C_{k}(n)) \\
H_{q}(T_{k},\mathbb{Z}) \arrow{u}{(i,(e_{1}, \dots, e_{k}))_{*}}[swap]{\cong} \arrow[swap]{r}{C_{D_{a,k}}}
  & H_{q}(T_{k},\mathbb{Z}) \arrow[swap]{u}{(i,(e_{1}, \dots, e_{k}))_{*}}[swap]{\cong},
\end{tikzcd}
\]
where the vertical arrows are the isomorphisms given by Shapiro's Lemma and the map $C_{D_{a,k}}$ denotes the map induced by conjugation with the element $D_{a,k}$ on the stabiliser $T_{k}$.
\end{proposition}

\begin{proof}
We use Lemma \ref{lemma:homologycommute}. Specifically, we have to consider the diagram 
\[
(T_{k},\mathbb{Z}) \stackrel[(f_{2}, \varphi_{2})]{(f_{1}, \varphi_{1})}\rightrightarrows(O_{n,n},C_{k}(n))
\]
where $(f_{1}, \varphi_{1}) := (i, (a^{-1}e_{1}, \dots, a^{-1}e_{k}))$ and $(f_{2}, \varphi_{2}) := (iC_{D_{a,k}}, (e_{1}, \dots, e_{k}))$,
and $i:T_k \to O_{n,n}$ is the natural inclusion of groups.
Let $\kappa = D_{a,k} \in O_{n,n}$. 
Then, for every $A \in T_{k}$, 
$$
f_{2}(A) = iC_{D_{a,k}}(A) 
= D_{a,k}A D_{a,k}^{-1} 
= D_{a,k}i(A) D_{a,k}^{-1} 
= \kappa f_{1}(A) \kappa^{-1}
$$
and
\[
(e_{1}. \dots, e_{k}) = D_{a,k}(a^{-1}e_{1}, \dots, a^{-1}e_{k}) = \kappa (a^{-1}e_{1}, \dots, a^{-1}e_{k}).
\]
By Lemma \ref{lemma:homologycommute}, the diagram commutes.  
\end{proof}

Thus, we have shown that there exists an $R^{*}$-action on the spectral sequence 
\begin{align*}
E^{1}_{p,q}(n) &= H_{q}(O_{n,n}, C_{p}(n)) \Rightarrow H_{p+q}(O_{n,n}, C_{*}(n))
\end{align*}
which induces the desired local actions considered previously. 
Using Corollary \ref{corollary:stabiso}, we obtain the following.

\begin{corollary}
For every $m \geq 1$, the localised spectral sequence 
\begin{align}
{}_{m}E^{1}_{p,q}(n) = s_{m}^{-1}E^{1}_{p,q}(n) \Rightarrow s_{m}^{-1}H_{p+q}(O_{n,n}, C_{*}(n)) \label{ss:localisedII}
\end{align}
has $_{m}E^{1}_{p,q}$ terms 
\[
{}_{m}E^{1}_{p,q} = s_{m}^{-1}H_{q}(O_{n,n}, C_{p}(n)) \cong H_{q}(O_{n-p, n-p})
\]
for all $q < m/2$ and for all $p \leq n$. 
\end{corollary}\qed

Under these identifications, the differentials $d^{1}:  \,\,_{m}E^{1}_{p,q} \rightarrow \,\, _{m}E^{1}_{p-1,q} $ take the form 
\[
d^{1}: H_{q}(O_{n-p, n-p}) \rightarrow H_{q}(O_{n-p+1, n-p+1})
\]
whenever $q < m/2$ and $p \leq n$. 
Our next task is to compute these differentials.

\subsection{Computation of the localised $d^{1}$ differentials, and proof of homological stability}

\begin{proposition} \label{prop:computelocalisedd1}
For all $q < m/2$ and $p \leq n$, 
the homomorphism $d^{1}_{p,q}: H_{q}(O_{n-p, n-p}) \rightarrow H_{q}(O_{n-p+1, n-p+1})$ is 
\[
d^1_{p,q}=
\begin{cases}
0, & p \,\, \text{even} \\
i_{*}, & p \,\, \text{odd},
\end{cases}
\]
where $i : O_{n-p, n-p} \hookrightarrow O_{n-p+1, n-p+1}$ denotes the inclusion. 
\end{proposition}

\begin{proof}
For all $p\leq n$, we want to show that the following diagram commutes:
\begin{equation}
\begin{tikzcd}
H_{q}(O_{n-p,n-p}) \arrow[swap]{d}{i_{*}} \arrow{rrr}{(\iota,(e_{1}, \dots, e_{p}))_{*}}
  &&& H_{q}(O_{n,n},C_{p}(n)) \arrow{d}{(d_{i})_{*}}  \\
H_{q}(O_{n-p+1,n-p+1})  \arrow[swap]{rrr}{(\iota,(e_{1}, \dots, e_{p-1}))_{*}}
  &&& H_{q}(O_{n,n}, C_{p-1}(n)),
\end{tikzcd}  \label{cd:computedifferentials}
\end{equation}
where $\iota : O_{n-p, n-p} \hookrightarrow O_{n,n}$ denotes the inclusion map; $(e_{1}, \dots, e_{p}) : 1 \mapsto (e_{1}, \dots, e_{p})$ and recall that $d_{i}(v_{1}, \dots, v_{p}) = (v_{1},\dots, \hat{v}_{i}. \dots, v_{p} )$.
Again, we will prove this diagram commutes using Lemma \ref{lemma:homologycommute}. 
Specifically, consider the diagram 
\[
(O_{n-p,n-p},\mathbb{Z}) \stackrel[(\iota  \circ i, (e_{1}, \dots, e_{p-1}))]{(\iota, (e_{1}, \dots, \hat{e}_{i}, \dots, e_{p}))}\rightrightarrows(O_{n,n},C_{p-1}(n)).
\]
Define $A \in O_{n,n}$ by sending a hyperbolic basis to a hyperbolic basis as follows:
\begin{align*}
(e_{1}, \dots, \hat{e}_{i}, \dots, e_{p}) &\mapsto (e_{1}, \dots, e_{p-1}) \\
(f_{1}, \dots, \hat{f}_{i}, \dots, f_{p}) &\mapsto (f_{1}, \dots, f_{p-1}) \\
e_{i} &\mapsto e_{p} \\
f_{i} &\mapsto f_{p} \\
e_{j} \mapsto e_{j} \,\, \text{and} \,\, f_{j} \mapsto f_{j} & \,\, \text{for all} \,\, p+1\leq j \leq n. 
\end{align*}
Then, by construction, $(e_{1}, \dots, e_{p-1}) = A (e_{1}, \dots, \hat{e}_{i}, \dots e_{p})$.
Note the matrix of $A$ has the form
\[
A = 
\begin{pmatrix}
\sigma & 0 \\
0 & 1_{2(n-p)}
\end{pmatrix}
\]
for some permutation matrix $\sigma$. Therefore, we deduce that for every $B \in O_{n-p, n-p}$, 
\[
\iota \circ i (B) = A \iota (B) A^{-1}.
\]
By Lemma \ref{lemma:homologycommute}, the diagram commutes. 
The proposition then follows from the fact that the differential $d^{1} : H_{q}(O_{n,n}, C_{p}(n)) \rightarrow H_{q}(O_{n,n}, C_{p-1}(n))$ is induced by the differential $d = \sum_{i=1}^{p}(-1)^{i+1}d_{i} : C_{p}(n) \rightarrow C_{p-1}(n)$ and the above remains true after localisation, with the horizontal arrows becoming the identification isomorphisms. 
\end{proof}

We immediately deduce the following corollary:

\begin{corollary} \label{corollary:mE2}
For all $q < m/2$ and for all $p \leq n$, 
\[
{}_{m}E^{2}_{p,q} = 
\begin{cases}
\ker \left( H_{q}(O_{n-p,n-p}) \xrightarrow{i_{*}} H_{q}(O_{n-p+1, n-p+1}))\right), & p \,\, \text{odd} \\
\coker \left( H_{q}(O_{n-p-1,n-p-1}) \xrightarrow{i_{*}} H_{q}(O_{n-p, n-p}))\right), & p \,\, \text{even}.
\end{cases}
\]
\end{corollary}
\qed

To prove homological stability, we will need to prove the following.
\begin{proposition} \label{prop:zerodifferentials}
The differentials $d^{r}_{p,q}$ in Spectral Sequence (\ref{ss:localisedII}) are zero for $r \geq 2$ and $q < m/2$, $p \leq n$. Hence, for all $q < m/2$ and $p \leq n$, ${}_{m}E^{2}_{p,q} \cong {}_{m}E^{\infty}_{p,q}$.
\end{proposition}

\begin{proof}
Similar to \cite{nesterenko1990homology} and \cite{schlichting2017euler}, we argue by induction on $n$. 
For $n = 0,1$, the spectral sequence under consideration is located in columns 0 and 1. Therefore, the differentials $d^{r}$ for $r \geq 2$ are zero by dimension arguments. 

Assume $n \geq 2$. We seek to define a homomorphism of complexes of $O_{n-2,n-2}$-modules
\[
\tau : C_{*}(n-2)[-2] \rightarrow C_{*}(n).
\]
For $(v_{1}, \dots, v_{p-2}) \in C_{p}(n-2)[-2]$, define 
\begin{align*}
\tau_{0}(v_{1}, \dots, v_{p-2}) &:= (e_{1}, e_{2}, \bar{v}_{1}, \dots, \bar{v}_{p-2}) \\
\tau_{1}(v_{1}, \dots, v_{p-2}) &:= (e_{1}, e_{2} - e_{1}, \bar{v}_{1}, \dots, \bar{v}_{p-2})  \\
\tau_{2}(v_{1}, \dots, v_{p-2}) &:= (e_{2}, e_{2} - e_{1}, \bar{v}_{1}, \dots, \bar{v}_{p-2}),
\end{align*}
where 
\[
\bar{v}_{i} := 
\begin{pmatrix}
0 \\
0 \\
0 \\
0 \\
v_{i}
\end{pmatrix} \in R^{2n}.
\]
Define $\tau := \tau_{0} - \tau_{1} + \tau_{2}$. 
Note that $\tau$ commutes with differentials and commutes with $O_{n-2,n-2}$ multiplication from the left, so that it indeed defines a homomorphism of chain complexes of $O_{n-2,n-2}$-modules. We need to check that $\tau$ is equivariant for the global $R^{*}$-actions on the spectral sequences so that $\tau$ induces a map on the localised spectral sequences.
By Proposition \ref{prop:ActionaBa}, the global action is induced by the map $(C_{B_{a}},B_{a})$ on the Spectral Sequence (\ref{ss:hyperhomologyII}).
Therefore, $R^*$-equivariance follows from the fact that for every $a \in R^{*}$ and for all $j = 0,1,2$, the diagrams
\[
\begin{tikzcd}
(O_{n-2,n-2}, C_{p-2}(n-2)) \arrow{r}{(i,\tau_{j})}
  & (O_{n,n}, C_{p}(n)) \\
(O_{n-2,n-2}, C_{p-2}(n-2)) \arrow{u}{(C_{B_{a}}, B_{a})} \arrow[swap]{r}{(i,\tau_{j})}
  & (O_{n,n}, C_{p}(n)) \arrow[swap]{u}{(C_{B_{a}}, B_{a})},
\end{tikzcd}
\]
commute, where $i : O_{n-2,n-2} \hookrightarrow O_{n,n}$ denotes the inclusion.
The point is that $B_a(e_i)=e_i$.
Therefore, $\tau$ induces a map of spectral sequences of $R^*$-modules
\[
\tau_{*} : \tilde{E} \rightarrow E
\]
where $\tilde{E} := E(n-2)[-2, 0]$ and $E := E(n)$.

Recall from Propositions \ref{prop:ActionBa2Phia} and \ref{prop:ActionDaPhia} that the local actions are globally induced by multiplication with $a^{-2}$ for $a\in R^*$.
Localising with respect to the last action, we obtain a map on the localised spectral sequences 
\[
{}_{m}\tau_{*} : {}_{m}\tilde{E} \rightarrow {}_{m}E.
\]
Note that for all $q < m/2$ and $2 \leq p \leq n$,
\[
{}_{m}\tilde{E}^{1}_{p,q} = {}_{m}E^{1}_{p,q}(n-2)[-2,0] = {}_{m}E^{1}_{p-2, q}(n-2) \cong H_{q}(O_{n-p, n-p}).
\]
The claim will then follow by induction on $r$ using the following lemma:

\begin{lemma} \label{lem:E1identity}
The map ${}_{m}\tau_{*} : {}_{m}\tilde{E}^{1}_{p,q} \rightarrow {}_{m}E^{1}_{p,q}$ is the identity for all $q < m/2$ and $2 \leq p \leq n$.
\end{lemma}

\begin{proof}
If we can show that for $j = 0, 1, 2$, the diagrams 
\begin{equation}
\begin{tikzcd}
H_{q}(O_{n-p,n-p}) \arrow[swap]{d}{=} \arrow{rr}{(\iota,(e_{1}, \dots, e_{p-2}))_{*}}
  & &H_{q}(O_{n-2,n-2},C_{p-2}(n-2)) \arrow{d}{(\iota, \tau_{j})_{*}}  \\
H_{q}(O_{n-p,n-p})  \arrow[swap]{rr}{(\iota,(e_{1}, \dots, e_{p}))_{*}}
  && H_{q}(O_{n,n}, C_{p}(n)),
\end{tikzcd}  \label{cd:identitydifferentials}
\end{equation}
commute, where the $\iota$'s denote inclusions, we will be done, as $\tau = \tau_{0} - \tau_{1} + \tau_{2}$. 

Again, we will prove these diagrams commute using Lemma \ref{lemma:homologycommute}. 
Specifically, consider diagram 
\[
(O_{n-p,n-p},\mathbb{Z}) \stackrel[(\iota, (e_{1}, \dots, e_{p}))]{(\iota, \tau_{j}(e_{1}, \dots, e_{p-2}))}\rightrightarrows(O_{n,n},C_{p}(n))
\]
Note that $\tau_{0}(e_{1}, \dots, e_{p-2}) = (e_{1}, e_{2}, e_{3}, \dots, e_{p})$, so that diagram (\ref{cd:identitydifferentials}) commutes in the case for $j = 0$ by functoriality of group homology. 
For $j = 1, 2$, we have $\tau_{1}(e_{1}, \dots, e_{p-2}) = (e_{1}, e_{2} - e_{1}, e_{3}, \dots, e_{p})$ and $\tau_{2}(e_{1}, \dots, e_{p-2}) = (e_{2}, e_{2} - e_{1}, e_{3}, \dots, e_{p})$. 
Define a matrix $A \in O_{n,n}(R)$ by 
\begin{align*}
&e_{1} \mapsto e_{1} \\
&e_{2} \mapsto e_{2} - e_{1} \\
&f_{1} \mapsto f_{1} + f_{2} \\
&f_{2} \mapsto f_{2} \\ 
&e_{j} \mapsto e_{j} \,\, \text{for all} \,\, 3 \leq j \leq n \\
 &f_{j} \mapsto f_{j} \,\, \text{for all} \,\, 3 \leq j \leq n.
\end{align*}
Similarly, define $B \in O_{n,n}$ by 
\begin{align*}
&e_{1} \mapsto e_{2} \\
&e_{2} \mapsto e_{2} - e_{1} \\
&f_{1} \mapsto f_{1} + f_{2} \\
&f_{2} \mapsto -f_{1} \\ 
&e_{j} \mapsto e_{j} \,\, \text{for all} \,\, 3 \leq j \leq n \\
 &f_{j} \mapsto f_{j} \,\, \text{for all} \,\, 3 \leq j \leq n.
\end{align*}
Then, $A(e_{1}, \dots e_{p}) = \tau_{1}(e_{1}, \dots, e_{p-2})$, $B(e_{1}, \dots e_{p}) = \tau_{2}(e_{1}, \dots, e_{p-2})$ and for every $M \in O_{n-p,n-p}$, $\iota(M) = A \iota(M)A^{-1} = B \iota(M)B^{-1}$. 
Thus, by Lemma \ref{lemma:homologycommute}, Diagram (\ref{cd:identitydifferentials}) commutes for every $j = 0, 1, 2$.
These diagrams still commute after localisation, but now the horizontal maps become the identification isomorphisms. 
\end{proof}
This proves the lemma, and thus Proposition \ref{prop:zerodifferentials}.
\end{proof}

\begin{theorem}
Let $R$ be a commutative local ring with infinite residue field such that $2 \in R^{*}$. Then, the natural homomorphism 
\[
H_{k}(O_{n,n}(R)) \longrightarrow H_{k}(O_{n+1,n+1}(R))
\]
is an isomorphism for $k \leq n-1$ and surjective for $k \leq n$. 
\end{theorem}

\begin{remark}
This improves Mirzaii's result \cite{mirzaii2004homology} by 1 and matches the analogous result for fields obtained by Sprehn-Wahl \cite{sprehn2020homological}. 
\end{remark}

\begin{proof}
Choose $m > 0$ sufficiently large so that we may apply Corollary \ref{corollary:mE2} when $q \leq n -1$.

Recall from Theorem \ref{thm:C*acyclic} that $H_{q}(C_{*}(n)) = 0$ for all $q \leq n-1$. Thus, from the Spectral Sequences (\ref{ss:firsthyperhomology}) and  (\ref{ss:localisedII}), Corollary \ref{corollary:mE2} and Proposition \ref{prop:zerodifferentials}, we deduce 
\begin{align*}
\coker \left( H_{q}(O_{n-1,n-1}) \xrightarrow{i_{*}} H_{q}(O_{n, n}))\right) &= {}_{m}E^{2}_{0,q} \\
&\cong {}_{m}E^{\infty}_{0,q} \\
&= 0
\end{align*}
for all $q \leq n-1$, and 
\begin{align*}
\ker \left( H_{q}(O_{n-1,n-1}) \xrightarrow{i_{*}} H_{q}(O_{n, n}))\right) &= {}_{m}E^{2}_{1,q} \\
&\cong {}_{m}E^{\infty}_{1,q} \\
&= 0
\end{align*}
for all $q \leq n-2$.

The theorem follows. 
\end{proof}

\section{Homological stability for $SO_{n,n}$} \label{sec:specialhomstability}

Recall we have two hyperhomology spectral sequences  
\begin{align}
E^{2}_{p,q}(n) &= H_{p}(SO_{n,n}, H_{q}(C_{*}(n))) \Rightarrow H_{p+q}(SO_{n,n}, C_{*}(n)) \label{ss:specialfirsthyperhomology} \\
E^{1}_{p,q}(n) &= H_{q}(SO_{n,n}, C_{p}(n)) \Rightarrow H_{p+q}(SO_{n,n}, C_{*}(n)).  \label{ss:specialsecondhyperhomology}
\end{align}
Moreover, recall that in Theorem \ref{thm:C*acyclic}, we proved $H_{q}(C_{*}(n)) = 0$ for every $q \leq n-1$. As we expect the homological stability range for $SO_{n,n}$ to be the same as for $O_{n,n}$, a reasonable proof strategy is to localise Spectral Sequence (\ref{ss:specialsecondhyperhomology}) in the same manner as we did for the $O_{n,n}$ and analyse the localised spectral sequences. The analysis will turn out to be very similar to the $O_{n,n}$ case, except for the situation when $p = n$, corresponding to the fact that the action of $SO_{n,n}$ on $\mathcal{IU}_{p}(R^{2n})$ is \textit{transitive} only for $p < n$, see Lemma \ref{lemma:specialtransitive}. But in the end, this will not prove to be too significant. 

Note that for all $n > 0$, we have short exact sequences 
\begin{align*}
1 \rightarrow SO_{n,n} \rightarrow O_{n,n} \rightarrow \mathbb{Z}_{2} \rightarrow 1,
\end{align*}
where the right arrow given is by the determinant map. Moreover, if we define $ST_{k} \leq T_{k}$ to be the subgroup of matrices in $T_{k}$ having determinant 1, the projection map $\rho: T_{k} \rightarrow O_{n-k, n-k}$ restricts to a map $\rho : ST_{k} \rightarrow SO_{n-k,n-k}$. Note that 
\[
\ker(\rho: T_{k} \rightarrow O_{n-k, n-k}) = \ker(\rho: ST_{k} \rightarrow SO_{n-k, n-k}),
\]
since, by inspection on the matrices in $T_{k}$, we deduce that $\det A = \det \rho(A)$ for every $A \in T_{k}$, and that both kernels consist precisely of those matrices that map to the identity matrix. This observation will turn out to be significant in the forthcoming analysis. Furthermore, note that $ST_{n} = T_{n}$. We use the conventions that $SO_{0,0} = 1$ and $ST_{0} = SO_{n,n}$. 

We obtain short exact sequences for every $0 \leq k \leq n$.  
\begin{align}
1 \rightarrow L_{k} \rightarrow ST_{k} \rightarrow SO_{n-k,n-k} \rightarrow 1. \label{s.e.s:speciallocalising}
\end{align}

\section{Local $R^*$-actions and transitivity}

Define a local $R^{*}$-action on short exact sequence (\ref{s.e.s:speciallocalising}) in exactly the same was as we did in Section \ref{subsubsec:local $R^*$-action}, namely we conjugate matrices in $ST_{k}$ with the matrix $D_{a,k} \in SO_{n,n}$. As $\ker(\rho: T_{k} \rightarrow O_{n-k, n-k}) = \ker(\rho: ST_{k} \rightarrow SO_{n-k, n-k})$, the \textit{exact} same reasoning as in Section \ref{subsubsec:local $R^*$-action} can be used to conclude that, \textit{after localistaion}, the homology of $ST_{k}$ and $SO_{n-k, n-k}$ coincide:
\begin{corollary} \label{corollary:specialstabiso}
The inclusion $SO_{n-k,n-k} \hookrightarrow T_{k}$ induces isomorphisms 
\[
H_{t}(SO_{n-k, n-k}) \xrightarrow{\cong} s_{m}^{-1}H_{t}(ST_{k})
\]
for all $t < m/2$.
\end{corollary}

Next, we study the transitivity of the $SO_{n,n}$ action on $\mathcal{IU}_{p}(R^{2n})$. 

\begin{lemma} \label{lemma:specialtransitive}
The action of $SO_{n,n}$ on $\mathcal{IU}_{p}(R^{2n})$ is transitive for all $p < n$.
\end{lemma}

\begin{proof}
Let $(u_{1},\dots, u_{p}), (v_{1},\dots, v_{p}) \in \mathcal{IU}_{p}(R^{2n})$. By Lemma \ref{lem:hyperbolicbasis}, we deduce there exists an $u_{1}^{\#}, \dots, u_{p}^{\#}$ such that $(u_{1}, u_{1}^{\#}, \dots, u_{p}, u_{p}^{\#})$ has Gram matrix $\psi_{2p}$, and which may be extended to hyperbolic basis $(u_{1}, u_{1}^{\#}, \dots, u_{p}, u_{p}^{\#}, x_{1}, x_{1}^{\#}, \dots, x_{n-p}, x_{n-p}^{\#})$. Similarly, there exists an $v_{1}^{\#}, \dots, v_{p}^{\#}$ such that $(v_{1}, v_{1}^{\#}, \dots, v_{p}, v_{p}^{\#})$ has Gram matrix $\psi_{2p}$, and which may be extended to hyperbolic basis $(v_{1}, v_{1}^{\#}, \dots, v_{p}, v_{p}^{\#}, y_{1}, y_{1}^{\#}, \dots, y_{n-p}, y_{n-p}^{\#})$.  

Let $B \in O_{n,n}$ be the matrix 
\[
B := (u_{1} \quad u_{1}^{\#} \quad \cdots \quad u_{p}\quad u_{p}^{\#}\quad x_{1} \quad x_{1}^{\#} \quad \cdots \quad x_{n-p} \quad x_{n-p}^{\#}), 
\] 
and let $C \in O_{n,n}$ be the matrix 
\[
C := (v_{1} \quad v_{1}^{\#} \quad \cdots \quad v_{p}\quad v_{p}^{\#}\quad y_{1} \quad y_{1}^{\#} \quad \cdots \quad y_{n-p} \quad y_{n-p}^{\#}).
\] 
If $\det B = \det C$, then $CB^{-1} \in SO_{n,n}$ and maps $(u_{1},\dots, u_{p})$ to  $(v_{1},\dots, v_{p})$. If $\det B \neq \det C$, then define $\hat{C} := CT$, where $T$ is the matrix that swaps $y_{n-p}$ and $y_{n-p}^{\#}$ in the columns of $C$. Then, as $\det T = -1$, it follows $\det B = \det \hat{C}$, and we are in the previous case. 
\end{proof}

Recall that Shapiro's Lemma gives an isomorphism 
\[
\bigoplus_{[x] \in S/G} (i_{x}, x)_{*} : \bigoplus_{[x] \in S/G} H_{*}(G_{x}, \mathbb{Z}) \xrightarrow{\cong} H_{*}(G, \mathbb{Z}[S])
\]
of homology groups, where the direct sum is over a set of representatives $x \in S$ of equivalence classes $[x] \in S/G$; the group $G_{x}$ is the \textit{stabiliser} of $G$ at $x \in S$; the homomorphism $i_{x} : G_{x} \subseteq G$ is the inclusion; and $x$ also denotes the homomorphism of abelian groups $\mathbb{Z} \rightarrow \mathbb{Z}[S]: 1 \mapsto x$.

We apply Shapiro's Lemma in the case $G = SO_{n,n}(R)$ and $S = \mathcal{IU}_{p}(R^{2n})$. 

Thus, by Lemma \ref{lemma:specialtransitive} we have identification isomorphisms given by Shapiro's Lemma for every, $ 0\leq p < n$ and $q \geq 0$
\begin{align}
H_{q}(ST_{p}) \xrightarrow{\cong} H_{q}(SO_{n,n}, C_{p}(n)).
\end{align}

For $p = n$, we claim that the action of $SO_{n,n}$ on $\mathcal{IU}_{n}(R^{2n})$ has two orbits:
\begin{proposition} \label{prop:specialorbits}
For $n \geq 1$, the action of $SO_{n,n}$ on $\mathcal{IU}_{n}(R^{2n})$ has orbits corresponding to $\mathbb{Z}_{2}$. 
\end{proposition}
\begin{proof} 
We know by Lemma \ref{lem:hyperbolicbasis} that the action of $O_{n,n}$ on $\mathcal{IU}_{n}(R^{2n})$ is transitive, so that we have an isomorphism of $O_{n,n}-$sets
\[
O_{n,n}/T_{n} \cong \mathcal{IU}_{n}(R^{2n}).
\]
Furthermore, note that $T_{n} = ST_{n} \leq SO_{n,n} \leq O_{n,n}$, so that we have a canonical surjection $O_{n,n}/T_{n} \rightarrow O_{n,n}/SO_{n,n}$ with fibre $SO_{n,n}/T_{n}$. Therefore, we have an isomorphism of $O_{n,n}$-sets $\mathcal{IU}_{n}(R^{2n})/SO_{n,n} \cong O_{n,n}/SO_{n,n}$. This gives us
\[
\left| \mathcal{IU}_{n}(R^{2n})/SO_{n,n} \right| = \left| O_{n,n}/SO_{n,n} \right| = \left| \mathbb{Z}_{2} \right|,
\]
where the last equality follows from the short exact sequence 
\begin{align*}
1 \rightarrow SO_{n,n} \rightarrow O_{n,n} \rightarrow  \mathbb{Z}_{2} \rightarrow 1.
\end{align*}
\end{proof}
Therefore, Shapiro's Lemma gives us an isomorphism 
\begin{align}
H_{q}(St(e_{1}, \dots, e_{n})) \oplus H_{q}(St(e_{1}, \dots, e_{n-1}, f_{n})) \xrightarrow{\cong} H_{q}(SO_{n,n}, C_{n}(n)).
\end{align}
where $St(e_{1}, \dots, e_{n-1}, f_{n})$ denotes the stabiliser of $(e_{1}, \dots, e_{n-1}, f_{n})$ in $SO_{n,n}$, and the identification map is given by Shapiro's Lemma. To ease notation, we define $\overline{T}_{n} := St(e_{1}, \dots, e_{n-1}, f_{n})$.

We will compute $\overline{T}_{n}$ and show that, \textit{after localisation}, all non-zero homology groups of $\overline{T}_{n}$ vanish. 

\subsection{Computation of  $\overline{T}_{n}$ and a local $R^{*}$-action}

We first compute $\overline{T}_{n}$. 

\begin{proposition}
Matrices $A \in \overline{T}_{n}$ are of the form 
\[
A =  
\begin{pmatrix}
1 & c^{1}_{1} & 0 & c^{1}_{2} & \cdots & 0 & c^{1}_{n-1} & c^{1}_{n} & 0 \\
0 & 1 & 0 & 0 & \cdots & 0 & 0 & 0 & 0 \\
0 & c^{2}_{1} & 1 & c^{2}_{2} & \cdots & 0 & c^{2}_{n-1} & c^{2}_{n} & 0 \\
0 & 0 & 0 & 1 & \cdots & 0 & 0 & 0 & 0 \\
\vdots & \vdots & \vdots & \vdots & \ddots & \vdots & \vdots & \vdots & \vdots\\
0 & c^{n-1}_{1} & 0 & c^{n-1}_{2} & \cdots & 1 & c^{n-1}_{n-1} & c^{n-1}_{n} & 0 \\
0 & 0 & 0 & 0 & \cdots & 0 & 1 & 0 & 0\\
0 & 0 & 0 & 0 & \cdots & 0 & 0 & 1 & 0\\
0 & c^{n}_{1} & 0 & c^{n}_{2} & \cdots & 0 & c^{n}_{n-1} & 0 & 1
\end{pmatrix}
\]
where  $c^{i}_{j} \in R $, subject to the conditions
\begin{align}            
c^{i}_{j} + c^{j}_{i}  = 0. \label{eqn:staboverline}           
\end{align}
\end{proposition}
\begin{proof}
Let $A \in \overline{T}_{n}$. Then, $Ae_{1} = e_{1}, \dots, Ae_{n-1} = e_{n-1}$ and $Af_{n} = f_{n}$. This gives the 1st, 3rd, $...., (2n-3)$rd and $2n$-th column of $A$. Moreover, for a fixed $1\leq i < n$ and for any $1 \leq j < n$, we have 
\[
\langle e_{i}, Af_{j} \rangle = \langle Ae_{i}, Af_{j} \rangle = \langle e_{i}, f_{j} \rangle = \delta_{ij}
\] 
and 
\[
\langle e_{i}, Ae_{n} \rangle = \langle Ae_{i}, Ae_{n} \rangle = \langle e_{i}, e_{n} \rangle = 0.
\]
This gives the 2nd, 4th, $\dots, (2n-2)$nd rows of $A$.

Furthermore, note that for $1 \leq j < n$, 
\[
\langle f_{n}, Af_{j} \rangle = \langle Af_{n}, Af_{j} \rangle = \langle f_{n}, f_{j} \rangle = 0,
\]
\[
\langle f_{n}, Ae_{j} \rangle = \langle Af_{n}, Ae_{j} \rangle = \langle f_{n}, e_{j} \rangle = 0
\]
and 
\[
\langle f_{n}, Ae_{n} \rangle = \langle Af_{n}, Ae_{n} \rangle = \langle f_{n}, e_{n} \rangle = 1.
\]
This gives the $(2n-1)$th row of $A$.
Filling in the remaining entries by constants to be determined, we have that 
\[
A =  
\begin{pmatrix}
1 & c^{1}_{1} & 0 & c^{1}_{2} & \cdots & 0 & c^{1}_{n-1} & c^{1}_{n} & 0 \\
0 & 1 & 0 & 0 & \cdots & 0 & 0 & 0 & 0 \\
0 & c^{2}_{1} & 1 & c^{2}_{2} & \cdots & 0 & c^{2}_{n-1} & c^{2}_{n} & 0 \\
0 & 0 & 0 & 1 & \cdots & 0 & 0 & 0 & 0 \\
\vdots & \vdots & \vdots & \vdots & \ddots & \vdots & \vdots & \vdots & \vdots\\
0 & c^{n-1}_{1} & 0 & c^{n-1}_{2} & \cdots & 1 & c^{n-1}_{n-1} & c^{n-1}_{n} & 0 \\
0 & 0 & 0 & 0 & \cdots & 0 & 1 & 0 & 0\\
0 & 0 & 0 & 0 & \cdots & 0 & 0 & 1 & 0\\
0 & c^{n}_{1} & 0 & c^{n}_{2} & \cdots & 0 & c^{n}_{n-1} & \Delta & 1
\end{pmatrix}
\]
where $c^{i}_{j}, \Delta \in R $.

We use the equation $^{t}A \psi_{2n} A = \psi_{2n}$ to determine the conditions on these variables. Specifically, we have that {\tiny
$$\begin{array}{l}
^{t}A \psi A 
 = \begin{pmatrix}
1 & 0 & 0 & 0 & \cdots & 0 & 0 & 0 & 0 \\
c^{1}_{1} & 1 & c^{2}_{1} & 0 & \cdots & c^{n-1}_{1} & 0 & 0 & c^{n}_{1} \\
0 & 0 & 1 & 0 & \cdots & 0 & 0 & 0 & 0 \\
c^{1}_{2} & 0 & c^{2}_{2} & 1 & \cdots & c^{n-1}_{2} & 0 & 0 & c^{n}_{2}  \\
\vdots & \vdots & \vdots & \vdots & \ddots & \vdots & \vdots & \vdots & \vdots \\
0 & 0 & 0 & 0 & \cdots & 1 & 0 & 0 & 0 \\
c^{1}_{n-1} & 0 & c^{2}_{n-1} & 0 & \cdots & c^{n-1}_{n-1} & 1 & 0 & c^{n}_{n-1} \\
c^{1}_{n} & 0 & c^{2}_{n} &0 & \cdots & c^{n-1}_{n} & 0 & 1 & \Delta \\
0 & 0 & 0 & 0 & \cdots & 0 & 0 & 0 & 1
\end{pmatrix}
\psi
\begin{pmatrix}
1 & c^{1}_{1} & 0 & c^{1}_{2} & \cdots & 0 & c^{1}_{n-1} & c^{1}_{n} & 0 \\
0 & 1 & 0 & 0 & \cdots & 0 & 0 & 0 & 0 \\
0 & c^{2}_{1} & 1 & c^{2}_{2} & \cdots & 0 & c^{2}_{n-1} & c^{2}_{n} & 0 \\
0 & 0 & 0 & 1 & \cdots & 0 & 0 & 0 & 0 \\
\vdots & \vdots & \vdots & \vdots & \ddots & \vdots & \vdots & \vdots & \vdots\\
0 & c^{n-1}_{1} & 0 & c^{n-1}_{2} & \cdots & 1 & c^{n-1}_{n-1} & c^{n-1}_{n} & 0 \\
0 & 0 & 0 & 0 & \cdots & 0 & 1 & 0 & 0\\
0 & 0 & 0 & 0 & \cdots & 0 & 0 & 1 & 0\\
0 & c^{n}_{1} & 0 & c^{n}_{2} & \cdots & 0 & c^{n}_{n-1} & \Delta & 1
\end{pmatrix} \\
\\
= \begin{pmatrix}
0 & 1 & 0 & 0 & \cdots & 0 & 0 & 0 & 0 \\
1 & c^{1}_{1} + c^{1}_{1} & 0 & c^{1}_{2} + c^{2}_{1} & \cdots & 0 & c^{1}_{n-1} + c^{n-1}_{1} & c^{1}_{n} + c^{n}_{1} & 0 \\
0 & 0 & 0 & 1 & \cdots & 0 & 0 & 0 & 0  \\
0 & c^{1}_{2} + c^{2}_{1} & 1 & c^{2}_{2} + c^{2}_{2} & \cdots & 0 & c^{2}_{n-1} + c^{n-1}_{2} & c^{2}_{n} + c^{n}_{2} & 0 \\
\vdots & \vdots & \vdots & \vdots & \ddots & \vdots & \vdots & \vdots & \vdots\\
0 & 0 & 0 & 0 & \cdots & 0 & 1 & 0 & 0\\
0 & c^{1}_{n-1} + c^{n-1}_{1} & 0 & c^{2}_{n-1} + c^{n-1}_{2} & \cdots & 1 & c^{n-1}_{n-1} + c^{n-1}_{n-1} & c^{n-1}_{n} + c^{n}_{n-1} & 0 \\
0 & c^{1}_{n} + c^{n}_{1} & 0 & c^{2}_{n} + c^{n}_{2} & \cdots & 0 & c^{n-1}_{n} + c^{n}_{n-1}  & 2\Delta & 1 \\
0 & 0 & 0 & 0 & \cdots & 0 & 0 & 1 & 0
\end{pmatrix} \\
\\
= \psi.
\end{array}$$
}
Thus, we conclude $\Delta = 0$ necessarily and we derive the required equations. 
\end{proof}

We now define a local $R^{*}$-action on $\overline{T}_{n}$. This will be slightly different from the local actions on $T_{k}$. We will show  that, \textit{after localisation}, the non-zero homology groups of $\overline{T}_{n}$ \textit{vanish}. Eventually, we will show the global action considered in subsection \ref{subsubsection:globalaction} induces this local action on $\overline{T}_{n}$. 

\vspace{1ex}


\begin{definition}[Local action]
Let $a \in R^{*}$. Define a $2n \times 2n$ matrix $\overline{D}_{a,n}$ by
\[
\overline{D}_{a,n} := 
\begin{pmatrix}
a \\
& a^{-1}\\
&&\ddots \\
&&&a \\
&&&& a^{-1} \\
&&&&&a^{-1} \\
&&&&&&a
\end{pmatrix}
= \left(\bigoplus_{1}^{n-1}
\begin{pmatrix}
a & 0 \\
0 & a^{-1}
\end{pmatrix}\right) \bigoplus
\begin{pmatrix}
a^{-1} & 0 \\
0 & a
\end{pmatrix}.
\]
Note that $\overline{D}_{a,n} \in SO_{n,n}(R)$. 
The {\em local action} of $R^*$ on $\overline{T}_{n}$ is the conjugation action of $\overline{D}_{a,n}$ on $\overline{T}_{n}$. 
\end{definition}
The local action preserves $\overline{T}_{n}$ because
\begin{align*}
\overline{D}_{a,n} 
&\begin{pmatrix}
1 & c^{1}_{1} & 0 & c^{1}_{2} & \cdots & 0 & c^{1}_{n-1} & c^{1}_{n} & 0 \\
0 & 1 & 0 & 0 & \cdots & 0 & 0 & 0 & 0 \\
0 & c^{2}_{1} & 1 & c^{2}_{2} & \cdots & 0 & c^{2}_{n-1} & c^{2}_{n} & 0 \\
0 & 0 & 0 & 1 & \cdots & 0 & 0 & 0 & 0 \\
\vdots & \vdots & \vdots & \vdots & \ddots & \vdots & \vdots & \vdots & \vdots\\
0 & c^{n-1}_{1} & 0 & c^{n-1}_{2} & \cdots & 1 & c^{n-1}_{n-1} & c^{n-1}_{n} & 0 \\
0 & 0 & 0 & 0 & \cdots & 0 & 1 & 0 & 0\\
0 & 0 & 0 & 0 & \cdots & 0 & 0 & 1 & 0\\
0 & c^{n}_{1} & 0 & c^{n}_{2} & \cdots & 0 & c^{n}_{n-1} & 0 & 1
\end{pmatrix}
\overline{D}_{a,n}^{-1} \\
&\\
= &\begin{pmatrix}
1 & a^{2}c^{1}_{1} & 0 & a^{2}c^{1}_{2} & \cdots & 0 & a^{2}c^{1}_{n-1} & a^{2}c^{1}_{n} & 0 \\
0 & 1 & 0 & 0 & \cdots & 0 & 0 & 0 & 0 \\
0 & a^{2}c^{2}_{1} & 1 & a^{2}c^{2}_{2} & \cdots & 0 & a^{2}c^{2}_{n-1} & a^{2}c^{2}_{n} & 0 \\
0 & 0 & 0 & 1 & \cdots & 0 & 0 & 0 & 0 \\
\vdots & \vdots & \vdots & \vdots & \ddots & \vdots & \vdots & \vdots & \vdots\\
0 & a^{2}c^{n-1}_{1} & 0 & a^{2}c^{n-1}_{2} & \cdots & 1 & a^{2}c^{n-1}_{n-1} & a^{2}c^{n-1}_{n} & 0 \\
0 & 0 & 0 & 0 & \cdots & 0 & 1 & 0 & 0\\
0 & 0 & 0 & 0 & \cdots & 0 & 0 & 1 & 0\\
0 & a^{2}c^{n}_{1}  & 0 & a^{2}c^{n}_{2}  & \cdots & 0 & a^{2}c^{n}_{n-1}  & 0 & 1
\end{pmatrix} \in \overline{T}_{n}.
\end{align*}

We show that localising with respect to the elements $s_{m}$ \textit{kills} the non-zero homology groups of $\overline{T}_{n}$ when $m$ is taken to infinity. This is used to make the identifications in Corollary \ref{corollary:speciallocalisedE1}.

\begin{lemma} \label{lemma:killoverline}
We have $s_{m}^{-1}H_{0}(\overline{T}_{n}) = \mathbb{Z} $ and for all $1\leq 2q < m$, $s_{m}^{-1}H_{q}(\overline{T}_{n}) = 0$.
\end{lemma}

\begin{proof}
We claim there is a short exact sequence of groups 
\begin{align}
1 \rightarrow (R^{n \choose 2} ,+) \rightarrow \overline{T}_{n} \rightarrow 1 \rightarrow 1. \label{ses:junkoverline}
\end{align}
The first arrow maps
\[
(c_{1}, \dots ) \mapsto A_{(c_{1}, \dots )}
\]
where $A_{(c_{1}, \dots )} \in \overline{T}_{n}$ is defined by conditions (\ref{eqn:staboverline}) (with some ordering specified beforehand). The second maps $A \in \overline{T}_{n}$ to its bottom right identity matrix. One may check that $\overline{T}_{n}$ is abelian, and these arrows define a short exact sequence of abelian groups. 

Furthermore, this short exact sequence of abelian groups is $R^{*}$-equivariant where $b \in R^{*}$ acts on $(R^{n \choose 2},+)$ via pointwise multiplication by $b^{2}$, the element $b\in R^*$ acts on $\overline{T}_{n}$ via conjugation by $\overline{D}_{b,n}$ and the action of $b$ on $1$ is taken to be trivial. 

By Proposition \ref{prop:magicallykill}, $s_{m}^{-1}H_{q}(\overline{T}_{n}) = 0$ for all $1\leq 2q < m$. The equality $s_{m}^{-1}H_{0}(\overline{T}_{n}) = \mathbb{Z}$ follows from fact that $R^{*}$ acts trivially on $H_{0}(\overline{T}_{n})$.
\end{proof}

\section{A global action on the $SO_{n,n}$ spectral sequence} \label{subsubsection:specialglobalaction}

As before, we want to realise these `local actions' as a `global action' on the spectral sequence 
\begin{align}
E^{1}_{p,q}(n) &= H_{q}(SO_{n,n}, C_{p}(n)) \Rightarrow H_{p+q}(SO_{n,n}, C_{*}(n))  \label{ss:specialhyperhomologyII}.
\end{align}
Again, we do this by defining an action on the associated exact couple with abutment. Specifically, the spectral sequence 
\[
E^{1}_{p,q} = H_{q}(SO_{n,n}, C_{p}(n)) \Rightarrow H_{p+q}(SO_{n,n}, C_{*}(n)) 
\]
may be obtained from the exact couple with abutment 
\begin{equation}
\begin{tikzcd}
\bigoplus_{p,q} E^{1}_{p,q} \arrow{r}{k} 
  & \bigoplus_{p,q} D^{1}_{p,q} \arrow{d}{i} \arrow{r}{\sigma} 
    & \bigoplus_{p+q} A_{p+q}  \\
    & \bigoplus_{p,q} D^{1}_{p,q} \arrow{lu}{j} \arrow{ru}{\sigma}
\end{tikzcd} \label{exactcouple:specialhyperhomologyII}
\end{equation}
with $E^{1}_{p,q} = H_{p+q}(SO_{n,n}, C_{\leq p}(n)/C_{\leq p-1}(n))$; $D^{1}_{p,q} = H_{p+q}(SO_{n,n}, C_{\leq p}(n))$; $A_{p+q} = H_{p+q}(SO_{n,n},C_{*}(n))$; the maps $i, j ,k$ being the maps of the long exact sequence of homology groups associated to the short exact sequence of complexes 
\[
0 \rightarrow C_{\leq p-1}(n) \rightarrow C_{\leq p}(n) \rightarrow C_{\leq p}(n)/C_{\leq p-1}(n) \rightarrow 0,
\]
and $\sigma$ is induced by the inclusion.

For $a \in R^{*}$, we define the global action on Spectral Sequence (\ref{ss:specialhyperhomologyII}) to be the action induced by the map $(C_{B_{a^{-2}}},B_{a^{-2}})$ on exact couple (\ref{exactcouple:specialhyperhomologyII}), where $C_{B_{a^{-2}}}$ denotes conjugation by the matrix $B_{a^{-2}}$ of Section \ref{subsubsection:globalaction}, and $B_{a^{-2}}$ also refers to multiplication by this matrix. 

As $D_{a,k} \in SO_{n,n}$ for every $0 \leq k \leq n$, the proof of Proposition \ref{prop:ActionBa2Phia} may be used to prove the following. 
\begin{proposition}
Let $k,q \geq 0$ and $n\geq 1$. Then, for all $a \in R^{*}$, the following diagram commutes:
\[
\begin{tikzcd}
H_{q}(SO_{n,n}, C_{k}(n)) \arrow{rr}{(C_{B_{a^{-2}}},B_{a^{-2}})_{*}}
  && H_{q}(SO_{n,n}, C_{k}(n)) \\
H_{q}(SO_{n,n},C_{k}(n)) \arrow{u}{id} \arrow[swap]{rr}{(id, \phi_{a})_{*}}
  && H_{q}(SO_{n,n},C_{k}(n)) \arrow[swap]{u}{id},
\end{tikzcd}
\]
where for $a \in R^{*}$, the map 
\[
(id, \phi_{a}): (SO_{n,n}, C_{k}(n)) \rightarrow (SO_{n,n}, C_{k}(n))
\]
is defined to be the identity on $SO_{n,n}$ and on basis elements of $C_{k}(n)$ as 
\[
\phi_{a} : (v_{1}, \dots , v_{k}) \mapsto (a^{-1}v_{1}, \dots, a^{-1}v_{k}).
\]
\end{proposition}

For the next proposition, we need to treat the case $k = n$ separately: 

\begin{proposition}
Let $q \geq 0$ and $0 \leq k < n$. Then, for all $a \in R^{*}$, the following diagram commutes:
\[
\begin{tikzcd}
H_{q}(SO_{n,n}, C_{k}(n)) \arrow{r}{(id,\phi_{a})_{*}}
  & H_{q}(SO_{n,n}, C_{k}(n)) \\
H_{q}(ST_{k},\mathbb{Z}) \arrow{u}{(i,(e_{1}, \dots, e_{k}))_{*}}[swap]{\cong} \arrow[swap]{r}{C_{D_{a,k}}}
  & H_{q}(ST_{k},\mathbb{Z}) \arrow[swap]{u}{(i,(e_{1}, \dots, e_{k}))_{*}}[swap]{\cong},
\end{tikzcd}
\]
where the vertical arrows are the isomorphisms given by Shapiro's Lemma and the map $C_{D_{a,k}}$ denotes the map induced by conjugation with the element $D_{a,k}$ on the stabiliser $ST_{k}$.

Moreover, for $q \geq 0$ and $k = n$, the following diagram commutes for all $a \in R^{*}$:
\[
\begin{tikzcd}
H_{q}(SO_{n,n}, C_{n}(n)) \arrow{rr}{(id,\phi_{a})_{*}}
  && H_{q}(SO_{n,n}, C_{n}(n)) \\
H_{q}(T_{n},\mathbb{Z})\oplus H_{q}(\overline{T}_{n},\mathbb{Z}) \arrow{u}{(i,(e_{1}, \dots, e_{n}))_{*} \oplus (i,(e_{1}, \dots, f_{n}))_{*} }[swap]{\cong} \arrow[swap]{rr}{C_{D_{a,n}} \oplus C_{\overline{D}_{a,n}} }
  && H_{q}(T_{n},\mathbb{Z})\oplus H_{q}(\overline{T}_{n},\mathbb{Z})  \arrow[swap]{u}{(i,(e_{1}, \dots, e_{n}))_{*} \oplus (i,(e_{1}, \dots, f_{n}))_{*}}[swap]{\cong},
\end{tikzcd}
\]
where the vertical arrows are the isomorphisms given by Shapiro's Lemma and the map $C_{D_{a,n}} \oplus C_{\overline{D}_{a,n}} $ denotes the map induced by conjugation with the element $D_{a,n}$ on the stabiliser $T_{n}$ sum with the map induced by conjugation with the element $\overline{D}_{a,n}$ on the stabiliser $\overline{T}_{n}$.
\end{proposition}
\begin{proof}
The proof of the first half of the proposition is exactly the same as the $O_{n,n}$ case, since $D_{a,k} \in SO_{n,n}$. See proposition \ref{prop:ActionDaPhia}.
The proof that the first component commutes is exactly the same as the $O_{n,n}$ case, since $D_{a,n} \in SO_{n,n}$.
For the second component, consider the diagram 
\[
(\overline{T}_{n},\mathbb{Z}) \stackrel[(f_{2}, \varphi_{2})]{(f_{1}, \varphi_{1})}\rightrightarrows(SO_{n,n},C_{n}(n))
\]
where $(f_{1}, \varphi_{1}) := (i, (a^{-1}e_{1}, \dots, a^{-1}f_{n}))$ and $(f_{2}, \varphi_{2}) := (iC_{\overline{D}_{a,n}}, (e_{1}, \dots, f_{n}))$,
and $i:\overline{T}_n \to SO_{n,n}$ is the natural inclusion of groups.
Let $\kappa = \overline{D}_{a,n} \in SO_{n,n}$. 
Then, for every $A \in \overline{T}_{n}$, 
$$
f_{2}(A) = iC_{\overline{D}_{a,n}}(A) 
= \overline{D}_{a,n}A \overline{D}_{a,n}^{-1} 
= \overline{D}_{a,n}i(A) \overline{D}_{a,n}^{-1} 
= \kappa f_{1}(A) \kappa^{-1}
$$
and
\[
(e_{1}, \dots, f_{n}) = \overline{D}_{a,n}(a^{-1}e_{1}, \dots, a^{-1}f_{n}) = \kappa (a^{-1}e_{1}, \dots, a^{-1}f_{n}).
\]
By Lemma \ref{lemma:homologycommute}, the diagram commutes.  
\end{proof}

Thus, we have shown that there exists an $R^{*}$-action on the spectral sequence 
\begin{align*}
E^{1}_{p,q}(n) &= H_{q}(SO_{n,n}, C_{p}(n)) \Rightarrow H_{p+q}(SO_{n,n}, C_{*}(n))
\end{align*}
which induces the desired local actions considered previously. 
Putting everything together, we obtain the following.

\begin{corollary} \label{corollary:speciallocalisedE1}
For every $m \geq 1$ and every $q < m/2$, the localised spectral sequence 
\begin{align}
{}_{m}E^{1}_{p,q}(n) = s_{m}^{-1}E^{1}_{p,q}(n) \Rightarrow s_{m}^{-1}H_{p+q}(SO_{n,n}, C_{*}(n)) \label{ss:speciallocalised}
\end{align}
has $_{m}E^{1}_{p,q}$ terms 
\[
{}_{m}E^{1}_{p,q}(n) = s_{m}^{-1}H_{q}(SO_{n,n}, C_{p}(n)) \cong 
\begin{cases}
H_{q}(SO_{n-p, n-p}) & 0 \leq p < n \\
\mathbb{Z}[\mathbb{Z}_{2}] & p = n, q =0 \\
0 & p =n , q > 0.
\end{cases}
\] 
\end{corollary}
\qed

Our next task is to compute the localised $d^{1}$ differentials $d^{1}:  \,\,_{m}E^{1}_{p,q} \rightarrow \,\, _{m}E^{1}_{p-1,q}$. 
\section{Computation of the localised $d^{1}$ differentials, and proof of homological stability}
\begin{proposition} \label{prop:speciallocalisedd1}
For all $q < m/2$ and $0 \leq p < n$, 
the homomorphism $d^{1}_{p,q}$ is 
\[
d^1_{p,q}=
\begin{cases}
0, & p \,\, \text{even} \\
i_{*}, & p \,\, \text{odd},
\end{cases}
\]
where $i : SO_{n-p, n-p} \hookrightarrow SO_{n-p+1, n-p+1}$ denotes the inclusion. 
For $p = n$, the homomorphism $d^{1}_{n,q}$ is 0 if $q > 0$ or if $n$ is even; and for $q = 0$ and $n$ odd, $d^{1}_{n,0}$ is the augmentation map $\varepsilon: \mathbb{Z}[\mathbb{Z}_{2}] \rightarrow \mathbb{Z}$.
\end{proposition}
\begin{proof}
For all $p < n$, we want to show that the following diagram commutes:
\begin{equation}
\begin{tikzcd}
H_{q}(SO_{n-p,n-p}) \arrow[swap]{d}{i_{*}} \arrow{rrr}{(\iota,(e_{1}, \dots, e_{p}))_{*}}
  &&& H_{q}(SO_{n,n},C_{p}(n)) \arrow{d}{(d_{i})_{*}}  \\
H_{q}(SO_{n-p+1,n-p+1})  \arrow[swap]{rrr}{(\iota,(e_{1}, \dots, e_{p-1}))_{*}}
  &&& H_{q}(SO_{n,n}, C_{p-1}(n)),
\end{tikzcd}  \label{cd:computedifferentials}
\end{equation}
where $\iota : SO_{n-p, n-p} \hookrightarrow SO_{n,n}$ denotes the inclusion map; $(e_{1}, \dots, e_{p}) : 1 \mapsto (e_{1}, \dots, e_{p})$ and recall that $d_{i}(v_{1}, \dots, v_{p}) = (v_{1},\dots, \hat{v}_{i}, \dots, v_{p} )$.

The same proof as in Proposition \ref{prop:computelocalisedd1} will work, so long as we can show $\sigma$ has determinant 1. 
Note that permutation matrix $\sigma$ will be of the form 
\[
\sigma = (e_{1} \quad f_{1} \quad \cdots \quad e_{i-1}\quad f_{i-1}\quad e_{p} \quad f_{p} \quad e_{i} \quad f_{i} \cdots \quad e_{p-1} \quad f_{p-1}).
\]
From this, we may write $id = \sigma \cdot T_{1}\cdot T_{2}\cdots T_{2(p-i)}$, where the matrices $T_{i}$ are the elementary matrices needed to swap columns in $\sigma$ to transform it into the identity matrix. Note that each of these matrices has determinant $-1$, and there are an even number of such matrices.    

Thus, we deduce $1 = \det(id) = \det(\sigma \cdot T_{1}\cdot T_{2}\cdots T_{2(p-i)}) = (-1)^{2(p-i)}\det(\sigma) = \det(\sigma)$.

For $p = n$, it suffices to show that the following diagram commutes:
\[
\begin{tikzcd}
\mathbb{Z}[\mathbb{Z}_{2}] \arrow[swap]{d}{\varepsilon} \arrow{rrrrr}{1 \mapsto (e_{1}, \dots, e_{n}), \sigma \mapsto (e_{1}, \dots, f_{n}) }
  &&&&& C_{n}(n)_{SO_{n,n}} \arrow{d}{d_{i}}  \\
\mathbb{Z} \arrow[swap]{rrrrr}{1 \mapsto (e_{1}, \dots, e_{n-1})}
  &&&&& C_{n-1}(n)_{SO_{n,n}}.
\end{tikzcd}  
\]
For $i = n$, it is easy to see by inspection that the diagram commutes. 

For $1 \leq i < n$, commutativity follows from the fact that $SO_{n,n}$ acts transitively on $\mathcal{IU}_{n-1}(R^{2n})$.


These diagrams still commutes after localisation, but now the horizontal arrows become the identification isomorphisms. 
\end{proof}

We need to prove the following:
\begin{proposition} \label{prop:specialzerodifferentials}
The differentials $d^{r}_{p,q}$ in Spectral Sequence (\ref{ss:speciallocalised}) are zero for $r \geq 2$ and $q < m/2$, $p \leq n$. Hence, for all $q < m/2$ and $p \leq n$, ${}_{m}E^{2}_{p,q} \cong {}_{m}E^{\infty}_{p,q}$.
\end{proposition}
\begin{proof}

For $n = 0,1$, the spectral sequence under consideration is located in columns 0 and 1. Therefore, the differentials $d^{r}$ for $r \geq 2$ are zero by dimension arguments. 

For $n \geq 2$, consider the homomorphism of complexes of $SO_{n-2,n-2}$-modules
\[
\tau : C_{*}(n-2)[-2] \rightarrow C_{*}(n).
\]
as defined in Proposition \ref{prop:zerodifferentials}. Note that the diagram 
\[
\begin{tikzcd}
(SO_{n-2,n-2}, C_{p-2}(n-2)) \arrow{r}{(i,\tau_{j})}
  & (SO_{n,n}, C_{p}(n)) \\
(SO_{n-2,n-2}, C_{p-2}(n-2)) \arrow{u}{(C_{B_{a}}, B_{a})} \arrow[swap]{r}{(i,\tau_{j})}
  & (SO_{n,n}, C_{p}(n)) \arrow[swap]{u}{(C_{B_{a}}, B_{a})},
\end{tikzcd}
\]
still commutes, so that we have an induced map on localised spectral sequences 
\[
{}_{m}\tau_{*} : {}_{m}\tilde{E} \rightarrow {}_{m}E.
\] 
The claim would then follow by induction on $r$ using the following lemma:
\begin{lemma}
The map ${}_{m}\tau_{*} : {}_{m}\tilde{E}^{1}_{p,q} \rightarrow {}_{m}E^{1}_{p,q}$ is the identity for all $q < m/2$ and $2 \leq p \leq n$.
\end{lemma}
\begin{proof}
For $2 \leq p < n$, the same proof as in Lemma \ref{lem:E1identity} works, as the matrices $A$ and $B$ in Lemma \ref{lem:E1identity} have determinant 1. Thus, we only need to consider the case $p = n$. 

It suffices to show that for $j = 0, 1, 2$, the following diagram commutes:
\[
\begin{tikzcd}
\mathbb{Z}[\mathbb{Z}_{2}] \arrow[swap]{d}{=} \arrow{rrrrr}{1 \mapsto (e_{1}, \dots, e_{n-2}), \sigma \mapsto (e_{1}, \dots, f_{n-2}) }
  &&&&& C_{n-2}(n-2)_{SO_{n-2,n-2}} \arrow{d}{\tau_{j}}  \\
\mathbb{Z}[\mathbb{Z}_{2}] \arrow[swap]{rrrrr}{1 \mapsto (e_{1}, \dots, e_{n}), \sigma \mapsto (e_{1}, \dots, f_{n}) }
  &&&&& C_{n}(n)_{SO_{n,n}}.
\end{tikzcd}
\]
For $j = 0$, the diagram is easily seen to commute by inspection. 

For $j = 1$, we have that 
\begin{align*}
A(e_{1}, \dots, e_{n}) &= (e_{1}, e_{2} - e_{1}, e_{3}, \dots, e_{n}) \\
A(e_{1}, \dots, f_{n}) &= (e_{1}, e_{2} - e_{1}, e_{3}, \dots, f_{n}),
\end{align*}
where $A \in SO_{n,n}$ is the matrix $A$ in the proof of Lemma \ref{lem:E1identity}. 

Similarly, for $j = 2$, 
we have that 
\begin{align*}
B(e_{1}, \dots, e_{n}) &= (e_{2}, e_{2} - e_{1}, e_{3}, \dots, e_{n}) \\
B(e_{1}, \dots, f_{n}) &= (e_{2}, e_{2} - e_{1}, e_{3}, \dots, f_{n}),
\end{align*}
where $B \in SO_{n,n}$ is the matrix $B$ in the proof of Lemma \ref{lem:E1identity}. 

Thus, the diagrams commute. These diagrams still commute after localisation, but now the horizontal maps become the identification isomorphisms. 
\end{proof}
This proves the lemma, and thus Proposition \ref{prop:specialzerodifferentials}.
\end{proof}

\begin{theorem}
Let $R$ be a commutative local ring with infinite residue field such that $2 \in R^{*}$. Then, the natural homomorphism 
\[
H_{k}(SO_{n,n}(R)) \longrightarrow H_{k}(SO_{n+1,n+1}(R))
\]
is an isomorphism for $k \leq n-1$ and surjective for $k \leq n$. 
\end{theorem}
\begin{remark}
This is the first known homological stability result for $SO_{n,n}$ over a local ring and generalises the result obtained by Essert \cite{essert2013homological} for infinite fields. 
\end{remark}
\begin{proof}
Choose $m > 0$ sufficiently large. 
We have a Spectral Sequence (\ref{ss:speciallocalised}) with $E^{1}$-terms given by Corollary \ref{corollary:speciallocalisedE1} and $d^{1}_{p,q}$ was computed for all $q < m/2$ in Proposition  \ref{prop:speciallocalisedd1}. From Theorem \ref{thm:C*acyclic}, Spectral Sequences (\ref{ss:specialfirsthyperhomology}) and (\ref{ss:speciallocalised}) and Proposition \ref{prop:specialzerodifferentials}, we deduce ${}_{m}E^{2}_{p,q} = {}_{m}E^{\infty}_{p,q}$ for all $p+q \leq n-1$ and $q < m/2$. The theorem follows. 
\end{proof}

\section{Homological stability for $EO_{n,n}$ and $\Spin_{n,n}$} \label{sec:elementaryhomstability}
We define $EO_{n,n}$ as follows. 
\begin{definition}
Define $EO_{n,n}$ to be the image of the map 
\[
EO_{n,n} := \Ima(\pi: \Spin_{n,n} \longrightarrow SO_{n,n}),
\]
where $\pi: \Spin_{n,n} \longrightarrow SO_{n,n}$ is the canonical map from the Spin group into the special orthogonal group. 
\end{definition}

From this definition, we see that $EO_{n,n}$ sits inside short exact sequence 
\[
1 \rightarrow \mathbb{Z}_{2} \rightarrow \Spin_{n,n} \xrightarrow{\pi} EO_{n,n} \rightarrow 1.
\]

We refer the reader to the appendix for more information about the Spin group and the exact sequence above. 

In this section, we will study the homological stability of $EO_{n,n}$. We will then apply the relative Hochschild-Serre Spectral Sequence to the above short exact sequence. This will then give us a homological stability result for $\Spin_{n,n}$. 

\begin{remark}
The given definition of the elementary group $EO_{n,n}$ coincides with the one presented in the introduction as stated in \cite[Theorem 9.2.8]{hahn1989classical}.
\end{remark}

\begin{remark}
Note that \cite[Theorem 9.2.8]{hahn1989classical} as stated is true for $n \geq 2$. For $n = 1$, we use the convention that $EO_{1,1}(R) = R^{*2}$, so that the above short exact sequence is still true. 
\end{remark}

To prove homological stability of $EO_{n,n}$, we will study the hyperhomology spectral sequences  
\begin{align}
E^{2}_{p,q}(n) &= H_{p}(EO_{n,n}, H_{q}(C_{*}(n))) \Rightarrow H_{p+q}(EO_{n,n}, C_{*}(n)) \label{ss:elementaryfirsthyperhomology} \\
E^{1}_{p,q}(n) &= H_{q}(EO_{n,n}, C_{p}(n)) \Rightarrow H_{p+q}(EO_{n,n}, C_{*}(n)).  \label{ss:elementarysecondhyperhomology}
\end{align}
As the action of $EO_{n,n}$ on $\mathcal{IU}_{p}(R^{2n})$ is \textit{transitive} only for $p < n$, see Lemma \ref{lemma:elementarytransitive},  it is reasonable to expect that the analysis for the $EO_{n,n}$ case should be similar to the $SO_{n,n}$ case. This is indeed what happens.

By Theorem \ref{theorem:shortexactsequences}, $EO_{n,n}$ also sits inside the short exact sequence 
\[
1 \rightarrow EO_{n,n} \rightarrow SO_{n,n} \xrightarrow{\theta} R^{*}/R^{*2} \rightarrow 1,
\]
where the first arrow is the inclusion and the second arrow $\theta$ is the \textit{spinor norm} (this short exact sequence is also true for $n=1$, given our convention $EO_{1,1}(R) = R^{*2}$). We refer the reader to the Appendix \ref{appendix:spinornorm} for more details about the spinor norm. See also \cite{scharlau2012quadratic} and \cite{hahn1989classical} as additional references. From this short exact sequence, we have the inclusion $[SO_{n,n}, SO_{n,n}] \subseteq EO_{n,n}$, where $[SO_{n,n}, SO_{n,n}]$ denotes the commutator subgroup of $SO_{n,n}$. Therefore, to prove a matrix is in $EO_{n,n}$, it will be sufficient to prove it is in $[SO_{n,n}, SO_{n,n}] $. This will be very convenient for us. 

\section{Transitivity and local $R^{*}$-actions}
We want to prove that the canonical action of $EO_{n,n}$ on $\mathcal{IU}_{p}(R^{2n})$ is transitive for all $p < n$. 
\begin{lemma} \label{lemma:elementarytransitive}
The action of $EO_{n,n}$ on $\mathcal{IU}_{p}(R^{2n})$ is transitive for all $p < n$.
\end{lemma}
\begin{proof}
Let $(v_{1}, \dots, v_{p}) \in \mathcal{IU}_{p}(R^{2n})$. It suffices to show that there exists an $A \in EO_{n,n}$ such that $A(e_{1}, \dots, e_{p}) = (v_{1}, \dots, v_{p})$. 

We know by Lemma \ref{lemma:specialtransitive} that the action of $SO_{n,n}$ is transitive for all $p < n$. Therefore, there exists a $B \in SO_{n,n}$ such that $B(e_{1}, \dots, e_{p}) = (v_{1}, \dots, v_{p})$. 

Furthermore, note that we have surjections $ST_{p} = \text{Stab}_{SO_{n,n}}(e_{1},\dots ,e_{p}) \twoheadrightarrow SO_{n-p, n-p}$ and $SO_{n-p,n-p} \twoheadrightarrow R^{*}/R^{*2}$. 

Therefore, we deduce that there exists a $C \in ST_{p} \leq SO_{n,n}$ such that $\theta(BC) = \theta(B)\theta(C) = 1$ and $BC(e_{1}, \dots, e_{p}) = B(e_{1}, \dots, e_{p}) = (v_{1}, \dots, v_{p})$. 

As $\theta(BC) = 1$, we have that $BC \in \ker(\theta) = EO_{n,n}$. We may thus set $A := BC$.
\end{proof}

Define $ET_{k} := \text{Stab}_{EO_{n,n}}(e_{1}, \dots, e_{k})$. Note that $ET_{k}$ is precisely $\ker(ST_{k} \xrightarrow{\theta} R^{*}/R^{*2})$. This gives us the following diagram with exact rows:
\[
\begin{tikzcd}
1 \arrow{r}
& ET_{k} \arrow[dashed]{d} \arrow{r}
& ST_{k} \arrow[two heads]{d}{\rho} \arrow{r}{\theta}
& R^{*}/R^{*2} \arrow{d}{=} \arrow{r}
& 1 \\
1 \arrow{r}
&EO_{n-k,n-k}  \arrow{r}  
& SO_{n-k,n-k}  \arrow{r}{\theta}
& R^{*}/R^{*2} \arrow{r}
&1
\end{tikzcd},
\]
where the existence of the dashed arrow for all $n \geq 2$ will follow if we can show that the right square commutes (the map trivially exists for $n=1$). Let us prove this. 
\begin{proposition}
For every $n \geq 2$, the square 
\[
\begin{tikzcd}
ST_{k} \arrow[two heads]{d} \arrow{r}{\theta}
&R^{*}/R^{*2} \arrow{d}{=}  \\
SO_{n-k,n-k}  \arrow{r}{\theta} 
& R^{*}/R^{*2}
\end{tikzcd}
\]
commutes.
\end{proposition}
\begin{proof}
Let $n \geq 2$. Recall that $[SO_{n,n}, SO_{n,n}] \subseteq EO_{n,n}$. 
Let $A \in ST_{k}$ such that $\rho(A) = B$. Want to show $\theta(A) = \theta 
\left(\begin{pmatrix}
1_{2k} & \\
& B
\end{pmatrix}\right)
$.
Equivalently, want to show $\theta \left( 
\begin{pmatrix}
1_{2k} & \\
& B^{-1}
\end{pmatrix}
A
 \right) = 1$.
Therefore, we want to show that $\begin{pmatrix}
1_{2k} & \\
& B^{-1}
\end{pmatrix}
A \in [SO_{n,n}, SO_{n,n}] \subseteq EO_{n,n}$. 

Note that $\begin{pmatrix}
1_{2k} & \\
& B^{-1}
\end{pmatrix}
A \in L_{k}$, therefore it suffices to prove that $L_{k} \subseteq [SO_{n,n}, SO_{n,n}]$.

The inclusion $L_{k} \hookrightarrow SO_{n,n}$ induces a map on homology $H_{q}(L_{k}) \rightarrow H_{q}(SO_{n,n})$. We claim this is the zero map for every $q \geq 1$.

Recall from Lemma \ref{lemma:kill} that $s_{m}^{-1}H_{q}(L_{k}) = 0$ for every $1 \leq 2q < m$ and note that $s_{m}^{-1}H_{q}(SO_{n,n}) = H_{q}(SO_{n,n})$, as the $R^{*}$-action defining this localization is trivial on $H_{q}(SO_{n,n})$. 

Taking $m$ sufficiently large, we obtain for every $q \geq 1$ commutative diagrams 
\[
\begin{tikzcd}
H_{q}(L_{k}) \arrow[two heads]{d} \arrow{r}
&H_{q}(SO_{n,n}) \arrow[two heads]{d}{=}  \\
s_{m}^{-1}H_{q}(L_{k})  = 0  \arrow{r}
& s_{m}^{-1}H_{q}(SO_{n,n}) = H_{q}(SO_{n,n})
\end{tikzcd}.
\]
We therefore deduce that $H_{q}(L_{k}) \rightarrow H_{q}(SO_{n,n})$ is the zero map for every $q \geq 1$.
In particular, as $H_{1}$ corresponds to taking abelianization, we have that the diagram  
\[
\begin{tikzcd}
L_{k} \arrow[two heads]{d} \arrow[hook]{r}
&SO_{n,n} \arrow[two heads]{d} \\
L_{k}/\text{[}L_{k}, L_{k}\text{]}  \arrow{r}{0} 
& SO_{n,n}/\text{[}SO_{n,n}, SO_{n,n}\text{]} 
\end{tikzcd}
\]
commutes.
Therefore, we conclude that the inclusion $L_{k} \hookrightarrow SO_{n,n}$ factors through $[SO_{n,n}, SO_{n,n}]$. 
\end{proof}
Thus, the projection map $\rho: ST_{k} \twoheadrightarrow SO_{n-k, n-k}$ induces a map $\rho: ET_{k} \twoheadrightarrow EO_{n-k, n-k}$. Moreover, the above proof shows that $ET_{n} = ST_{n} =T_{n}$ and $L_{k} = \ker\left(\rho: ET_{k} \twoheadrightarrow EO_{n-k, n-k}\right)$, so that we have short exact sequence 
\[
1 \rightarrow L_{k} \rightarrow ET_{k} \rightarrow EO_{n-k,n-k} \rightarrow 1.
\]
The associated Hochschild-Serre Spectral Sequence is 
\[
E^{2}_{p,q} = H_{p}(EO_{n-k,n-k};H_{q}(L_{k})) \Rightarrow H_{p+q}(ET_{k}).
\]
Knowing that $ET_{n} = ST_{n} = T_{n}$ allows us to prove the following proposition:
\begin{proposition} \label{prop:elementaryorbits}
For $n \geq 1$, the action of $EO_{n,n}$ on $\mathcal{IU}_{n}(R^{2n})$ has orbits corresponding to $R^{*}/R^{*2} \times \mathbb{Z}_{2}$. 
\end{proposition}
\begin{proof} 
We know by Lemma \ref{lem:hyperbolicbasis} that the action of $O_{n,n}$ on $\mathcal{IU}_{n}(R^{2n})$ is transitive, so that we have an isomorphism of $O_{n,n}-$sets
\[
O_{n,n}/T_{n} \cong \mathcal{IU}_{n}(R^{2n}).
\]
Furthermore, note that $T_{n} = ET_{n} \leq EO_{n,n} \leq O_{n,n}$, so that we have a canonical surjection $O_{n,n}/T_{n} \rightarrow O_{n,n}/EO_{n,n}$ with fibre $EO_{n,n}/T_{n}$. Therefore, we have an isomorphism of $O_{n,n}$-sets $\mathcal{IU}_{n}(R^{2n})/EO_{n,n} \cong O_{n,n}/EO_{n,n}$. This gives us
\[
\left| \mathcal{IU}_{n}(R^{2n})/EO_{n,n} \right| = \left| O_{n,n}/EO_{n,n} \right| = \left| R^{*}/R^{*2} \times \mathbb{Z}_{2} \right|,
\]
where the last equality follows from the short exact sequence 
\begin{align}
1 \rightarrow EO_{n,n} \rightarrow O_{n,n} \rightarrow  R^{*}/R^{*2} \times \mathbb{Z}_{2} \rightarrow 1,\label{ses:EO_{n,n}O_{n,n}}
\end{align}
see Theorem \ref{theorem:shortexactsequences}.
\end{proof}

\subsection{The local $R^{*}$-action}
Note that for every $a \in R^{*}$, 
\[
\theta(D_{a^{2}}) = \theta \left(
\begin{pmatrix}
a^{2} & 0 \\
0 & a^{-2}
\end{pmatrix} 
\right) = 1,
\]
so that 
\[
D_{a^{2},k} := 
\begin{pmatrix}
D_{a^{2}} \\
& \ddots \\
&&D_{a^{2}} \\
&&&1_{2n-2k}
\end{pmatrix} \in EO_{n,n}.
\]
We will define the local action of $R^{*}$ on $ET_{k}$ to be the conjugation action of $D_{a^{2},k}$ on $ET_{k}$.

Replacing every unit by its square where necessary in the proof of Lemma \ref{lemma:kill} shows that:
\begin{corollary} \label{corollary:elementarystabiso}
The inclusion $EO_{n-k,n-k} \hookrightarrow ET_{k}$ induces isomorphisms 
\[
H_{t}(EO_{n-k,n-k}) \xrightarrow{\cong} s_{m}^{-1}H_{t}(ET_{k})
\]
for all $t < m/2$.
\end{corollary}

\section{A global action on the $EO_{n,n}$ spectral sequence}
As before, we want to realise these local actions as a global action on the spectral sequence 
\begin{align}
E^{1}_{p,q}(n) = H_{q}(EO_{n,n}, C_{p}(n)) \Rightarrow H_{p+q}(EO_{n,n}, C_{*}(n)).\label{ss:elementarysecondhyperhomology}
\end{align}
Again, we do this by defining an action on the associated exact couple with abutment. Specifically, the spectral sequence 
\[
E^{1}_{p,q} = H_{q}(EO_{n,n}, C_{p}(n)) \Rightarrow H_{p+q}(EO_{n,n}, C_{*}(n)) 
\]
may be obtained from the exact couple with abutment 
\begin{equation}
\begin{tikzcd}
\bigoplus_{p,q} E^{1}_{p,q} \arrow{r}{k} 
  & \bigoplus_{p,q} D^{1}_{p,q} \arrow{d}{i} \arrow{r}{\sigma} 
    & \bigoplus_{p+q} A_{p+q}  \\
    & \bigoplus_{p,q} D^{1}_{p,q} \arrow{lu}{j} \arrow{ru}{\sigma}
\end{tikzcd} \label{exactcouple:elementaryhyperhomologyII}
\end{equation}
with $E^{1}_{p,q} = H_{p+q}(EO_{n,n}, C_{\leq p}(n)/C_{\leq p-1}(n))$; $D^{1}_{p,q} = H_{p+q}(EO_{n,n}, C_{\leq p}(n))$; $A_{p+q} = H_{p+q}(EO_{n,n},C_{*}(n))$; the maps $i, j ,k$ being the maps of the long exact sequence of homology groups associated to the short exact sequence of complexes 
\[
0 \rightarrow C_{\leq p-1}(n) \rightarrow C_{\leq p}(n) \rightarrow C_{\leq p}(n)/C_{\leq p-1}(n) \rightarrow 0,
\]
and $\sigma$ is induced by the inclusion.

For $a \in R^{*}$, we define the global action on Spectral Sequence (\ref{ss:elementarysecondhyperhomology}) to be the action induced by the map $(C_{B_{a^{-4}}},B_{a^{-4}})$ on exact couple (\ref{exactcouple:elementaryhyperhomologyII}), where $C_{B_{a^{-4}}}$ denotes conjugation by the matrix $B_{a^{-4}}$ of Section \ref{subsubsection:globalaction}, and $B_{a^{-4}}$ also refers to multiplication by this matrix. 

\begin{proposition}
Let $k,q \geq 0$ and $n\geq 1$. Then, for all $a \in R^{*}$, the following diagram commutes:
\[
\begin{tikzcd}
H_{q}(EO_{n,n}, C_{k}(n)) \arrow{rr}{(C_{B_{a^{-4}}},B_{a^{-4}})_{*}}
  && H_{q}(EO_{n,n}, C_{k}(n)) \\
H_{q}(EO_{n,n},C_{k}(n)) \arrow{u}{id} \arrow[swap]{rr}{(id, \phi_{a^{2}})_{*}}
  && H_{q}(EO_{n,n},C_{k}(n)) \arrow[swap]{u}{id},
\end{tikzcd}
\]
where for $a \in R^{*}$, the map 
\[
(id, \phi_{a^{2}}): (EO_{n,n}, C_{k}(n)) \rightarrow (EO_{n,n}, C_{k}(n))
\]
is defined to be the identity on $EO_{n,n}$ and on basis elements of $C_{k}(n)$ as 
\[
\phi_{a^{2}} : (v_{1}, \dots , v_{k}) \mapsto (a^{-2}v_{1}, \dots, a^{-2}v_{k}).
\]
\end{proposition}

\begin{proof}
We use Lemma \ref{lemma:homologycommute}. Specifically, consider the diagram 
\[
(EO_{n,n},C_{k}(n)) \stackrel[(f_{2}, \varphi_{2})]{(f_{1}, \varphi_{1})}\rightrightarrows(EO_{n,n},C_{k}(n))
\]
where $(f_{1}, \varphi_{1}) := (id, \phi_{a^{2}})$ and $(f_{2}, \varphi_{2}) := (C_{B_{a^{-4}}}, B_{a^{-4}})$. 
Define
\[
\kappa := D_{a^{2},n} =
 \begin{pmatrix}
a^{2} \\
& a^{-2}\\
&&\ddots \\
&&& a^{2} \\
&&&& a^{-2}
\end{pmatrix}.
\]
Denoting for $a \in R^{*}$, 
\[
\underline{a} :=  \begin{pmatrix}
a \\
& a\\
&&\ddots \\
&&& a \\
&&&& a
\end{pmatrix},
\]
note that $B_{a^{-4}} = \kappa \underline{a}^{-2}$, so that $C_{B_{a^{-4}}} = C_{\kappa}C_{\underline{a}^{-2}}$. But, $C_{\underline{a}^{-2}} = id$, so that $C_{B_{a^{-4}}} = C_{\kappa}$. 
Furthermore, note that for every $(v_{1}, \dots, v_{k}) \in \mathcal{IU}_{k}(R^{2n})$, $B_{a^{-4}}(v_{1}, \dots, v_{k}) = \kappa \phi_{a^{2}} (v_{1}, \dots, v_{k})$, since $B_{a^{-4}} = \kappa \underline{a}^{-2}$. Thus, by Lemma \ref{lemma:homologycommute}, the diagram commutes. 
\end{proof}

\begin{proposition}
Let $q \geq 0$ and $0 \leq k < n$. Then, for all $a \in R^{*}$, the following diagram commutes:
\[
\begin{tikzcd}
H_{q}(EO_{n,n}, C_{k}(n)) \arrow{r}{(id,\phi_{a^{2}})_{*}}
  & H_{q}(EO_{n,n}, C_{k}(n)) \\
H_{q}(ET_{k},\mathbb{Z}) \arrow{u}{(i,(e_{1}, \dots, e_{k}))_{*}}[swap]{\cong} \arrow[swap]{r}{C_{D_{a^{2},k}}}
  & H_{q}(ET_{k},\mathbb{Z}) \arrow[swap]{u}{(i,(e_{1}, \dots, e_{k}))_{*}}[swap]{\cong},
\end{tikzcd}
\]
where the vertical arrows are the isomorphisms given by Shapiro's Lemma and the map $C_{D_{a^{2},k}}$ denotes the map induced by conjugation with the element $D_{a^{2},k}$ on the stabiliser $ET_{k}$.
\end{proposition}
\begin{proof}
We use Lemma \ref{lemma:homologycommute}. Specifically, we have to consider the diagram 
\[
(ET_{k},\mathbb{Z}) \stackrel[(f_{2}, \varphi_{2})]{(f_{1}, \varphi_{1})}\rightrightarrows(EO_{n,n},C_{k}(n))
\]
where $(f_{1}, \varphi_{1}) := (i, (a^{-2}e_{1}, \dots, a^{-2}e_{k}))$ and $(f_{2}, \varphi_{2}) := (iC_{D_{a^{2},k}}, (e_{1}, \dots, e_{k}))$,
and $i:ET_k \to EO_{n,n}$ is the natural inclusion of groups.
Let $\kappa = D_{a^{2},k} \in EO_{n,n}$. 
Then, for every $A \in T_{k}$, 
$$
f_{2}(A) = iC_{D_{a^{2},k}}(A) 
= D_{a^{2},k}A D_{a^{2},k}^{-1} 
= D_{a^{2},k}i(A) D_{a^{2},k}^{-1} 
= \kappa f_{1}(A) \kappa^{-1}
$$
and
\[
(e_{1}, \dots, e_{k}) = D_{a^{2},k}(a^{-2}e_{1}, \dots, a^{-2}e_{k}) = \kappa (a^{-2}e_{1}, \dots, a^{-2}e_{k}).
\]
By Lemma \ref{lemma:homologycommute}, the diagram commutes.  
\end{proof}

Furthermore, we need to compute $H_{q}(EO_{n,n},C_{n}(n))$ and show that, \textit{after localisation}, they vanish for all $q >0$. 
\begin{proposition} \label{prop:$H_{q}(EO_{n,n},C_{n}(n))$}
For every $m\geq 1$ and $q < m/2$, we have 
\[
s_{m}^{-1}H_{q}(EO_{n,n},C_{n}(n)) \cong 
\begin{cases}
\mathbb{Z}[R^{*}/R^{*2} \times \mathbb{Z}_{2}] & q = 0 \\
0 & q > 0.
\end{cases}
\]
\end{proposition}
\begin{proof}
We have isomorphisms 
\begin{align*}
H_{q}(EO_{n,n},\mathbb{Z}[\mathcal{IU}_{n}]) &\cong \Tor_{q}^{EO_{n,n}}(\mathbb{Z}, \mathbb{Z}[\mathcal{IU}_{n}])\\
&\cong \Tor_{q}^{O_{n,n}}(\mathbb{Z}[R^{*}/R^{*2} \times \mathbb{Z}_{2}],\mathbb{Z}[\mathcal{IU}_{n}]) \\
&\cong H_{q}(\mathbb{Z}[R^{*}/R^{*2} \times \mathbb{Z}_{2}] \otimes_{O_{n,n}}^{\mathbb{L}}\mathbb{Z}[\mathcal{IU}_{n}]) \\
&\cong H_{q}(\mathbb{Z}[R^{*}/R^{*2} \times \mathbb{Z}_{2}] \otimes_{O_{n,n}}^{\mathbb{L}}\mathbb{Z}[O_{n,n}/T_{n}]) \\
&\cong H_{q}(\mathbb{Z}[R^{*}/R^{*2} \times \mathbb{Z}_{2}] \otimes_{T_{n}}^{\mathbb{L}}\mathbb{Z}) \\
&\cong H_{q}(T_{n},\mathbb{Z}[R^{*}/R^{*2} \times \mathbb{Z}_{2}])\\
&\cong \mathbb{Z}[R^{*}/R^{*2} \times \mathbb{Z}_{2}] \otimes_{\mathbb{Z}}H_{q}(T_{n}, \mathbb{Z}), 
\end{align*}
where we have used short exact sequence (\ref{ses:EO_{n,n}O_{n,n}}), transitivity of the $O_{n,n}$-action, the identity $\mathbb{Z}[G/N] \cong \mathbb{Z} \otimes^{\mathbb{L}}_{N}\mathbb{Z}[G]$ for $N$ a subgroup of a group $G$ and the Universal Coefficient Theorem. 
We therefore want to show that the action kills the $H_{q}(T_{n})$ terms for all $q > 0$, whilst leaving the $\mathbb{Z}[R^{*}/R^{*2} \times \mathbb{Z}_{2}]$ term invariant. This will follow from the commutativity of the following two diagrams, which we state as lemmas. 
\begin{lemma}
The following diagram commutes:
\[
\begin{tikzcd}
\Tor_{q}^{EO_{n,n}}(\mathbb{Z}, \mathbb{Z}[\mathcal{IU}_{n}]) \arrow{d}[swap]{(i,i,id)_{*}}{\cong} \arrow{rr}{(id,id,\phi_{a^{2}})_{*}}
  && \Tor_{q}^{EO_{n,n}}(\mathbb{Z}, \mathbb{Z}[\mathcal{IU}_{n}]) \arrow{d}{(i,i,id)_{*}}[swap]{\cong} \\
\Tor_{q}^{O_{n,n}}(\mathbb{Z}[R^{*}/R^{*2} \times \mathbb{Z}_{2}], \mathbb{Z}[\mathcal{IU}_{n}]) \arrow{rr}[swap]{(id,id,\phi_{a^{2}})_{*}} 
 &&\Tor_{q}^{O_{n,n}}(\mathbb{Z}[R^{*}/R^{*2} \times \mathbb{Z}_{2}], \mathbb{Z}[\mathcal{IU}_{n}]),
\end{tikzcd}
\]
where the vertical maps are the isomorphisms given by short exact sequence (\ref{ses:EO_{n,n}O_{n,n}}); $i: \mathbb{Z} \hookrightarrow \mathbb{Z}[R^{*}/R^{*2} \times \mathbb{Z}_{2}]$ and $i : EO_{n,n} \hookrightarrow O_{n,n}$ denote the canonical inclusions and recall that $\phi_{a^{2}}: \mathbb{Z}[\mathcal{IU}_{n}] \rightarrow \mathbb{Z}[\mathcal{IU}_{n}]$ is the map defined on basis elements by $(v_{1}, \dots, v_{n}) \mapsto (a^{-2}v_{1}, \dots, a^{-2}v_{n})$. 
\end{lemma}
\begin{proof}
Easily seen by inspection. 
\end{proof}
\begin{lemma}
The following diagram commutes:
\[
\begin{tikzcd}
\Tor_{q}^{O_{n,n}}(\mathbb{Z}[R^{*}/R^{*2} \times \mathbb{Z}_{2}], \mathbb{Z}[\mathcal{IU}_{n}]) \arrow{rr}{(id,id,\phi_{a^{2}})_{*}}
  && \Tor_{q}^{O_{n,n}}(\mathbb{Z}[R^{*}/R^{*2} \times \mathbb{Z}_{2}], \mathbb{Z}[\mathcal{IU}_{n}])  \\
\Tor_{q}^{T_{n}}(\mathbb{Z}[R^{*}/R^{*2} \times \mathbb{Z}_{2}], \mathbb{Z}) \arrow{u}{(id,i,(e_{1}, \dots, e_{n}))_{*}}[swap]{\cong} \arrow{rr}[swap]{(D_{a^{2},n}^{-1},C_{D_{a^{2},n}},id)_{*}} 
 &&\Tor_{q}^{T_{n}}(\mathbb{Z}[R^{*}/R^{*2} \times \mathbb{Z}_{2}], \mathbb{Z}) \arrow{u}[swap]{(id,i,(e_{1}, \dots, e_{n}))_{*}}{\cong},
\end{tikzcd}
\]
where the vertical maps are the isomorphisms given by the transitivity of the $O_{n,n}$-action; $D_{a^{2}, n}^{-1}$ denotes the map induced right multiplication by $D_{a^{2}, n}^{-1} \in O_{n,n}$ and $C_{D_{a^{2},n}}$ denotes the map induced by conjugation with $D_{a^{2}, n}$. 
\end{lemma}
\begin{proof}
We use Lemma \ref{lemma:torcommute}. Specifically, consider the diagram 
\[
(\mathbb{Z}[R^{*}/R^{*2} \times \mathbb{Z}_{2}],T_{n},\mathbb{Z}) \stackrel[(f_{2}, \varphi_{2}, g_{2})]{(f_{1}, \varphi_{1},g_{1})}\rightrightarrows(\mathbb{Z}[R^{*}/R^{*2} \times \mathbb{Z}_{2}],O_{n,n},\mathbb{Z}[\mathcal{IU}_{n}])
\]
where $(f_{1}, \varphi_{1},g_{1}) := (id, i, (a^{-2}e_{1}, \dots, a^{-2}e_{n}))$ and $(f_{2}, \varphi_{2},g_{2}) := (D_{a^{2},n}^{-1}, iC_{D_{a^{2},n}}, (e_{1}, \dots, e_{n}))$. 
Let $\kappa:= D_{a^{2},n} \in O_{n,n}$. We have that $\varphi_{2} = \kappa \varphi_{1} \kappa^{-1}$; $g_{2} = \kappa g_{1}$ and $f_{2} = f_{1}\kappa^{-1}$, so that by Lemma \ref{lemma:torcommute}, the diagram commutes. 
\end{proof}
Note that in the previous lemma, $D_{a^{2},n} \in EO_{n,n}$, so that the action on $\mathbb{Z}[R^{*}/R^{*2} \times \mathbb{Z}_{2}]$ is \textit{trivial}. It follows therefore that the action on $ \mathbb{Z}[R^{*}/R^{*2} \times \mathbb{Z}_{2}] \otimes_{\mathbb{Z}}H_{q}(T_{n}, \mathbb{Z})$ is trivial on $ \mathbb{Z}[R^{*}/R^{*2} \times \mathbb{Z}_{2}]$ and is the action induced by conjugation by $D_{a^{2},n}$ on $H_{q}(T_{n})$. By Lemma \ref{lemma:kill} (using the fact that $T_{n} = L_{n}$), the proposition follows. 
\end{proof}

Thus, we have shown that there exists an $R^{*}$-action on the spectral sequence 
\begin{align*}
E^{1}_{p,q}(n) &= H_{q}(EO_{n,n}, C_{p}(n)) \Rightarrow H_{p+q}(EO_{n,n}, C_{*}(n))
\end{align*}
which induces the desired local actions considered previously. 
Using Corollary  \ref{corollary:elementarystabiso}  and Proposition \ref{prop:$H_{q}(EO_{n,n},C_{n}(n))$}, we obtain the following.
\begin{corollary} \label{corollary:elementarylocalisedE1} 
For every $m \geq 1$ and every $q < m/2$, the localised spectral sequence 
\begin{align}
{}_{m}E^{1}_{p,q}(n) = s_{m}^{-1}E^{1}_{p,q}(n) \Rightarrow s_{m}^{-1}H_{p+q}(EO_{n,n}, C_{*}(n)) \label{ss:elementarylocalised}
\end{align}
has $_{m}E^{1}_{p,q}$ terms 
\[
{}_{m}E^{1}_{p,q}(n) = s_{m}^{-1}H_{q}(EO_{n,n}, C_{p}(n)) \cong 
\begin{cases}
H_{q}(EO_{n-p, n-p}), & 0 \leq p < n \\
\mathbb{Z}[R^{*}/R^{*2} \times \mathbb{Z}_{2}], & p = n, q = 0 \\
0 & p = n, q > 0.  
\end{cases}
\] 
\end{corollary}\qed

Our next task is to compute the localised $d^{1}$ differentials $d^{1}:  \,\,_{m}E^{1}_{p,q} \rightarrow \,\, _{m}E^{1}_{p-1,q}$. 
\section{Computation of the localised $d^{1}$ differentials, and proof of homological stability}
\begin{proposition} \label{prop:elementarylocalisedd1}
For all $q < m/2$ and $0 \leq p < n$, 
the homomorphism $d^{1}_{p,q}$ is 
\[
d^1_{p,q}=
\begin{cases}
0, & p \,\, \text{even} \\
i_{*}, & p \,\, \text{odd},
\end{cases}
\]
where $i : EO_{n-p, n-p} \hookrightarrow EO_{n-p+1, n-p+1}$ denotes the inclusion. 
For $p = n$, the homomorphism $d^{1}_{n,q}$ is 0 if $q > 0$ or if $n$ is even; and for $q = 0$ and $n$ odd, $d^{1}_{n,0}$ is the augmentation map $\varepsilon: \mathbb{Z}[R^{*}/R^{*2} \times \mathbb{Z}_{2}] \rightarrow \mathbb{Z}$.
\end{proposition}

\begin{proof}
For all $p < n$, we want to show that the following diagram commutes:
\begin{equation}
\begin{tikzcd}
H_{q}(EO_{n-p,n-p}) \arrow[swap]{d}{i_{*}} \arrow{rrr}{(\iota,(e_{1}, \dots, e_{p}))_{*}}
  &&& H_{q}(EO_{n,n},C_{p}(n)) \arrow{d}{(d_{i})_{*}}  \\
H_{q}(EO_{n-p+1,n-p+1})  \arrow[swap]{rrr}{(\iota,(e_{1}, \dots, e_{p-1}))_{*}}
  &&& H_{q}(EO_{n,n}, C_{p-1}(n)),
\end{tikzcd} 
\end{equation}
where $\iota : EO_{n-p, n-p} \hookrightarrow EO_{n,n}$ denotes the inclusion map; $(e_{1}, \dots, e_{p}) : 1 \mapsto (e_{1}, \dots, e_{p})$ and recall that $d_{i}(v_{1}, \dots, v_{p}) = (v_{1},\dots, \hat{v}_{i}. \dots, v_{p} )$.
Suppose the matrix $A$ in the proof of Proposition \ref{prop:computelocalisedd1} has spinor norm $\theta(A) = a$. Note that, as $A = 
\begin{pmatrix}
\sigma & \\
& 1_{2(n-p)}
\end{pmatrix}$, it follows that $\theta(A) = \theta(\sigma)$.  If $a = 1$, we are done. Otherwise, define $\hat{A} \in O_{n,n}$ by sending a hyperbolic basis to a hyperbolic basis as follows:
\begin{align*}
(e_{1}, \dots, \hat{e}_{i}, \dots, e_{p}) &\mapsto (e_{1}, \dots, e_{p-1}) \\
(f_{1}, \dots, \hat{f}_{i}, \dots, f_{p}) &\mapsto (f_{1}, \dots, f_{p-1}) \\
e_{i} &\mapsto ae_{p} \\
f_{i} &\mapsto a^{-1}f_{p} \\
e_{j} \mapsto e_{j} \,\, \text{and} \,\, f_{j} \mapsto f_{j} & \,\, \text{for all} \,\, p+1\leq j \leq n. 
\end{align*}
Write $\hat{A} = 
\begin{pmatrix}
\hat{\sigma} & \\
& 1_{2(n-p)}
\end{pmatrix}$,
so that $\theta(\hat{A}) = \theta(\hat{\sigma})$. We prove that $\hat{\sigma} \in EO_{p,p}$. Indeed, this follows from the matrix equation 
\[
\hat{\sigma} = 
\begin{pmatrix}
1 \\
& \ddots \\
&& 1 \\
&&& a \\
&&&& a^{-1}
\end{pmatrix}\sigma.
\]
Clearly, we still have $(e_{1}, \dots, e_{p-1}) = \hat{A} (e_{1}, \dots, \hat{e}_{i}, \dots e_{p})$ and for every $B \in EO_{n-p, n-p}$, 
\[
\iota \circ i (B) = \hat{A} \iota (B) \hat{A}^{-1},
\]
so that by Lemma \ref{lemma:homologycommute}, the diagram commutes.

For $p = n$, it suffices to show that the diagram commutes:
\[
\begin{tikzcd}
\mathbb{Z}[R^{*}/R^{*2} \times \mathbb{Z}_{2}] \arrow[swap]{d}{\varepsilon} \arrow{rr}
  && C_{n}(n)_{EO_{n,n}} \arrow{d}{d_{i}}  \\
\mathbb{Z} \arrow[swap]{rr}{1 \mapsto (e_{1}, \dots, e_{n-1})}
  && C_{n-1}(n)_{EO_{n,n}},
\end{tikzcd}  
\]
where the top horizontal arrow maps a given basis element $x \in R^{*}/R^{*2} \times \mathbb{Z}_{2} $ to the element given by the isomorphism of $O_{n,n}$-sets $R^{*}/R^{*2} \times \mathbb{Z}_{2} \cong O_{n,n}/EO_{n,n} \cong \mathcal{IU}_{n}(R^{2n})/EO_{n,n}$, see Proposition \ref{prop:elementaryorbits}. But this follows from the fact that $EO_{n,n}$ acts transitively on $\mathcal{IU}_{n-1}(R^{2n})$. 
\end{proof}
In the proof of the next proposition, we will use the so called \textit{hyperbolic map}. 

\begin{definition}
Define the hyperbolic map as a group homomorphism $H:GL_{n}(R) \rightarrow O_{n,n}(R)$ given by 
\begin{align*}
H: GL_{n}(R) &\longrightarrow O_{n,n}(R) \\
A &\longmapsto 
\begin{pmatrix}
A & \\
& ^{t}(A^{-1})
\end{pmatrix}.
\end{align*}
\end{definition}
\begin{remark}
In the above definition, we have used the convention that $R^{2n}$ is equipped with symmetric bilinear form given by $\begin{pmatrix}
0 & I_{n} \\
I_{n} & 0
\end{pmatrix}$, and has ordered basis $e_{1}, \dots, e_{n}, f_{1}, \dots, f_{n}$, so that $\langle e_{i}, e_{j} \rangle = \langle f_{i}, f_{j} \rangle = 0$ and $\langle e_{i}, f_{j} \rangle = \delta_{ij}$. We have done this for the sake of notation.  It is clear that this convention differs from our usual convention up to matrix conjugation (by a suitable permutation matrix). We tacitly assume this whenever using the hyperbolic map. 
\end{remark}

We need to prove the following proposition:

\begin{proposition} \label{prop:elementaryzerodifferentials}
The differentials $d^{r}_{p,q}$ in Spectral Sequence (\ref{ss:elementarylocalised}) are zero for $r \geq 2$ and $q < m/2$, $p \leq n$. Hence, for all $q < m/2$ and $p \leq n$, ${}_{m}E^{2}_{p,q} \cong {}_{m}E^{\infty}_{p,q}$.
\end{proposition}
\begin{proof}
For $n = 0,1$, the spectral sequence under consideration is located in columns 0 and 1. Therefore, the differentials $d^{r}$ for $r \geq 2$ are zero by dimension arguments. 

For $n \geq 2$, consider the homomorphism of complexes of $EO_{n-2,n-2}$-modules
\[
\tau : C_{*}(n-2)[-2] \rightarrow C_{*}(n).
\]
as defined in Proposition \ref{prop:zerodifferentials}. Note that the diagram 
\[
\begin{tikzcd}
(EO_{n-2,n-2}, C_{p-2}(n-2)) \arrow{r}{(i,\tau_{j})}
  & (EO_{n,n}, C_{p}(n)) \\
(EO_{n-2,n-2}, C_{p-2}(n-2)) \arrow{u}{(C_{B_{a}}, B_{a})} \arrow[swap]{r}{(i,\tau_{j})}
  & (EO_{n,n}, C_{p}(n)) \arrow[swap]{u}{(C_{B_{a}}, B_{a})},
\end{tikzcd}
\]
still commutes, so that we have an induced map on localised spectral sequences 
\[
{}_{m}\tau_{*} : {}_{m}\tilde{E} \rightarrow {}_{m}E.
\] 
The claim would then follow by induction on $r$ using the following lemma:
\begin{lemma}
The map ${}_{m}\tau_{*} : {}_{m}\tilde{E}^{1}_{p,q} \rightarrow {}_{m}E^{1}_{p,q}$ is the identity for all $q < m/2$ and $2 \leq p \leq n$.
\end{lemma}
\begin{proof}
For $2 \leq p < n$, the same proof as in lemma \ref{lem:E1identity} will work, as long as the matrices $A$ and $B$ of Lemma \ref{lem:E1identity} are in $EO_{n,n}$.
Recall that $A \in O_{n,n}$ was defined by 
\begin{align*}
&e_{1} \mapsto e_{1} \\
&e_{2} \mapsto e_{2} - e_{1} \\
&f_{1} \mapsto f_{1} + f_{2} \\
&f_{2} \mapsto f_{2} \\ 
&e_{j} \mapsto e_{j} \,\, \text{for all} \,\, 3 \leq j \leq n \\
 &f_{j} \mapsto f_{j} \,\, \text{for all} \,\, 3 \leq j \leq n,
\end{align*}
and $B \in O_{n,n}$ was defined by
\begin{align*}
&e_{1} \mapsto e_{2} \\
&e_{2} \mapsto e_{2} - e_{1} \\
&f_{1} \mapsto f_{1} + f_{2} \\
&f_{2} \mapsto -f_{1} \\ 
&e_{j} \mapsto e_{j} \,\, \text{for all} \,\, 3 \leq j \leq n \\
 &f_{j} \mapsto f_{j} \,\, \text{for all} \,\, 3 \leq j \leq n.
\end{align*}
It suffices to prove that 
\[
M:= 
\begin{pmatrix}
1 & 0 & -1 & 0 \\
0 & 1 & 0 & 0 \\
0 & 0 & 1 & 0 \\
0 & 1 & 0 & 1
\end{pmatrix}
\]
and 
\[
N:= 
\begin{pmatrix}
0 & 0 & -1 & 0 \\
0 & 1 & 0 & -1 \\
1 & 0 & 1 & 0 \\
0 & 1 & 0 & 0
\end{pmatrix}
\]
are in $EO_{2,2}$. 
Note that the hyperbolic map $H : GL_{2}(R) \rightarrow O_{2,2}(R)$ is a \textit{group homomorphism}. In addition, note that $SL_{2}(R)$ is \textit{perfect}. For example, this follows from \cite[Theorem 4.3.9.]{hahn1989classical} and \cite[Lemma 3.8.]{schlichting2017euler}. Therefore, we deduce $H(SL_{2}(R)) \subseteq [SO_{2,2}(R), SO_{2,2}(R)] \subseteq EO_{2,2}(R)$. We then note that $M = H\left(
\begin{pmatrix}
1 & -1 \\
0 & 1
\end{pmatrix}\right)$ and $N = H\left(
\begin{pmatrix}
0 & -1 \\
1 & 1
\end{pmatrix}\right)$.

Finally, we need to consider the case $p = n$. It suffices to show that for $j = 0, 1, 2$, the following diagram commutes:
\[
\begin{tikzcd}
\mathbb{Z}[R^{*}/R^{*2} \times \mathbb{Z}_{2}] \arrow[swap]{d}{=} \arrow{r}
  & C_{n-2}(n-2)_{EO_{n-2,n-2}} \arrow{d}{\tau_{j}}  \\
\mathbb{Z}[R^{*}/R^{*2} \times \mathbb{Z}_{2}] \arrow[swap]{r}
  & C_{n}(n)_{EO_{n,n}},
\end{tikzcd}
\]
where the top and bottom horizontal arrows map a given basis element $x \in R^{*}/R^{*2} \times \mathbb{Z}_{2} $ to the element given by the isomorphism of $O_{n-2,n-2}$-sets $R^{*}/R^{*2} \times \mathbb{Z}_{2} \cong O_{n-2,n-2}/EO_{n-2,n-2} \cong \mathcal{IU}_{n-2}(R^{2(n-2)})/EO_{n-2,n-2}$ and isomorphism of $O_{n,n}$-sets $R^{*}/R^{*2} \times \mathbb{Z}_{2} \cong O_{n,n}/EO_{n,n} \cong \mathcal{IU}_{n}(R^{2n})/EO_{n,n}$ respectively, see Proposition \ref{prop:elementaryorbits}.

Under the isomorphism $R^{*}/R^{*2} \times \mathbb{Z}_{2} \cong \mathcal{IU}_{n-2}(R^{2(n-2)})/EO_{n-2,n-2}$, an element $x \in R^{*}/R^{*2} \times \mathbb{Z}_{2}$ is sent to the element $P(e_{1}, \dots, e_{n-2})$ for some $P \in O_{n-2,n-2}$, and under the isomorphism $R^{*}/R^{*2} \times \mathbb{Z}_{2} \cong \mathcal{IU}_{n}(R^{2n})/EO_{n,n}$, the same element $x \in R^{*}/R^{*2} \times \mathbb{Z}_{2}$ is sent to the element $\tilde{P}(e_{1}, \dots, e_{n})$, where $\tilde{P} := 
\begin{pmatrix}
1_{4} &  \\
& P 
\end{pmatrix}$.
Recalling that each $\tau_{j}$ is a map of $O_{n-2,n-2}$-modules, we have that 
\[
\tau_{j}(P(e_{1},\dots, e_{n-2})) = \tilde{P} \tau_{j}(e_{1}, \dots, e_{n-2}) = 
\begin{cases}
\tilde{P}(e_{1}, \dots, e_{n}), & j = 0 \\
\tilde{P}(e_{1}, e_{2}-e_{1}, e_{3}, \dots, e_{n}), & j = 1 \\
\tilde{P}(e_{2}, e_{2}-e_{1}, e_{3}, \dots, e_{n}), & j = 2.
\end{cases}
\]
Thus, for $j = 0$, the diagram commutes by inspection. For $j = 1$, we note that 
\[
A\tilde{P}(e_{1}, e_{2}, e_{3}, \dots, e_{n}) = \tilde{P}A(e_{1}, e_{2}, e_{3}, \dots, e_{n}) = \tilde{P}(e_{1}, e_{2}-e_{1}, e_{3}, \dots, e_{n}),
\]
since 
\[
A\tilde{P} = 
\begin{pmatrix}
M &  \\
& 1 
\end{pmatrix}
\begin{pmatrix}
1_{4} &  \\
& P 
\end{pmatrix} =
\begin{pmatrix}
1_{4} &  \\
& P 
\end{pmatrix}
\begin{pmatrix}
M &  \\
& 1 
\end{pmatrix}
= \tilde{P}A.
\]
Similarly, for $j = 2$, we note that 
\[
B\tilde{P}(e_{1}, e_{2}, e_{3}, \dots, e_{n}) = \tilde{P}B(e_{1}, e_{2}, e_{3}, \dots, e_{n}) = \tilde{P}(e_{2}, e_{2}-e_{1}, e_{3}, \dots, e_{n}),
\]
where $B$ and $\tilde{P}$ commute for similar reasons. 
Thus, the diagrams commute. These diagrams still commute after localisation, but now the horizontal maps become the identification isomorphisms. 
\end{proof}
This proves the lemma, and thus Proposition \ref{prop:elementaryzerodifferentials}.
\end{proof}
\begin{theorem} \label{thm:Ehstability}
Let $R$ be a commutative local ring with infinite field such that $2 \in R^{*}$. Then, the natural homomorphism 
\[
H_{k}(EO_{n,n}(R)) \longrightarrow H_{k}(EO_{n+1,n+1}(R))
\]
is an isomorphism for $k \leq n-1$ and surjective for $k \leq n$. 
\end{theorem}
\begin{remark}
This improves the range for homological stability given by Randal-Williams and Wahl in \cite[Theorem 5.16.]{randal2017homological} by a factor of 3.
\end{remark}
\begin{proof}
Choose $m > 0$ sufficiently large. 
We have a Spectral Sequence (\ref{ss:elementarylocalised}) with $E^{1}$-terms given by Corollary \ref{corollary:elementarylocalisedE1} and $d^{1}_{p,q}$ was computed for all $q < m/2$ in Proposition  \ref{prop:elementarylocalisedd1}. From Theorem \ref{thm:C*acyclic}, Spectral Sequences (\ref{ss:elementaryfirsthyperhomology}) and (\ref{ss:elementarylocalised}) and Proposition \ref{prop:elementaryzerodifferentials}, we deduce ${}_{m}E^{2}_{p,q} = {}_{m}E^{\infty}_{p,q}$ for all $p+q \leq n-1$ and $q < m/2$. The theorem follows. 
\end{proof}
\section{Homological Stability for $\Spin_{n,n}$}
Homological stability for $EO_{n,n}$ immediately gives homological stability for $\Spin_{n,n}$:
\begin{theorem}
Let $R$ be commutative local ring with infinite field such that $2 \in R^{*}$. Then, the natural homomorphism 
\[
H_{k}(\Spin_{n,n}(R)) \longrightarrow H_{k}(\Spin_{n+1,n+1}(R))
\]
is an isomorphism for $k \leq n-1$ and surjective for $k \leq n$. 
\end{theorem}
\begin{remark}
This coincides with the $H_{1}$-stability result for $\Spin_{n,n}$ in \cite[Theorem 9.1.15.]{hahn1989classical} and the $H_{2}$-stability result for $\Spin_{n,n}$ in \cite[Theorem 9.1.17, Theorem 9.1.19 and discussion thereafter]{hahn1989classical}. To our best knowledge, this is the first known homological stability result for $\Spin_{n,n}$ that accounts for all homology groups. 
\end{remark}
\begin{proof}
Immediate from Theorem \ref{thm:Ehstability} and the relative Hochschild-Serre Spectral Sequence 
\[
E^{2}_{p,q} = H_{p}(EO_{n,n}, EO_{n-1,n-1}; H_{q}(\mathbb{Z}_{2})) \Rightarrow H_{p+q}(\Spin_{n,n}, \Spin_{n-1,n-1})
\]
obtained from the short exact sequence $1 \rightarrow \mathbb{Z}_{2}\rightarrow \Spin_{n,n} \rightarrow EO_{n,n} \rightarrow 1$.
\end{proof}

\appendix

\section{Spin Groups and The Spinor Norm over a Local Ring} \label{appendix:spinornorm}
\subsection{Definitions, existence and basic properties}

To begin with, let $M = (M,q)$ be a non-singular quadratic module over a commutative ring $R$, which for the purposes of this paper, is such that $2 \in R^{*}$. We will call an element $x \in M$ \textit{anisotropic} if $q(x) \in R^{*}$. We define $b(x,y) = b_{q}(x,y) := \frac{1}{2}(q(x+y) - q(x) -q(y))$ to be the symmetric bilinear form associated to $q$. We will say that $x , y \in M$ are \textit{orthogonal} if $b(x,y) = 0$.
\begin{definition}
A pair $(A,f)$ consisting of an $R$-algebra $A$ and a homomorphism of $R$-modules $f: M \rightarrow A$ is said to be \textit{compatible} with $M$ if for every $x \in M$,
\[
f(x)^{2} = q(x)1_{A}.
\]
\end{definition}

\begin{definition}
A Clifford algebra of $M$ is a compatible pair $(\Cl(M), i)$ which satisfies the following universal property:

If $(A,f)$ is any pair which is compatible with $M$, then there exists a unique homomorphism of $R$-algebras $g: \Cl(M) \rightarrow A$ such that the diagram 
\[ 
\begin{tikzcd}[scale cd=1.1]
M \arrow{r}{i} \arrow[swap]{dr}{f} & \Cl(M) \arrow[dashed]{d}{g} \\
& A
\end{tikzcd}
\]
commutes.
\end{definition}
We establish that any quadratic module $M$ has a Clifford algebra:
\begin{theorem}
Let $M$ be a quadratic module over $R$. Then, $M$ has a Clifford algebra $(\Cl(M), i)$, which is unique up to unique isomorphism. 
\end{theorem}
\begin{proof}
The uniqueness statement follows from the universal property of the Clifford algebra, so it suffices to prove existence. 
We define 
\begin{align*}
M^{\otimes n} &:= M \otimes_{R} \dots \otimes_{R} M \quad (\text{$n$ times}) \quad \text{for $n>0$}, \\
M^{\otimes 0} &:= R, \\
M^{\otimes n} &:= 0 \quad \text{for $n<0$},
\end{align*}
and define 
\[
T(M) := \bigoplus_{n \in \mathbb{Z}} M^{\otimes n},
\]
the tensor algebra of $M$. Let $i: M \rightarrow T(M)$ denote the inclusion. 

Note that the tensor algebra $T(M)$ is a $\mathbb{Z}$-graded $R$-algebra, with product 
\[
(x_{1}\otimes \dots \otimes x_{m})(x_{m+1}\otimes \dots \otimes x_{n}) := (x_{1}\otimes \dots \otimes x_{m} \otimes x_{m+1}\otimes \dots \otimes x_{n}).
\]
Also note that $T(M)$ has the following universal property: If $A$ is a $R$-algebra and $f: M \rightarrow A$ is a $R$-module homomorphism, then there exists a unique $R$-algebra homomorphism $g : T(M) \rightarrow A$ such that the diagram
\[ 
\begin{tikzcd}[scale cd=1.1]
M \arrow{r}{i} \arrow[swap]{dr}{f} & T(M) \arrow[dashed]{d}{g} \\
& A
\end{tikzcd}
\]
commutes. Of course, $g$ is defined by 
\[
g(x_{1}\otimes \dots \otimes x_{n}) := f(x_{1})\dots f(x_{n}).
\]

Continuing with the construction, define $I(q)$ to be the two-sided ideal of $T(M)$ generated by the set 
\[
\{x \otimes x - q(x) | x \in M \}.
\]
We then define the quotient $R$-algebra 
\[
\Cl(M) := T(M)/I(q),
\]
and define $i: M \rightarrow \Cl(M)$ to be the canonical map. By construction, it is clear that $(\Cl(M), i)$ is a compatible pair, so it remains to check the universal property. 

Let $(A,f)$ be a pair compatible with $M$. By the universal property of $T(M)$, there exists an unique $R$-algebra homomorphism $g: T(M) \rightarrow A$ such that $gi = f$. Furthermore, note that 
\[
g(x \otimes x - q(x)) = g(x)^{2} - q(x) = f(x)^{2} - q(x) = q(x) - q(x) = 0. 
\]
Thus, $g$ factors through the quotient, to give a map $g: \Cl(M) \rightarrow A$. 
\end{proof}

\begin{remark}
If $x, y \in M$ are orthogonal, then in $\Cl(M)$, $xy = -yx$, as $0 = b(x,y) = q(x+y) - q(x) -q(y) = (x+y)^{2} -x^{2} -y^{2} = xy +yx$.
\end{remark}
\begin{remark}
The identity of $\Cl(M)$, denoted $1_{\Cl(M)}$, together with the elements $\{ i(x) | x \in M \}$, generate $\Cl(M)$ as an $R$-algebra. 
\end{remark}
\begin{remark}
The Clifford algebra $\Cl(M)$ is canonically a $\mathbb{Z}_{2}$-graded algebra, with the grading defined as follows:
We define $\Cl(M)_{0}$ be the submodule of $\Cl(M)$ spanned by $1_{\Cl(M)}$ and $\{ i(x_{i_{1}})\dots i(x_{i_{k}}) | k \,\, \text{even}\}$; and we define $\Cl(M)_{1}$ be the submodule of $\Cl(M)$ spanned by $\{ i(x_{i_{1}})\dots i(x_{i_{k}}) | k \,\, \text{odd}\}$. Clearly, $\Cl(M)_{0}$ is a subalgebra of $\Cl(M)$. 
\end{remark}
\begin{remark}
Consider the graded centre $Z_{gr}(\Cl(M))$ of the Clifford algebra $\Cl(M)$. This is defined to be the graded subspace of the Clifford algebra $\Cl(M)$ whose homogeneous elements $h(Z_{gr}(\Cl(M)))$ are determined by 
\[
c \in h(Z_{gr}(\Cl(M))) \Longleftrightarrow cs = -(1)^{\partial s \partial c} sc \quad \forall s \in h(\Cl(M)),
\]
where $h(\Cl(M))$ denotes the homogeneous elements of $\Cl(M)$ and $\partial$ denotes the degree of the homogeneous element. 

When $M$ is free of finite rank, we cite the following important structural result: 
\begin{lemma}\label{lemma:gradedcenter}
Let $M$ be a free non-singular quadratic module of finite rank. Then
\[
Z_{gr}(\Cl(M)) = R.
\]
\end{lemma}
\begin{proof}
See \cite[Theorem 7.1.11.]{hahn1989classical}.
\end{proof}
\end{remark}
\begin{remark}
By the universal property of the Clifford algebra, every $\sigma \in O(M)$ in the orthogonal group of $M$ uniquely determines an automorphism of $R$-algebras $\Cl(\sigma): \Cl(M) \rightarrow \Cl(M)$. This association gives rise to a group homomorphism 
\[
\Cl : O(M) \rightarrow \text{Aut}(\Cl(M)).
\]
In particular, taking $\sigma := -1_{M}$ provides a unique automorphism $\Cl(-1_{M}): \Cl(M) \rightarrow \Cl(M)$ such that $\Cl(-1_{M})(i(x)) = -i(x)$ for all $x \in M$. Observe that $\Cl(-1_{M})|_{\Cl(M)_{0}} = 1_{\Cl(M)_{0}}$ and $\Cl(-1_{M})|_{\Cl(M)_{1}} = -1_{\Cl(M)_{1}}$.

The map $\Cl(-1_{M})$ is used to define the so called `canonical involution'  $\overline{\phantom{x}}$ on $\Cl(M)$. 

But first, let $\Cl(M)^{op}$ denote the opposite algebra of $\Cl(M)$. By the universal property of the Clifford algebra, there exists an unique algebra homomorphism $\sim: \Cl(M) \rightarrow \Cl(M)^{op}$ such that the diagram 
\[ 
\begin{tikzcd}[scale cd=1.1]
M \arrow{r}{i} \arrow[swap]{dr}{i} & \Cl(M) \arrow[dashed]{d}{\sim} \\
& \Cl(M)^{op}
\end{tikzcd}
\]
commutes. We will consider $\sim$ as a map $\sim: \Cl(M) \rightarrow \Cl(M)$. Note that $\widetilde{cd} = \tilde{d}\tilde{c}$ for every $c, d \in \Cl(M)$, and $\widetilde{i(x)} = i(x)$ for every $x \in M$, so that $\tilde{\tilde{c}} = c$ for every $c \in \Cl(M)$ and $\sim$ is therefore an involution on $\Cl(M)$. 

We then define the \textit{canonical involution} $\overline{\phantom{x}}: \Cl(M) \rightarrow \Cl(M)$ to be the composite 
\[
\Cl(M) \xrightarrow{\Cl(-1)} \Cl(M) \xrightarrow{\sim} \Cl(M).
\] 
One easily checks that this does indeed define an involution on $\Cl(M)$. Observe that $\overline{\phantom{x}}$ is the unique $R$-linear anti-automorphism of $\Cl(M)$ which satisfies $\overline{i(x)} = -i(x)$ for every $x \in M$. We will use the canonical involution in our definition of the $\Spin$ group. 
\end{remark}
\subsection{The groups $\Gamma(M), S\Gamma(M)$, $\Spin(M)$ and the Spinor Norm}
We define the groups $\Gamma(M), S\Gamma(M)$ and $\Spin(M)$. We also define the Spinor norm and study some of its basic properties, as needed in this paper. Unless stated otherwise, our exposition will closely follow \cite[Chapter 7]{hahn1989classical}. 
\subsubsection{The Groups $\Gamma(M), S\Gamma(M)$, $\Spin(M)$}
\begin{definition}
We define the \textit{Clifford group} $\Gamma(M)$ to be the group
\[
\Gamma(M) := \{c \in \Cl(M)^{*} | cMc^{-1} = M \}. 
\]
\end{definition}
Note that for every $c \in \Gamma(M)$, we canonically obtain a map 
\begin{align*}
\pi c&: M \rightarrow M \\
(\pi c)(x) &:= cxc^{-1}.
\end{align*}
Furthermore, note that $\pi c$ preserves the quadratic form $q$ as $q(\pi c(x)) = q(cxc^{-1}) = cxc^{-1}\otimes cxc^{-1} = cx^{2}c^{-1} = q(x)$. Thus, the assignment $c \mapsto \pi c$ defines a group homomorphism 
\[
\pi : \Gamma(M) \rightarrow O(M).
\]
\begin{definition}
We define the \textit{Special Clifford group} $S\Gamma(M)$ to be the group
\[
S\Gamma(M) := \{c \in \Cl(M)_{0}^{*} | cMc^{-1} = M \}. 
\]
\end{definition}
Note that $S\Gamma(M) = \Gamma(M) \cap \Cl(M)_{0}^{*}$.
\begin{definition}
We define the \textit{Spin group} $\Spin(M)$ to be the group 
\[
\Spin(M) := \{ c \in S\Gamma(M) | c\overline{c} = 1\}.
\]
\end{definition}
Thus, by construction, we have a chain of inclusions $\Spin(M) \subseteq S\Gamma(M) \subseteq \Gamma(M)$.

Later, it will be important for us to understand $\ker(\pi|_{S\Gamma(M)})$ and $\ker(\pi|_{\Spin(M)})$.
\begin{proposition} \label{prop: kerSGamma}
$\ker(\pi: S\Gamma(M) \rightarrow O(M)) = R^{*}$.
\end{proposition}
\begin{proof}
Let $c \in \ker(\pi|_{S\Gamma(M)})$. Then, $c \in \Cl(M)_{0}^{*}$ and $cxc^{-1} = x$ for every $x \in M$. Therefore, as $M$ generates $\Cl(M)$ as an $R$-algebra, we use Lemma \ref{lemma:gradedcenter} to conclude that $c \in Z_{gr}(\Cl(M)) = R$. Similarly, $c^{-1} \in R$, so that $\ker(\pi|_{S\Gamma(M)}) \subseteq R^{*}$. The other inclusion is trivial. 
\end{proof}

\begin{corollary} \label{corollary:kerpi}
$\ker(\pi: \Spin(M) \rightarrow O(M)) \cong \mathbb{Z}_{2}$.
\end{corollary}
\begin{proof}
From Proposition \ref{prop: kerSGamma}, it is clear that 
\[
\ker(\pi: \Spin(M) \rightarrow O(M)) = \{ r \in R^{*} | r^{2} = 1 \}.
\]
Passing to the residue field, we deduce that the square roots of 1 are of the form $r = \varepsilon \pm 1$ for some $\varepsilon$ in the maximal ideal. Using the equation $r^{2} = 1$, we obtain equation $\varepsilon(\varepsilon \pm 2) = 0$. As 2 is a unit, we deduce $\varepsilon \pm 2$ is a unit, so that $\varepsilon = 0$.
\end{proof}
\begin{definition}
We define the \textit{spinorial kernel}
\[
O'(M)
\]
to be the image of the homomorphism $\pi: \Spin(M) \rightarrow O(M)$.
\end{definition}
When $R$ is a local ring with $2 \in R^{*}$, we show that $O'(M)$ is precisely the kernel of the \textit{spinor map} $\theta: SO(M) \rightarrow R^{*}/R^{*2}$, see Definition \ref{def:spinornorm} and Proposition \ref{prop:spinorialkernel}.

In addition, we cite the following theorem, which says that when $R^{2n}$ is a free hyperbolic module over a (semi-)local ring $R$, the spinorial kernel is \textit{precisely} the elementary orthogonal group $EO_{n,n}(R)$ when $n \geq 2$.

\begin{theorem} \label{theorem:imageofspin}
Let $R$ be a commutative semi-local ring. Let $R^{2n}$ be the free hyperbolic module. Denote $O_{n,n}'(R) := O'(R^{2n})$. 
Then, for every $n \geq 2$, $O_{n,n}'(R) = EO_{n,n}(R)$. 
\end{theorem}
\begin{proof}
See \cite[Theorem 9.2.8.]{hahn1989classical}.
\end{proof}
Thus, when $R$ is a (semi-)local ring with $2 \in R^{*}$ and $n \geq 2$, we have the short exact sequences
\[
1 \rightarrow \mathbb{Z}_{2} \rightarrow \Spin_{n,n}(R) \xrightarrow{\pi} EO_{n,n}(R) \rightarrow 1.
\]

\textit{From now on, we will assume that $R$ is a local ring with $2 \in R^{*}$, and all modules over $R$ are finitely generated projective, so that they are free of finite rank.}

\subsubsection{The Spinor Norm}

In order to define the spinor norm, we first need to define an important class of isometries. 

\begin{definition}
Let $x \in M$ be anisotropic and define $N := \langle x \rangle^{\perp}$. Then, the linear map 
\begin{align*}
\tau_{x}&: M \rightarrow M \\
y &\mapsto y - 2\frac{b(x,y)}{b(x,x)}x
\end{align*}
is called a \textit{reflection} in hyperplane $N$ orthogonal to $x$. 
\end{definition}
This name is suggested by the following lemma:
\begin{lemma}
\begin{enumerate}
\item $\tau_{x}(x) = -x$, $\tau_{x}|_{N} = 1_{N}$.
\item $\tau_{x}$ is an isometry of $(M,b)$.
\item $\tau_{x} \circ \tau_{x} = 1_{M}$.
\item $\det \tau_{x} = -1$.
\end{enumerate}
\end{lemma}

\begin{proof}
The first three statements follow from direct computations. For the last statement, note that $M = \langle x \rangle \oplus N$. Therefore, by Witt's Cancellation Theorem \cite[Chapter I, Theorem 4.4]{milnor1973symmetric}, we may choose a basis $e_{2}, \dots, e_{n}$ of $N$ and complete it to a basis of $M$ by setting $e_{1} = x$. The matrix of $\tau_{x}$ with respect to this basis shows that $\det \tau_{x} = -1$.
\end{proof}

\begin{proposition} \label{prop:reflectionconjugation}
For every $x \in M$ anisotropic, we have $\pi(x) = -\tau_{x}$. 
\end{proposition}
\begin{proof}
For every $y \in M$, we have 
\begin{align*}
\tau_{x}(y) &= y -2\frac{b(x,y)}{b(x,x)}x \\
&= y - \frac{q(x+y) - q(x) -q(y)}{q(x)}x \\
&= y - (xy +yx)x^{-1} \\
&= -xyx^{-1} \\
&= -\pi(x)(y).
\end{align*}
\end{proof}

The next proposition will be used to show that our definition of the spinor norm is well-defined. 
\begin{proposition} \label{prop:spinornormwd}
Let $u_{1},\dots, u_{r}$ be anisotropic elements in M. If the product $\tau_{u_{1}}\tau_{u_{2}}\cdots\tau_{u_{r}}$ is the identity in $O(M)$, then the product $q(u_{1}) \cdots q(u_{r})$ belongs to $R^{*2}$. 
\end{proposition}
\begin{proof}
We proceed in a similar way to \cite[Proposition 1.12.V]{lam2005introduction}.
By Proposition \ref{prop:reflectionconjugation}, $\pi(u)|_{M} = -\tau_{u}$, so that $1_{M} = (-1)^{r}\pi(u_{1} \cdots u_{r})|_{M}$. But, 
\[
(-1)^{r} = \det(\tau_{u_{1}}\tau_{u_{2}}\cdots\tau_{u_{r}}) \,\, = \det(1_{M}) = 1. 
\]
Thus, we deduce that $r$ must be even. Therefore, we have that 
\[
c := u_{1}\cdots u_{r} \in \Cl(M)_{0} \cap Z(\Cl(M)) \subset Z_{gr}(\Cl(M)) = R.
\]
Similarly, we have that $c^{-1} \in R$, so that $c \in R^{*}$. 
We conclude that 
\[
R^{*2} \ni c^{2} = c\bar{c} = u_{1} \cdots u_{r} u_{r} \cdots u_{1} = q(u_{1}) \cdots q(u_{r}). 
\]
\end{proof}

Consider any isometry $\sigma \in O(M)$, where the rank of $M$ is at least 2. By the Cartan-Dieudonn\'e Theorem for local rings, see for example \cite[Theorem 2]{klingenberg1961orthogonale}, there exists a factorisation $\sigma = \tau_{u_{1}}\tau_{u_{2}}\cdots\tau_{u_{r}}$, where the $u_{i}$ are anisotropic vectors. We define 
\[
\theta(\sigma) := q(u_{1}) \cdots q(u_{r}) \in R^{*}/R^{*2}.
\]
By Proposition \ref{prop:spinornormwd}, $\theta(\sigma)$ does not depend on the choice of factorisation chosen to represent $\sigma$. 

\begin{definition} \label{def:spinornorm}
The map $\theta: O(M) \rightarrow R^{*}/R^{*2}$ is called the \textit{spinor norm}. 
\end{definition} 

The spinor norm is the unique group homomorphism satisfying the property $\theta(\tau_{u}) = q(u) R^{*2}$ for every anisotropic element $u \in M$.

For $R$ a local ring with $2 \in R^{*}$, we want to establish the existence of short exact sequences
\begin{align*}
&1 \rightarrow EO_{n,n}(R) \rightarrow SO_{n,n}(R) \xrightarrow{\theta} R^{*}/R^{*2} \rightarrow 1 \\
&1 \rightarrow EO_{n,n}(R) \rightarrow O_{n,n}(R) \xrightarrow{\theta \times \det} R^{*}/R^{*2} \times \mathbb{Z}_{2} \rightarrow 1. 
\end{align*}

We begin with the following proposition, which is useful when computing with the Spinor norm. 

\begin{proposition} \label{prop: spinorblocksum}
Let $(M,q_{M})$ and $(N,q_{N})$ be free non-singular quadratic modules of finite rank over $R$. Let $A \in O(M)$ and let $B \in O(N)$, considered as matrices. Let $A\oplus B \in O(M \perp N)$ denote the block sum of matrices $A \oplus B = 
\begin{pmatrix}
A & \\
& B
\end{pmatrix}$. 
Then, $\theta(A \oplus B) = \theta(A)\theta(B)$.  
\end{proposition}
\begin{proof}
Suppose $A \in O(M)$ is represented by $A = \tau_{v_{1}}\cdots \tau_{v_{k}}$ and suppose $B \in O(N)$ is represented by $B = \tau_{w_{1}}\cdots \tau_{w_{l}}$. Then, $A\oplus B \in O(M\perp N)$ is represented by $\tau_{\bar{v}_{1}}\cdots \tau_{\bar{v}_{k}}\tau_{\bar{w}_{1}}\cdots \tau_{\bar{w}_{l}}$, where $\bar{v_{i}}, \bar{w_{j}} \in M \perp N$ are the images of of the vectors $v_{i}$ and $w_{j}$ under the canonical embeddings $M \hookrightarrow M\perp N$ and $N \hookrightarrow M\perp N$ respectively. 

Therefore, 
\begin{align*}
\theta(A\oplus B) &= q_{M\perp N}(\bar{v}_{1})\cdots q_{M\perp N}(\bar{v}_{k})q_{M\perp N}(\bar{w}_{1})\cdots q_{M\perp N}(\bar{w}_{l}) \\
&= q_{M}(v_{1})\cdots q_{M}(v_{1})q_{N}(w_{1})\cdots q_{N}(w_{l}) \\
&= \theta(A)\theta(B).
\end{align*}
\end{proof}
The above proposition is used to prove that the spinor norm $\theta: SO_{n,n}(R) \rightarrow R^{*}/R^{*2}$ is \textit{surjective}. 

\begin{proposition} \label{prop:spinornormsurj}
The spinor norm $\theta: SO_{n,n}(R) \rightarrow R^{*}/R^{*2}$ is surjective. 
\end{proposition}

\begin{proof}
Let $r \in R^{*}$, and consider the matrix $\sigma = 
\begin{pmatrix}
r && \\
& r^{-1} &\\
&& 1
\end{pmatrix}$ . 
Note that $\sigma \in SO_{n,n}(R)$. By Proposition \ref{prop: spinorblocksum}, we have that $\theta(\sigma) = \theta\left(
\begin{pmatrix}
r & 0 \\
0 & r^{-1}
\end{pmatrix}\right)$. As $ 
\begin{pmatrix}
r & 0 \\
0 & r^{-1}
\end{pmatrix} = 
\begin{pmatrix}
0 & r \\
r^{-1} & 0
\end{pmatrix}
\begin{pmatrix}
0 & 1 \\
1 & 0
\end{pmatrix}
$ is a product of reflections defined by vectors $(r, -1)$ and $(1,-1)$, we compute that 
\[
\theta\left(
\begin{pmatrix}
r & 0 \\
0 & r^{-1}
\end{pmatrix}\right) = q(r,-1)q(1,-1) = 4r = r \pmod{R^{*2}}.
\]
\end{proof}

Finally, we want to show that the spinorial kernel $O'(M)$ is precisely the kernel of the spinor map $\theta: SO(M) \rightarrow R^{*}/R^{*2}$. This is done in by the following two propositions.

\begin{proposition} \label{prop: ImaSGamma}
$\Ima(\pi: S\Gamma(M) \rightarrow O(M)) = SO(M)$.
\end{proposition}
\begin{proof}
Firstly, note that $SO(M) \subseteq \Ima(\pi)$. Indeed, if $\tau_{v}\tau_{w}$ are a product of any two reflections, then $\pi(vw) = \tau_{v}\tau_{w}$. 

Suppose that  $SO(M) \subsetneq \Ima(\pi)$. Then, there exists a $\sigma \in O(M) \setminus SO(M)$ such that $\sigma \in \Ima(\pi)$. As $\sigma \in O(M) \setminus SO(M)$, we have that $\sigma = \tau_{v_{1}}\cdots \tau_{v_{k}}$ for $v_{i}$ anisotropic and $k$ odd. Furthermore, as $\sigma \in \Ima(\pi)$, there exists $c \in S\Gamma(M)$ such that $\pi(c) = \sigma$. Note, $\pi(v_{1}\cdots v_{k}) = -\tau_{v_{1}} \cdots \tau_{v_{k}} = -\sigma$. Defining $d := v_{1} \cdots v_{k}$, we deduce $\pi(cd^{-1}) = -1_{M}$. This means that $cd^{-1}x(cd^{-1})^{-1} = -x$ for every $x \in M$. 

Note that $c \in \Cl(M)_{0}$ and $d^{-1} \in \Cl(M)_{1}$. Therefore, $cd^{-1} \in \Cl(M)_{1}$. As $cd^{-1}x(cd^{-1})^{-1} = -x$ for every $x \in M$ and $M$ generates $\Cl(M)$ as an $R$-algebra, we use Lemma \ref{lemma:gradedcenter} to conclude that $cd^{-1} \in Z_{gr}(\Cl(M)) = R$. Thus, $c = dr$ for some $r \in R$, and it therefore follows that $c \in \Cl(M)_{1}$. Thus, $c \in \Cl(M)_{0} \cap \Cl(M)_{1} = 0$, which is a contradiction as $c$ is invertible. 
\end{proof}

\begin{proposition} \label{prop:spinorialkernel}
We have $O'(M) = \ker(\theta: SO(M) \rightarrow R^{*}/R^{*2})$. 
\end{proposition}
\begin{proof}
Let $\sigma \in \ker(\theta|_{SO(M)})$. We want to show $\sigma \in O'_{n,n}(M)$. Suppose $\sigma = \tau_{v_{1}}\cdots \tau_{v_{k}}$. Note that $k$ is even and each $v_{i}$ is anisotropic. 

By definition, $1 = \theta(\sigma) = q(v_{1})\cdots q(v_{k})$. Therefore, $r := q(v_{1})\cdots q(v_{k}) \in R^{*2}$. Suppose that $r = s^{2}$. As $\tau_{v_{1}} = \tau_{s^{-1}v_{1}}$, we may replace $v_{1}$ with $s^{-1}v_{1}$ to obtain $\sigma = \tau_{v_{1}}\cdots \tau_{v_{k}}$ such that $q(v_{1})\cdots q(v_{k}) = 1$. Therefore, in $\Cl(M)$, $v_{1}\cdots v_{k}\overline{v_{1}\cdots v_{k}} = 1$. Define $c := v_{1} \dots v_{k}$. As all $v_{i} \in \Cl(M)^{*}$ and $k$ is even, we deduce $c \in \Spin(M)$. By construction, $\pi(c) = \tau_{v_{1}}\cdots \tau_{v_{k}} = \sigma$, so that $\sigma \in O'(M)$. 

Now let $c \in \Spin(M)$ and consider $\pi(c) \in O'(M)$. We want to show $\pi(c) \in \ker(\theta|_{SO(M)})$. By Proposition  \ref{prop: ImaSGamma}, $\pi(c) \in SO(M)$. Therefore, $\pi(c) = \tau_{v_{1}}\cdots \tau_{v_{k}}$ for $v_{i}$ anisotropic and $k$ even. As $k$ is even, we deduce $c^{-1}v_{1} \cdots v_{k} \in S\Gamma(M)$. Furthermore, by definition, $\pi(c^{-1}v_{1}\cdots v_{k}) = 1$. Therefore, by Proposition \ref{prop: kerSGamma} $c^{-1}v_{1}\dots v_{k} \in \ker(\pi|_{S\Gamma}) = R^{*}$. Thus, $c = rv_{1}\cdots v_{k}$ for some $r \in R^{*}$. As $c \in \Spin(M)$, we obtain $1 = c\overline{c} = r^{2}q(v_{1})\cdots q(v_{k})$, so that $q(v_{1}) \cdots q(v_{k}) \in R^{*2}$. Thus, by definition, $\theta(\pi(c)) = 1$. 
\end{proof}

Thus, for $R$ a local ring with $2 \in R^{*}$, we have established the following theorem:
\begin{theorem} \label{theorem:shortexactsequences}
Let $R$ be a commutative local ring with $2 \in R^{*}$, and let $n \geq 2$. Then, we have short exact sequences
\begin{align*}
&1 \rightarrow \mathbb{Z}_{2} \rightarrow \Spin_{n,n}(R) \xrightarrow{\pi} EO_{n,n}(R) \rightarrow 1 \\
&1 \rightarrow EO_{n,n}(R) \rightarrow SO_{n,n}(R) \xrightarrow{\theta} R^{*}/R^{*2} \rightarrow 1 \\
&1 \rightarrow EO_{n,n}(R) \rightarrow O_{n,n}(R) \xrightarrow{\theta \times \det} R^{*}/R^{*2} \times \mathbb{Z}_{2} \rightarrow 1. 
\end{align*}
\end{theorem}
\begin{proof}
Combine Corollary \ref{corollary:kerpi}; Theorem \ref{theorem:imageofspin}; Proposition \ref{prop:spinorialkernel} and Proposition \ref{prop:spinornormsurj}.
\end{proof}
These short exact sequences are used to prove homological stability for $EO_{n,n}(R)$ and $\Spin_{n,n}(R)$, when $R$ is a local ring with infinite residue field such that $2 \in R^{*}$.

\bibliographystyle{alpha}
\bibliography{Homological_stability}

\end{document}